\documentclass[11pt,leqno]{amsart}

\usepackage[all,cmtip]{xy}
\usepackage{a4,latexsym}
\usepackage[utf8]{inputenc}
\usepackage{textcomp,fontenc,exscale,ifthen}
\usepackage[hidelinks]{hyperref}
\usepackage{amsxtra,amsmath,amsfonts,amscd, amssymb, mathrsfs, amsthm}
\usepackage{color, mathtools}
\usepackage{enumitem}

\newtheorem{thm}{Theorem}[subsection]
\newtheorem{prop}[thm]{Proposition}
\newtheorem{cor}[thm]{Corollary}
\newtheorem{lem}[thm]{Lemma}
\newtheorem{lemma}[thm]{Lemma}

\newtheorem{theointro}{Theorem}

\theoremstyle{definition}
\newtheorem{definition}[thm]{Definition}
\newtheorem{ex}[thm]{Example}
\newtheorem{rem}[thm]{Remark}
\newtheorem{art}[thm]{}

\numberwithin{paragraph}{section}
\numberwithin{equation}{section}

\setlist[enumerate]{label=\it{(\roman*)},
ref=\it{(\roman*)}}

\usepackage{color}
\usepackage{comment}

\def\P{{\mathbb P}}
\def\N{{\mathbb N}}
\def\Z{{\mathbb Z}}
\def\Q{{\mathbb Q}}
\def\R{{\mathbb R}}
\def\C{{\mathbb C}}

\def\A{{\mathbb A}}

\def\H{{\mathbb H}}
\def\G{{\mathbb G}}

\def\N{{\mathbb N}}
\def\T{{\mathbb T}}
\def\SS{{\mathbb S}}

\newcommand{\metr}{{\|\hspace{1ex}\|}}
\newcommand{\Hom}{{\rm Hom}}

\def\an{{\rm an}}
\def\div{{\rm div}}

\DeclareMathOperator{\trop}{trop}
\DeclareMathOperator{\tropbar}{\overline{trop}}
\DeclareMathOperator{\MA}{MA}

\newcommand{\Xan}{{X^{\rm an}}}
\newcommand{\Yan}{{Y^{\rm an}}}

\newcommand{\Acal}{{\mathscr A}}
\newcommand{\Bcal}{{\mathscr B}}

\newcommand{\Ecal}{{\mathscr E}}
\newcommand{\Fcal}{{\mathscr F}}

\newcommand{\Hcal}{{\mathscr H}}

\newcommand{\Lcal}{{\mathscr L}}
\newcommand{\Mcal}{{\mathscr M}}

\newcommand{\Ocal}{{\mathscr O}}

\newcommand{\Ucal}{{\mathscr U}}

\newcommand{\Wcal}{{\mathscr W}}
\newcommand{\Xcal}{{\mathscr X}}

\newcommand{\Ycal}{{\mathscr Y}}

\newcommand{\codim}{{\rm codim}}
\newcommand{\coker}{{\rm coker}}

\newcommand{\Pic}{{\rm Pic}}

\newcommand{\Spec}{{\rm Spec}}

\newcommand{\id}{{\rm id}}

\newcommand{\cl}{{\rm cl}}

\DeclareMathOperator{\supp}{supp}

\newcommand{\relint}{{\rm relint}}
\newcommand{\kcirc}{{ K^\circ}}
\newcommand{\ktilde}{{ \tilde{K}}}

\newcommand{\Lan}{{L^{\rm an}}}

\newcommand{\AS}{{A}}

\newcommand{\Rinf}{\R_{-\infty}}
\newcommand{\Rsup}{\R_\infty}
\newcommand{\torus}{{\T}}
\newcommand{\Xsigmaan}{{X_\Sigma^{\rm an}}}
\newcommand{\torusan}{{\torus^{\rm an}}}

\newcommand{\dell}{\partial}
\newcommand{\dellbar}{\overline{\partial}}
\newcommand{\vol}{{\rm  vol}}

\DeclareMathOperator{\im}{Im}
\DeclareMathOperator{\PSH}{PSH}

\title[Tropical toric pluripotential theory]{Pluripotential theory for tropical toric varieties and  non-archimedean Monge--Amp\`ere equations}

\author[J.I.Burgos Gil]{Jos\'e Ignacio Burgos Gil}
\address{
J. I. Burgos Gil,
Instituto de Ciencias Matem\'aticas (CSIC-UAM-UCM-UCM3), 
Calle Ni\-co\-l\'as Cabrera 15, Campus de la Universidad 
Aut\'onoma de Madrid, Cantoblanco, 28049 Madrid, Spain}
\email{burgos@icmat.es}

\author[W.~Gubler]{Walter Gubler}
\address{W. Gubler, Mathematik, Universit{\"a}t 
Regensburg, 93040 Regensburg, Germany}
\email{walter.gubler@mathematik.uni-regensburg.de}

 \author[P.~Jell]{Philipp Jell}
 \address{P. Jell,  Mathematik, Universit{\"a}t 
Regensburg, 93040 Regensburg, Germany}
\email{philipp.jell@mathematik.uni-regensburg.de}

\author[K.~K\"unnemann]{Klaus K{\"u}nnemann}
\address{K. K{\"u}nnemann, Mathematik, Universit{\"a}t 
Regensburg, 93040 Regensburg, Germany}
\email{klaus.kuennemann@mathematik.uni-regensburg.de}

 \thanks{J.~I.~Burgos was partially supported by MINISTERIO
DE CIENCIA E INNOVACION research projects PID2019-108936GB-C21 and 
ICMAT Severo Ochoa project CEX2019-000904-S.
W.~Gubler, P.~Jell and K.~K{\"u}nnemann 
were supported by the collaborative research 
center SFB 1085 \emph{Higher Invariants - Interactions between Arithmetic Geometry and Global Analysis} funded by the Deutsche Forschungsgemeinschaft.}

\begin{document}

\begin{abstract}
Tropical toric varieties are partial compactifications of finite dimensional real vector spaces associated with rational polyhedral fans. We introduce  plurisubharmonic functions and a Bedford--Taylor product for Lagerberg currents on open subsets of a tropical toric variety. The resulting tropical toric pluripotential theory provides  the link to give a canonical correspondence between complex and non-archimedean pluripotential theories of invariant plurisubharmonic functions on toric varieties. We will apply this correspondence to solve invariant  non-archimedean Monge--Ampère equations on toric and abelian varieties over arbitrary non-archimedean fields.  
\end{abstract}


\maketitle

\setcounter{tocdepth}{1}

\tableofcontents

\section{Introduction}  \label{section: introduction}

Pluripotential theory  is a non-linear complex counterpart of classical potential theory. It is the study of plurisubharmonic (short:~psh) functions and the complex Monge--Amp\`ere equation. 
There are rich applications to multidimensional complex analysis, K\"ahler geometry, algebraic geometry and to dynamics. 
A celebrated result was Yau's solution \cite{yau78} of the  complex Monge--Amp\`ere equation for smooth forms on a K\"ahler manifold more than twenty years after Calabi proved uniqueness. 
Bedford and Taylor \cite{bedford-taylor1982} showed that one can go beyond the smooth case as there is a well-defined product for the positive $(1,1)$-currents associated to locally bounded psh functions and hence the  Monge--Amp\`ere measure makes sense for such functions. 
Later, Demailly \cite{demailly_1993} 
defined also a Bedford--Taylor product for singular psh functions. 
Here, singular means that the functions $\varphi$ are allowed to take the value $-\infty$ and the product is only well-defined under certain restrictions on the unbounded loci $\varphi^{-1}(\{-\infty\})$. 

In non-archimedean geometry, continuous semipositive metrics were introduced by Zhang \cite{zhang-95}.  Chambert-Loir \cite{chambert-loir-2006} defined corresponding Monge--Amp\`ere measures. 
Kontsevich and Soibelman \cite{kontsevich-soibelman2006} emphasized the role of non-archimedean geometry for mirror symmetry, while Kontsevich and Tschinkel \cite{kontsevitch-tschinkel} proposed a strategy to solve the non-archimedean Monge--Amp\`ere equation. 
Yuan and Zhang \cite{yuan-zhang} showed uniqueness {of the solution up to scaling}. 
Liu \cite{liu2011} solved the non-archimedean Monge--Amp\`ere equation for totally degenerate abelian varieties. 
In case of a trivially or discretely valued complete field $K$ of residue characteristic zero, Boucksom, Favre and Jonsson \cite{bfj-solution} gave a global approach to non-archimedean pluripotential theory and presented a variational method which gives the existence of a solution for any smooth projective variety over $K$.  

The goal of this paper is to introduce a pluripotential theory for certain partial compactifications of $\R^n$ called tropical toric varieties. 
Noting that convex functions are the natural real analogue of psh functions, it is clear that the Monge--Amp\`ere measure on such a compactification is closely related to the classical real Monge--Amp\`ere operator studied by Aleksandrov. 
The classical real Monge--Amp\`ere equation has several variants. 
The corresponding Dirichlet problem was solved by Aleksandrov
\cite{aleksandrov_1958} and Bakelman \cite{bakelman57}, the
corresponding second boundary problem (i.e.~with prescribed gradient
image) was done by Pogorelov \cite{pogorelov1964}.

Toric geometry is very useful for testing conjectures in algebraic and
arithmetic geometry. 
The reason is that invariant objects in toric geometry are completely described by combinatorial means which comes handy for such tests. 
In particular, a toric variety with dense {split} torus $\T$ is completely determined by a fan $\Sigma$ in $N_\R= N \otimes_\Z \R$ where $N$ is the cocharacter lattice of $\T$ and so we denote the associated toric variety by $X_\Sigma$. 
Note that $X_\Sigma$ is {naturally} defined over $\Spec\, \Z$ and by base change, we get a toric variety over any base field. 
In the introduction, we consider the associated complex toric variety
$X_{\Sigma,\infty}$ and the associated toric variety $X_{\Sigma,v}$
for a non-archimedean field $K$ with valuation $v$ of rank one. 

There is a tropical toric variety $N_\Sigma$ associated to the fan $\Sigma$ which is a partial compactification of $N_\R$. 
For $w \in \{\infty,v\}$, it comes with a canonical tropicalization map
\[
\trop_w\colon X_{\Sigma,w}^{\rm an} \longrightarrow N_\Sigma
\] 
which is a continuous proper surjective map. 
We exploit $\trop_\infty$ and $\trop_v$ to relate the complex space $X_{\Sigma,\infty}^\an$ and  the non-archimedean Berkovich space $X_{\Sigma,v}^\an$.

Lagerberg \cite{lagerberg-2012} has introduced real valued $(p,q)$-forms on $N_\R$ which form a bigraded sheaf of algebras $A_{N_\Sigma}^{\cdot,\cdot}$ with differential operators $d'$ and $d''$. 
{Lagerbergs construction} was generalized by Jell, Shaw and Smacka \cite{jell-shaw-smacka2015} to polyhedral spaces as $N_\Sigma$. 
For an open subset $U$ of $N_\Sigma$, the space $D^{p,q}(U)$ of Lagerberg currents  of bidegree $(p,q)$ is the topological dual of $A_c^{p,q}(U)$. 
In \cite{burgos-gubler-jell-kuennemann1},
  we have worked out a suitable positivity notion for Lagerberg currents  of bidegree $(p,p)$. 
For a summary and a comparison to the corresponding notions in complex and non-archimedean geometry, we refer to Section \ref{sec:posit-real-compl}.

\subsection{Tropical psh functions} 
In Section \ref{sec: Plurisubharmonic functions on partial compactifications}, we introduce psh functions on an open subset $U$ of $N_\Sigma$. 
These are strongly upper-semicontinuous functions $\varphi\colon U  \to \R \cup \{-\infty\}$ which restrict to convex functions on $U\cap N_\R$. 
We will show in Theorem \ref{thm:1} that every psh function on $U$ is continuous. 
This is not true for classical psh functions {on complex manifolds}. 
We show that psh functions on $U$ share many properties with complex psh functions. 
In Theorem \ref{psh and positivity of Lagerberg current}, we prove that a function $\varphi$ on $U$ is psh if and only if $\varphi$ is a locally integrable strongly upper-semicontinuous function such that the associated Lagerberg current is positive. 
In Theorem \ref{thm:3}, we will see that any psh function on $U$ is locally a decreasing limit of smooth psh functions. 
In Section \ref{sec:bedf-tayl-calc}, we introduce a Bedford--Taylor product for psh functions on the tropical toric variety $N_\Sigma$ associated to the fan $\Sigma$ in $N_\R \simeq \R^n$.

\begin{theointro} \label{tropical Bedford-Taylor in intro}
For an open subset $U$ of $N_\Sigma$ and  locally bounded psh functions $u_0,\ldots,u_q$ on $U$, there is a unique Lagerberg current 
\[
u_0 d'd''u_1 \wedge \ldots \wedge d'd''u_q \in D^{q,q}(U)
\]
such that this product is {determined} locally in $U$, agrees with the product of Lagerberg forms in the smooth case, and is continuous with respect to uniform convergence of psh functions.
\end{theointro}

For $u_0 \equiv 1$, we show that the Bedford--Taylor product $d'd''u_1 \wedge \ldots \wedge d'd''u_q$ is a closed positive Lagerberg current on $U$. 
For $q=n$, we get a positive Radon measure on $U$ which we call the associated \emph{Monge--Amp\`ere measure}. 
If $\Sigma$ is a smooth fan, then we allow in the Bedford--Taylor product in Theorem \ref{tropical Bedford-Taylor in intro} additionally a factor given by a closed positive $(p,p)$-current $T$ (see Theorem \ref{thm:7}). 
Moreover, we give in Theorem \ref{thm:8} a Bedford--Taylor product for unbounded psh functions if the unbounded loci intersect transversally. 

The proofs of the above results on the tropical Bedford--Taylor product use that the functions $v_j \coloneqq u_j \circ \trop_\infty$ are psh functions on the open subset $V \coloneqq \trop_\infty^{-1}(U)$ of the complex toric manifold $X_{\Sigma,v}^{\rm an}$  and that $T=\trop_{\infty,*}(S)$ for a canonical positive $(p,p)$-current $S$ on $V$ by the main theorem in \cite{burgos-gubler-jell-kuennemann1}. 
Then we define 
\[
u_0 d'd''u_1 \wedge \ldots \wedge d'd''u_q \wedge T \coloneqq \trop_{\infty,*}(v_0 d'd''v_1 \wedge \ldots \wedge d'd''v_q \wedge S)
\]
by using the corresponding complex Bedford--Taylor product on $V$ and the claim follows from the complex case and local regularization. For a non-smooth fan, we reduce Theorem \ref{tropical Bedford-Taylor in intro} to the smooth case by toric resolution of singularities and use a projection formula (see Remark \ref{projection formula for unbounded psh}). In principle, it is possible to prove Theorem \ref{tropical Bedford-Taylor in intro} purely by means of real analysis, but a reduction to the complex case by using tropicalization is very handy and gives the results much quicker. Such comparison principles in the toric case are a permanent guideline in the whole paper.

\subsection{Tropical $\theta$-psh functions} 
For a complete fan $\Sigma$, we show in Proposition \ref{maximum principle} that psh functions on $N_\Sigma$ are constant. 
To get also an interesting global theory, we proceed similarly as in the complex case and introduce $\theta$-psh functions.
Here, we fix a closed $(1,1)$-current $\theta$ given by a similar construction as the first Chern current {associated with} a continuous metric of a line bundle in the complex or non-archimedean setting.
In Section \ref{section: semipositive metrics and theta-psh functions}, we show that {our construction makes sense also} for non-complete fans and we introduce $\theta$-psh functions on any open subset $U$ of $N_\Sigma$. 
Since $\theta$-psh functions can be locally identified with psh functions, most of the properties of psh functions generalize to $\theta$-psh functions. 
In particular, we have a Bedford--Taylor product for $\theta$-psh functions as in Theorem \ref{tropical Bedford-Taylor in intro}.

In Theorems \ref{main comparison thm of psh} and
\ref{main correspondence theorem for semipositiv in non-arch}, we show the following \emph{correspondence theorem}. We consider any non-archimedean field $K$ with valuation $v$. For $w \in \{\infty, v\}$, let $\SS_w$ be the maximal compact torus in the dense torus orbit of $X_{\Sigma,w}^{\an}$. We pick $\theta$ as above and set $\theta_w \coloneqq \trop_w^*(\theta)$.

\begin{theointro} \label{main correspondence theorem for theta-psh functions in intro}
Let $\Sigma$ be a smooth fan and let $U$ be an open subset of $N_\Sigma$.
Then there are canonical isomorphisms between the following cones:
\begin{enumerate}
	\item\label{thm-b-i} 
	the cone of continuous $\SS_\infty$-invariant $\theta_\infty$-psh functions $\varphi_\infty$ on $\trop_\infty^{-1}(U)$;
	\item\label{thm-b-ii}	
	the cone of continuous $\theta$-psh functions on $U$;
	\item\label{thm-b-iii}
	the cone of continuous $\SS_v$-invariant $\theta_v$-psh functions $\varphi_\infty$ on $\trop_v^{-1}(U)$.
\end{enumerate}
For $w 	\in \{\infty, v\}$, the correspondence is given by $\varphi \mapsto \varphi_w \coloneqq \varphi \circ \trop_w$. 
\end{theointro}

The above bijective correspondence is true for any fan $\Sigma$. For simplicity, we consider in this paper only psh functions on complex manifolds. There is also a  theory of psh functions on  singular complex analytic spaces as described in \cite{demailly_1985}. If $\pi:X_{\Sigma'}\to X_\Sigma$ is a toric desingularization, then pull-back gives an isomorphism from the space of psh-functions on $U$ to the space of psh-functions on $\pi^{-1}(U)$ \cite[Theorem 1.7]{demailly_1985} which can be used to prove Theorem \ref{main correspondence theorem for theta-psh functions in intro} and many other statements in this paper also in the singular case.

We generalize the correspondences in Theorem \ref{main correspondence theorem for theta-psh functions in intro} also to singular $\theta$-psh functions. 
In \S \ref{global-regularization}, we prove for a complete fan
$\Sigma$ a global regularization result for $\theta$-psh functions. 

\subsection{Toric Monge--Amp\`ere equations} 
In the non-archimedean case, we show that the correspondence of continuous psh-functions {in Theorem  \ref{main correspondence theorem for theta-psh functions in intro}} is compatible with the tropical and non-archimedean Bedford--Taylor products (see Theorem \ref{comparision to non-archimedean Bedford-Taylor}). 
Let $\mu$ be a positive Radon measure on an open subset $U$ of the $n$-dimensional tropical toric variety $N_\Sigma$. 
In the complex case, we assume again that $\Sigma$ is a smooth fan, then there is a unique $\SS_\infty$-invariant Radon measure on $\mu_\infty$ on the toric manifold $X_{\Sigma,\infty}^\an$ with ${\trop_\infty}(\mu_\infty)= \mu$.  
In the non-archimedean case, there is a canonical section $\iota_v\colon N_\Sigma \to X_{\Sigma,v}^\an$ of $\trop_v$ which allows to identify $N_\Sigma$ with the canonical skeleton of $X_{\Sigma,v}^\an$.  
There is a unique Radon measure $\mu_v$ on $X_{\Sigma,v}^\an$ supported on the canonical skeleton $N_\Sigma$ and with  ${\trop_v}(\mu_v)= \mu$. 
In Proposition \ref{local MA equation}, we show for $w \in
\{v,\infty\}$  that  
\begin{equation} \label{correspondence of local MA equations}
(d'd''\varphi + \theta)^{\wedge n}=\mu \quad \Leftrightarrow \quad (d'd''\varphi_w + \theta_w)^{\wedge n}=\mu_w
\end{equation}
using  $\varphi_w \coloneqq \varphi \circ \trop_w$ for continuous $\theta$-psh functions $\varphi$ on $U$ as in Theorem \ref{main correspondence theorem for theta-psh functions in intro}.

We assume that $\Sigma$ is a complete fan and that $L$ is an ample line bundle on the toric variety $X_\Sigma$. For any non-archimedean field $K$ with valuation $v$, Boucksom and Jonsson \cite[Theorem 6.9]{bj:singular_metrics} introduce a Monge--Amp\`ere measure $c_1(L_v^\an,\metr)^{\wedge n}$ for   singular semipositive metrics $\metr$  on $L_v^\an$ which are (locally) bounded or more generally of finite energy. The construction is based on the continuous case using continuity along decreasing nets.  Similarly as in complex geometry, we can define the \emph{non-pluripolar Monge--Amp\`ere operator} 
\begin{equation*} 
\mu_{\metr} \coloneqq \lim_{k \to \infty} 1_{\{\metr < e^k\metr_0 \}} c_1(L_v^\an,\min(\metr_v,e^k\metr_0))^{\wedge n}
\end{equation*}
for any singular semipositive metric $\metr_v$ on $L_v^\an$. The construction does not depend on the choice of a fixed model metric $\metr_0$ of $L$. In general, we have $\mu(X_{\Sigma,v}^\an) \leq \deg_L(X)$. We define $\Ecal(L_v^\an)$ to be the space of singular semipositive metrics on $L_v^\an$ where equality is achieved. 

\begin{theointro} \label{solution of non-arch global toric MAeq in intro}
	Under the above assumptions, let $\mu_v$ be a positive Radon measure on $X_{\Sigma,v}^{\rm an}$ which is supported on the canonical skeleton $N_\Sigma$ satisfying $\mu_v(N_\Sigma \setminus N_\R)=0$ and with $\mu_v(X_{\Sigma,v}^{\rm an})= \deg_{L}(X_\Sigma)$. 
	Then there is a  metric $\metr_v$ in $\Ecal(L_v^\an)$  solving the non-archimedean  Monge--Amp\`ere equation 
	$c_{1}(L^{\an}_{v},\metr_{v})^{\wedge n}=\mu_v$.
\end{theointro} 

If $K$ is a discretely or trivially valued field of residue characteristic $0$ and $\mu_v$ is a measure of finite energy, then the above theorem follows from the solution of the non-archimedean Monge--Amp\`ere equation of Boucksom and Jonsson \cite[Theorem B]{bj:singular_metrics}. 
However, there are Radon measure $\mu_v$ as in Theorem \ref{solution of non-arch global toric MAeq in intro} which are not of finite energy and the result holds for any non-archimedean field $K$. 
In Remark \ref{generalization of main global thm}, we give also a natural generalization of Theorem \ref{solution of non-arch global toric MAeq in intro} for $L$ semiample and big.

In the proof of Theorem \ref{solution of non-arch global toric MAeq in intro}, we show first the corresponding statement for $\theta$-psh functions on $N_\Sigma$ which can be done either by referring to the second boundary problem for the real Monge--Amp\`ere operator (see Theorem \ref{thm:6}) or by referring to the solution of the corresponding complex Monge--Amp\`ere equation due to Boucksom, Eyssidieux, Guedi and Zeriahi \cite{boucksometal10:_monge_amper_big}. 
Then the correspondence results for $\theta$-psh functions yield the claim. 
For details, we refer to the proof of Theorem \ref{solution of non-arch global toric MAeq}.

\subsection{Abelian varieties} In Section \ref{section: totally degenerate abelian varieties}, we show that a similar correspondence as in Theorem \ref{main correspondence theorem for theta-psh functions in intro} is possible between $\theta$-psh functions on polarized complex and non-archimedean totally degenerate abelian varieties over any non-archimedean field $K$ with non-trivial valuation $v$. This is a bit surprising as abelian varieties are not defined over $\Spec(\Z)$ in contrast to toric varieties. For simplicity, we assume here in the introduction that $K$ is algebraically closed. 

An  abelian variety $A_v$ over $K$ is called \emph{totally degenerate}
if there is a \emph{non-archimedean Tate uniformization} $A_v^\an
\simeq T_v^\an/\Lambda$ for a multiplicative torus $T$  and a discrete
subgroup $\Lambda$ of $A(K)$. Let $N$ be the cocharacter lattice of
$T$, then it is additionally required that $\trop_v\colon T_v^\an \to N_\R$ maps $\Lambda$  isomorphically onto a complete lattice in $N_\R$ and we will identify $\Lambda$ with this lattice. 
Note that $\Lambda$ can be intrinsically attached to $A_v^\an$ and hence we get a canonical tropicalization map 
\[
\tropbar_v\colon A_v^\an \longrightarrow N_\R/\Lambda
\]
onto the tropical abelian variety $N_\R/\Lambda$.

We consider now a polarization $\lambda$ on the totally degenerate
abelian variety $A_v$ over $K$. This is an isogeny $\phi\colon A_v \to A_v^\vee$ to the dual abelian variety induced by an ample line bundle $L_v$. Let $\T_v^\vee$ be the torus with character lattice $\Lambda$, then we have the Tate uniformization $(A_v^\vee)^\an \simeq 
(\T_v^\vee)^\an/M$ for $M \coloneqq \Hom(N,\Z)$ and so the
lift of $\varphi$ to the Tate uniformizations gives a homomorphism
$\lambda\colon \Lambda \to M$. Using that $M$ is the character lattice of $\T$, set $[\gamma,m] \coloneqq m(\gamma)$ for $\gamma \in \Lambda$ and $m \in M$. We define $\trop_v(A_v,\phi)\coloneqq (\Lambda,M,[\cdot,\cdot],\lambda)$ and 
we will see in \S \ref{subsection: tropicalization of polarized totally degenerate abelian varieties} that the following properties hold:
\begin{enumerate}
	\item \label{lattice axiom}
	$M$ and $\Lambda$ are finitely generated free abelian groups of rank $n$;
	\item \label{pairing axiom}
	the paring $[\cdot,\cdot]$ extends to a non-degenerate real bilinear form $\Lambda_\R \times M_\R \to \R$;
	\item \label{polarization axiom}
	the homomorphism $\lambda\colon  \Lambda \to M$ has finite cokernel of cardinality $d$;
	\item \label{scalar product axiom}
	the bilinear form $[\cdot,\lambda(\cdot)]$ on $\Lambda$ is symmetric and positive definite. 
\end{enumerate}
A tuple $(\Lambda,M,[\cdot,\cdot],\lambda)$ with these four properties is called a \emph{polarized tropical abelian variety of degree $d^2$} as studied in \cite{foster-rabinoff-shokrieh-soto}. By \ref{pairing axiom}, we identify $\Lambda_\R = N_\R$  and the bilinear form in \ref{scalar product axiom} extends to {an euclidean inner product} $b$ on $N_\R$. Similarly as in Riemannian geometry, $b$ induces a smooth positive $(1,1)$-form $\omega$ on  the underlying tropical abelian variety  $N_\R/\Lambda$. For details, we refer to \S \ref{subsection:Polarized tropical abelian varieties} where we study polarized tropical abelian varieties.

Back to the case of a totally degenerate abelian variety $(A_v,\phi)$, we get a canonical smooth $(1,1)$-form $\omega_v \coloneqq \tropbar_v^*(\omega)$ on $A_v^\an$ which agrees with the first Chern form of a canonical metric of the ample line bundle $L_v$.

It is well-known that a complex abelian variety $A_\infty$ has a \emph{complex Tate uniformization} $A_\infty^\an \cong \T_\infty^\an/\Lambda$ for a multiplicative torus $\T_\infty$ over $\C$. Note that $\T_\infty^\an$ is not simply connected.  
 The torus $\T$ and the lattice $\Lambda$ are not intrinsically associated to $A_\infty$. This is in contrast to the classical uniformization $A_\infty^\an \simeq  V/U$ where $V$ is the tangent space of $A_\infty^\an$ at $0$ and $U$ is the complete lattice $H_1(A_\infty^\an,\Z)$ in $V$. 
 
We note that a polarization on a complex abelian variety $A_\infty$ is given by a Riemann form $H$, i.e.~a positive definite sesquilinear form $H$ on  $V$ such that the alternate form  $E = \im(H)$ takes values in $\Z$ on the lattice $U$. A symplectic basis for $E$ in $U$ generates isotropic sublattices $U_1$ and $U_2$ of rank $n$ with $U = U_1 \oplus U_2$. We will see in \S \ref{subsection: tropicalization of polarized complex abelian varieties} that 
we have a complex Tate uniformization with cocharacter lattice $N \simeq U_1$ of the torus $\T_\infty$ and with $\Lambda \simeq U_2$. Similarly as in the non-archimedean case, we can associate to $(A_\infty,H,U_1,U_2)$ a canonical polarized tropical abelian variety $\trop_\infty(A_\infty,H,U_1,U_2) = (\Lambda,M,[\cdot,\cdot],\lambda)$
and a canonical tropicalization map $\tropbar_\infty\colon A_\infty^\an \longrightarrow N_\R/\Lambda$.
For the canonical $(1,1)$-form $\omega$ of $\trop_\infty(A_\infty,H,U_1,U_2)$, we note that $\omega_\infty \coloneqq \tropbar_\infty^*(\omega)$ is the canonical K\"ahler form of the polarization $H$ of $A_\infty$.

For any polarized totally degenerate abelian variety $(A_v,\phi)$ over $K$, there is a corresponding decomposed complex polarized abelian variety $(A_\infty,H,U_1,U_2)$ with 
$$\trop_\infty(A_\infty,H,U_1,U_2) = (\Lambda,M,[\cdot,\cdot],\lambda)=\trop_v(A_v,\phi).$$
We will see this in  \S \ref{subsection: totally degenerate and complex abelian varieties}. For $w \in \{\infty,v\}$, the maximal compact torus $\SS_w$ in $\T_w^\an$ acts on $A_w^\an$ through the Tate uniformization. 
We will prove the following analogue of Theorem \ref{main correspondence theorem for theta-psh functions in intro} for corresponding polarized complex, tropical and totally degenerate abelian varieties:

\begin{theointro} \label{main correspondence theorem for theta-psh functions on av in intro}
	For any open subset $W$ of the tropical abelian variety $N_\R/\Lambda$,  there are canonical isomorphisms between the following cones:
	\begin{enumerate}
		\item the cone of continuous $\SS_\infty$-invariant $\omega_\infty$-psh functions $\varphi_\infty$ on ${\tropbar}_\infty^{-1}(W)$;
		\item the cone of  $\omega$-psh functions on $W$;
		\item the cone of continuous $\SS_v$-invariant $\theta_v$-psh functions $\varphi_\infty$ on ${\tropbar}_v^{-1}(W)$.
	\end{enumerate}
	For $w 	\in \{\infty, v\}$, the correspondence is given by $\varphi \mapsto \varphi_w \coloneqq \varphi \circ \tropbar_w$. 
\end{theointro}

This follows from Theorem \ref{main correspondence theorem for theta-psh functions in intro} using that $W$ can be locally lifted to an isomorphic open subset of $N_\R$.  
Let $\mu$ be any positive Radon measure on the tropical abelian variety $N/\Lambda$ with polarization $(\Lambda,M,[\cdot,\cdot],\lambda)$ of degree $d^2$. Assuming the necessary condition $\mu(N_\R/\Lambda)= n! \cdot \sqrt{d}$, we will see in Corollary \ref{non-arch-ma-tot-deg-av'} that the tropical Monge--Amp\`ere equation
\begin{equation} \label{tropical MA-equation on abelian variety in intro}
(\omega + d'd''\varphi)^{\wedge n}=\mu
\end{equation} 
has an $\omega$-psh function $\varphi\colon N_\R \to \R$ as a solution, unique up to adding constants. Using Theorem \ref{main correspondence theorem for theta-psh functions in intro}, this follows from the corresponding case in complex K\"ahler geometry. 

Section \ref{section: short variant} is dedicated to solve the invariant Monge--Amp\`ere equation for an arbitrary abelian variety $A$ over any non-trivially valued non-archimedean field $K$ with non-trivial valuation. 
Using Raynaud extensions, Berkovich \cite[\S 6.5]{berkovich-book} has introduced a \emph{canonical skeleton} as a compact subset of $A^\an$ which is a strong deformation retraction.

\begin{theointro} \label{solution of non-archimedean MAeqn for abelian varieties in intro}
Let $L$ be an ample line bundle on any abelian variety $A$ over the non-trivially valued non-archimedean field $K$. Let $\mu$ be a positive Radon measure on $A^\an$ supported in the canonical skeleton of $A$ and with $\mu(A^\an)	= \deg_L(A)$. Then there is a continuous semipositive metric $\metr$ on $L^\an$,  unique up to scaling, with
\[
c_1(L,\metr)^{\wedge \dim(A)}= \mu.
\]
\end{theointro}
Uniqueness follows from a much more general result of Yuan and Zhang. Existence was proven for totally degenerate abelian varieties by Liu \cite{liu2011} in case of such a measure $\mu$ with smooth density proving that the solution is a smooth metric. 
If $A$ has good reduction, then the canonical skeleton is a single point and the solution is the canonical metric on $L^\an$.

In \S \ref{subsection: descent}, we give a descent argument which reduces Theorem \ref{solution of non-archimedean MAeqn for abelian varieties in intro} to algebraically closed fields $K$ and for such fields we give the argument in \S \ref{short subsection: Raynaud uniformization and canonical tropicalization}--\S \ref{short subsection: MA equations on abelian varieties}. For existence, we use that a corresponding tropical Monge--Amp\`ere equation as in \eqref{tropical MA-equation on abelian variety in intro} has a solution. Then we lift the solution to $A^\an$ along the canonical tropicalization map and 
  a similar approximation argument as in \cite{gubler-compositio} shows semipositivity.

\subsection{Notation and conventions} \label{section: notation} 
\addtocontents{toc}{\protect\setcounter{tocdepth}{1}}

The set of natural numbers $\N$ includes $0$. For any $t \in \R$, let $\R_{\geq t} \coloneqq \{r \in \R \mid r \geq t\}$. 
We write $\Rsup =\R \cup \{ \infty \}$ and $\Rinf  =\R \cup \{ -\infty \}$. 
 A \emph{lattice} is a free $\Z$-module of finite rank. A \emph{lattice} in a finite dimensional real vector space is a discrete subgroup. 
 For a ring $A$, the group of invertible elements is denoted by $A^\times$.
 A \emph{variety} over a field $F$ is an integral scheme which is of finite type and separated over $\Spec\, F$.

In this paper, we denote by $K$ a \emph{non-archimedean field} which is
 a field $K$  complete with respect to
a given ultrametric absolute value $|\phantom{a}|\colon K\to \R_{\geq 0}$. The  \emph{valuation} is $v \coloneqq -\log|\phantom{a}|$ and $\Gamma \coloneqq v(K^\times)$ is the \emph{value group}. We have the \emph{valuation ring} $K^\circ \coloneqq \{\alpha \in K \mid v(\alpha)\geq 0\}$ with maximal ideal $K^{\circ\circ}\coloneqq \{\alpha \in K \mid v(\alpha)> 0\}$ and \emph{residue field} $\ktilde \coloneqq \kcirc/K^{\circ \circ}$. 
If $Y$ is a variety over $K$, then $\Yan$ denotes the analytification of $Y$ as a Berkovich analytic space.

Let $N$ be a free abelian group of rank $n$, $M=\Hom_\Z(N,\Z)$ its dual and denote by $N_\R$ resp.~$M_\R$ the respective scalar extensions to $\R$. 
We also fix a subgroup $\Gamma$ of $\R$, usually the value group of a non-archimedean field. 
A function $f\colon N_\R \to \R$ is called \emph{affine} if $f=u+c$ for some $u \in M_\R$ and $c \in \R$. 
We call $u$ the \emph{slope} of $f$  and \emph{integral slope} means $u \in M$.  
We say that $f$ is \emph{integral $\Gamma$-affine} if $u \in M$ and $c \in \Gamma$. 

A \emph{polyhedron in $N_\R$} is  a finite intersection of half-planes $\{f \leq 0\}$ for affine functions $f$ on $N_\R$. A \emph{polytope} is a bounded polyhedron. A polyhedron is \emph{integral $\Gamma$-affine} if the  affine functions $f$ can be chosen integral $\Gamma$-affine. The \emph{relative interior} of a polyhedron  $\sigma$ is denoted by $\relint(\sigma)$. We use $\Delta^\circ$ for the \emph{interior of $\Delta$ in $N_\R$}.
A \emph{face} of a polyhedron $\sigma$ is  the intersection of $\sigma$ with the boundary of a half-space containing $\sigma$. 
We also allow $\sigma$ and $\emptyset$ as faces of $\sigma$. 
{We write $\tau\prec\sigma$ if $\tau$ is a face of a polyhedron $\sigma$.}
The \emph{open faces} of $\sigma$ are the relative interiors of the faces of $\sigma$. 
A \emph{polyhedral complex} in $N_\R$ is a finite set $\Pi$ of polyhedra in $N_\R$ such that for $\sigma\in \Pi$ all faces of $\sigma$ are in $\Pi$ and for another $\rho \in \Pi$ we have that $\sigma \cap \rho$ is a common face of $\sigma$ and $\rho$. The \emph{support} of $\Pi$ is defined by  $|\Pi| \coloneqq \bigcup_{\sigma \in \Pi} \sigma$.

We mean by a \emph{fan} a finite polyhedral complex $\Sigma$ in $N_\R$ consisting of strictly convex polyhedral rational cones. We denote by $X_\Sigma$ the corresponding toric variety with dense torus $\torus$ and by $N_\Sigma$ the corresponding partial compactification of $N_\R$. We refer to \cite{burgos-gubler-jell-kuennemann1} for details about $N_\Sigma$ which we often call the \emph{tropical toric variety associated to $\Sigma$}.

A \emph{piecewise affine function $f$} is usually defined on a finite union $P$ of polyhedra $\sigma$ such that $f|_\sigma$ is affine. There is always a polyhedral complex $\Pi$ with $P=|\Pi|$ such that $f|_\sigma$ is affine for every $\sigma \in \Pi$. Then we say that $f$ is  \emph{piecewise affine with respect to $\Pi$}. If we may choose $\Pi$ as a fan $\Sigma$  and if $f|_\sigma \in M_\R$ for all $\sigma \in \Sigma$, then $f$ is called \emph{piecewise linear}. 

We refer to \cite[Appendix A]{burgos-gubler-jell-kuennemann1}
for our conventions about Radon measures. We recall the following useful fact for a locally compact Hausdorff space $Y$ with a continuous  action by a compact group $G$. 
Let $\pi\colon Y \to X \coloneqq Y/G$ be the quotient which is also a locally compact Hausdorff space. 
Then for any positive Radon measure $\mu_X$ on $X$, there is a unique $G$-invariant positive Radon measure $\mu_Y$ on $Y$ with image measure  $\pi(\mu_Y)= \mu_X$. 
This follows  by averaging compactly supported continuous functions with respect to the probability Haar measure on $G$ and then using the Riesz representation theorem, see \cite[chap.~VII, \S 2]{bourbaki-integration-7-8}.

\subsection*{Acknowledgements}
We are grateful to Ana Botero, Yanbo Fang, Roberto Gualdi, Mattias Jonsson, Yifeng Liu, and César Martínez for their comments on a previous version.

\section{Forms and currents}
\label{sec:posit-real-compl}
In this section we fix notation for forms and currents in the complex,
tropical and non-archimedean setting.
We also introduce and discuss positivity of forms and currents. 
In the toric setting, we relate complex (resp.~non-archimedean) forms
and currents to Lagerberg forms and currents on the associated
tropical variety.

\subsection{The complex situation} \label{section: complex situation}

Let $X$ be a connected complex manifold of dimension $n$. 
We denote by $A=A_X$ the sheaf of complex smooth differential forms on
$X$, which is a bigraded differential sheaf of complex algebras with
respect to the differential operators $\dell$ and $\dellbar$. 
We will use also the differential operators $d=\dell + \dellbar$ and 
$d^c= \frac{1}{2 \pi i}(\dell - \dellbar)$ on $A$. 
Note that $dd^c = \frac{i}{\pi}\dell \dellbar$ and our definition of $d^c$ follows the convention from pluripotential theory as in \cite{demailly_agbook_2012}.
Observe that our $d^c$ is two times the $d^c$-operator used in 
Arakelov theory \cite{gillet-soule1990, soule-etal1992}.

Dually, we have the  sheaf of complex currents on $X$ which we denote
by $D$. The bigrading of $D^{p,q}$ is made in such a way that  
a locally integrable form  $\omega$ of bidegree $(p,q)$ on an open
subset $V$ of $X$ gives rise to an associated current $[\omega] \in
D^{p,q}(V)$ by
\[
[\omega]\colon A_{c}^{n-p,n-q}(V) \longrightarrow \R , \quad
\alpha \mapsto [\omega](\alpha) \coloneqq \int_V \omega \wedge \alpha
\]
where $A_c(V)$ denotes the {space}  of compactly
supported smooth differential forms as usual.

A differential form  $\omega$ on $V$ of bidegree $(p,p)$ is called
\emph{positive}, if for some $m \in \N$, there exist non-negative functions $f_j\colon V \to \R$ and $(p,0)$-forms $\alpha_j$ for $j=1, \ldots, m$, such that
\[ 
\omega = \sum_{j=1}^m i^{p^2} f_j \alpha_j \wedge \overline \alpha_j.
\]
Here, we differ from the terminology used in
\cite[III.1.A]{demailly_agbook_2012},
where the weakly positive forms from 
\cite{burgos-gubler-jell-kuennemann1} are called positive.
We observe that weakly positive forms are always positive
if $p\in\{0,1,n-1,n\}$ and
refer to \cite[Proposition 2.3.5]{burgos-gubler-jell-kuennemann1} for
an explanation, why our positivity notation is more suitable for
comparing with Lagerberg forms.
A current $T \in D^{p,p}(V)$ is called \emph{positive}, if $T(\omega)
\geq  0$ for all positive forms $\omega \in A_c^{n-p,n-p}(V)$.

A function $\varphi\colon V\to \Rinf$  is called \emph{strongly upper semicontinuous} if  for every {null set} $E\subset V$  and every  $a\in V$, we have
\begin{equation} \label{def:2'}
\varphi(a) =  \limsup_{\substack{x \to a, \,{x \notin E}}} \varphi(x).
\end{equation}

We refer to \cite[I.5]{demailly_agbook_2012} for the definition and the study of \emph{plurisubharmonic functions} (psh functions for short). 
For $V$ connected, a function $\varphi\colon V \to \Rinf$, not identically $-\infty$, is psh if and only if $\varphi$ is strongly usc, locally integrable and ${dd^{c}} [\varphi]$ is a positive current \cite[Th\'eor\`eme II.3]{lelong68:psh}.
Bedford--Taylor theory shows that for locally bounded psh functions $v_1, \ldots, v_q$ on $V$ and a closed positive current $T \in D^{p,p}(V)$ on $V$, there is a unique way to define a positive current
\begin{equation} \label{first Bedford-Taylor product}
v_1 dd^c v_{2} \wedge \ldots \wedge 
dd^c v_{q}\wedge T \in D^{p+q-1,p+q-1}(V)
\end{equation}
and\ a closed positive current 
\begin{equation} \label{second Bedford-Taylor product}
dd^c v_{1} \wedge \ldots \wedge 
dd^c v_{q}\wedge T   \in D^{p+q,p+q}(V)  
\end{equation}
such that the products \eqref{first Bedford-Taylor product} and \eqref{second Bedford-Taylor product} agree with the usual products in the case of smooth psh functions $v_1, \ldots, v_q$ and such that these products are continuous along monotone decreasing sequences with respect to weak convergence of currents. For details and a generalization to the unbounded case, we refer to \cite[Sections III.3-4]{demailly_agbook_2012}.

\subsection{The tropical situation} \label{section: tropical situation}
We recall some facts from \cite{burgos-gubler-jell-kuennemann1}. 
Let $N$ be a a free abelian group of rank $n$ with dual $M=\Hom_\Z(N,\Z)$ and let $N_\R \coloneqq N \otimes_\Z\R$. 
We consider a fan $\Sigma$ in $N_\R$ and the \emph{associated partial compactification $N_\Sigma$} of $N_\R$.
The partial compactification $N_\Sigma$ is also called \emph{tropical toric variety} and has a natural stratification
\[
N_\Sigma = \coprod_{\sigma \in \Sigma} N(\sigma)
\] 
with $N(\sigma) \coloneqq N_\R/\langle \sigma \rangle_\R$ \cite[Section 3.1]{burgos-gubler-jell-kuennemann1}.
The toric variety $X_\Sigma$  comes with a continuous tropicalization map
\[
\trop_\infty\colon X_{\Sigma,\infty}^{\rm an} \longrightarrow N_\Sigma
\]
where $X_{\Sigma,\infty}^{\rm an}$ is the complex analytification of $X_\Sigma$. 
We denote by $A=A_{N_\Sigma}$ the sheaf of smooth Lagerberg forms on $N_\Sigma$. 
It is a  bigraded sheaf of real algebras with  differentials  $d',d''$. 
We have the Lagerberg involution $J$ which sends $A^{p,q}$ to $A^{q,p}$. 
Similarly as in the complex case, we have a bigraded sheaf $D$ of Lagerberg currents on $N_\Sigma$ with differentials $d',d''$ 
\cite[Section 3.2]{burgos-gubler-jell-kuennemann1}. 
By integration again, any locally integrable Lagerberg form $\omega$ has an associated Lagerberg current $[\omega]$. 
Note that integration on $N_\R$ depends on the underlying lattice $N$ in $N_\R$. 
Sections of $A^{0,0}$ are called \emph{smooth functions}.

A Lagerberg form $\omega$ of bidegree $(p,p)$ on the open subset $U$
of $N_\Sigma$ is called \emph{positive}, if for some $m \in \N$ there
exist functions $f_j\colon U \to \R_{\geq 0}$ and $(p,0)$-forms
$\alpha_j$ on $U$ for $j=1, \ldots, m$ such that
\[
\omega = \sum_{j=1}^m (-1)^{\frac{p(p-1)}{2}}f_j\alpha_j \land J(
\alpha_j ).
\]
We call $T \in D^{p,p}(U)$ \emph{symmetric} if $T(J\omega)=(-1)^qT(\omega)$ for all $\omega \in A_c^{q,q}(U)$ where $q \coloneqq n-p$. A \emph{positive Lagerberg current}  is a symmetric $T \in D^{p,p}(U)$ with $T(\omega) \geq 0$ for all positive $\omega \in A_c^{q,q}(U)$.

\begin{rem} \label{same top dim forms}\label{locally integrable gives current}
Let $(p,q)$ be with $\max(p,q) = n$ and let $U \subset N_\Sigma$ be open. 
By definition of the forms on $U$, they are constant towards the boundary and hence any $\omega \in A^{p,q}(U)$ has support in $U \cap N_\R$. 
In particular, the inclusion $\AS^{p,q}_c(U \cap N_\R) \to \AS^{p,q}_c(U)$ is an isomorphism. 
We conclude that the dual restriction map $D^{n-p,n-q}(U) \to
D^{n-p,n-q}(U \cap N_\R)$ is  an isomorphism.  
For a locally integrable form $\eta \in A^{n-p,n-q}(U \cap N_\R)$, we get  
$[\eta] \in D^{n-p,n-q}(U \cap N_\R)=D^{n-p,n-q}(U)$. 
In particular, a locally integrable function $f$ on $U \cap N_\R$
  defines a current $[f]\in D^{0,0}(U)$ on the whole $U$.
\end{rem}

\begin{rem} \label{tropical spaces}
The partial compactification $N_\Sigma$ is a special case of a \emph{polyhedral space}. 
We refer to \cite[\S 2]{jell-shaw-smacka2015} for
definitions and for an introduction of the sheaf $A$ of smooth
Lagerberg forms on polyhedral spaces.
Integration of compactly supported top dimensional forms is well-defined on rational weighted polyhedral spaces.  
A special case of a rational weighted polyhedral space is a \emph{tropical space} where  the weights satisfy a balancing condition. 
\end{rem}

\begin{rem}\label{toric morphism motivation}
  Let $N'$ be a free abelian group of finite rank and let $\Sigma'$ be
  a fan in $N_\R'$.  Recall that a morphism
  $\psi \colon X_{\Sigma'} \to X_{\Sigma}$ is called \emph{toric} if
  $\psi$ restricts to a homomorphism on the dense tori and if $\psi$
  is equivariant with respect to the torus actions.  Toric morphisms
  $ X_{\Sigma'} \to X_{\Sigma}$ are in bijective correspondence to
  group homomorphisms $N' \to N$, which map every cone of $\Sigma'$
  into a cone of $\Sigma$ \cite[Theorem 3.3.4]{cox-little-schenck}.

  More generally, let $L\colon N' \to N$ be a homomorphism of
  cocharacter lattices inducing the homomorphism
  $\rho\colon \torus' \to \torus$ of associated split complex tori.
  We recall from \cite[\S 3.2]{bps-asterisque} that a
  \emph{$\rho$-equivariant morphism
    $\psi\colon X_{\Sigma'} \to X_\Sigma$} of
  $\torus'$-(resp.~$\torus$-)toric varieties is given by the
  composition of a translation by an element in $\torus'$ with a toric
  morphism to a stratum of $X_\Sigma$.  Clearly, $\rho$-equivariant
  toric morphisms $\psi\colon X_{\Sigma'} \to X_\Sigma$ are in
  bijective correspondence to $L$-equivariant morphisms of tropical
  toric varieties $E\colon {N'_{\Sigma'}}\to N_\Sigma$ as defined below.
\end{rem}

\begin{definition} \label{new morphism of toric tropical varieties}
  Let $N'$ be a free abelian group of finite rank and let $\Sigma'$ be
  a fan in $N_\R'$.  Let $L\colon N' \to N$ be a homomorphism of
  abelian groups.  A map $E\colon {N'_{\Sigma'}}\to N_\Sigma$ is
  called \emph{$L$-equivariant morphism of tropical toric varieties
    (or of partial compactifications)}, if the map $E$ is continuous
  and $L_\R$-equivariant.
\end{definition}

\begin{rem}\label{remark-toric-morphism}
Let $E\colon {N'_{\Sigma'}}\to N_\Sigma$ be an $L$-equivariant
morphism of tropical toric varieties.
\begin{enumerate}
\item\label{remark-toric-morphism:1}
The map $E$ is a morphism of polyhedral spaces by \cite[\S 2.1.4]{jell-thesis}. 
Hence we have a pull back $E^* \colon A^{p,q}(U') \to A^{p,q}(U)$ for every open subset $U$ of $N_{\Sigma}$ and $U' \coloneqq {E}^{-1}(U)$. 
\item\label{remark-toric-morphism:2}
The map $E$ is proper (in the sense of topological spaces), if and only if
the preimage of $|\Sigma|$ with respect to $L_\R\colon N_{\R}'\to N_{\R}$ 
equals $|\Sigma'|$.
Observe that this is also equivalent to properness of 
the corresponding $\rho$-equivariant morphism $\psi \colon X_{\Sigma'} \to X_{\Sigma}$ from Remark \ref{toric morphism motivation}
\cite[Theorem 3.4.11]{cox-little-schenck}.
\item\label{remark-toric-morphism:3} 
Assume that $E$ is proper. 
Then we get a linear map $E^* \colon A^{p,q}_c(U) \to A^{p,q}_c(U')$ which is continuous with respect to the locally convex topologies defined in \cite[3.2.4]{burgos-gubler-jell-kuennemann1}. 
By duality, we get a linear map $E_* \colon D^{p, q}(U') \to D^{p+m-n, q+m-n}(U)$ which is continuous with respect to the weak topology of currents.  
\item\label{remark-toric-morphism:4}
Pull back of forms and push forward of currents along $E$ respect positivity.
\end{enumerate}
\end{rem}

\begin{art} \label{comparison with complex forms} We assume now that
  $\Sigma$ is a smooth fan.  Let $U$ be an open subset of $N_\Sigma$
  and let $V \coloneqq \trop_\infty^{-1}(U)$ be the corresponding
  $\SS$-invariant open subset of the complex toric manifold
  $X_{\Sigma,\infty}^{\rm an}$ where $\SS=\SS_\infty$ denotes the
  maximal compact torus in the dense orbit $\T_\infty^\an$ of
  $X_{\Sigma,\infty}^{\rm an}$.  There is a unique homomorphism
  \cite[Proposition 4.1.7]{burgos-gubler-jell-kuennemann1}
\[
\trop_\infty^*\colon A(U) \longrightarrow A(V)
\]
with 
\begin{math}
\trop_\infty^{\ast}\varphi = \varphi\circ \trop_\infty
\end{math} for all $\varphi \in A^{0,0}(U)$ and which
satisfies
\begin{equation}
\label{correspondences of differential operators}
\trop_\infty^{\ast} \circ \,d' = \pi ^{-1/2}\partial\circ \trop_\infty^{\ast},\quad
\trop_\infty^{\ast} \circ \,d'' = \pi ^{-1/2}i\bar\partial\circ \trop_\infty^{\ast}.
\end{equation}
This homomorphism 
 respects the bigrading, positivity of forms and integration of top
 dimensional forms. On the  $\SS$-invariant forms in $A(V)$, we have
 defined an antilinear involution $F$. Let $A(V)^{\SS,F}$ be the real
 subalgebra of $\SS$- and $F$-invariant forms. Then we have
 $\trop_\infty^*A(U) \subset A(V)^{\SS,F}$ with equality if $U \subset N_\R$.
 By equation \eqref{correspondences of differential operators}, we obtain
   \begin{equation}\label{eq:3}
     \trop_\infty^{\ast} \circ \,d'd'' = dd^{c}\circ \trop_\infty^{\ast}.
   \end{equation}

Let $D(V)^{\SS,F}$ be the space of $\SS$- and $F$-invariant currents on $V$. By duality, we get a linear map
\[
\trop_{\infty,*}\colon  D^{p,p}(V)^{\SS,F} \longrightarrow D^{p,p}(U).
\]
By \cite[Theorem 7.1.5]{burgos-gubler-jell-kuennemann1}, $\trop_{\infty,*}$ maps the cone of closed positive currents inside of
$D^{p,p}(V)^{\SS,F}$ isomorphically onto the cone of closed positive
Lagerberg currents in $D^{p,p}(U)$.
\end{art}

\begin{prop} \label{projection formula trop current cont function}
Let $f$ be a function on an open subset $U$ of $N_\Sigma$ with values in $\R \cup \{\pm \infty\}$. 
Assume that the function $f \circ \trop_\infty$ on the complex manifold $V \coloneqq \trop_\infty^{-1}(U)$ is locally integrable. 
Then we have 
\begin{align*}
\trop_{\infty,*} [f \circ \trop_\infty] = [f].
\end{align*}
\end{prop}

\begin{proof}
Note first that local integrability of $f \circ
\trop_\infty$ implies local integrability of
$f|_{U \cap N_\R}$. 
Hence the Lagerberg current $[f]$ is well defined 
by Remark \ref{locally integrable gives current}. 
It is enough to show that 
\begin{equation} \label{projection formula with function and form}
\int_V (f \circ \trop_\infty) \cdot \trop_\infty^*(\omega)= \int_U f \cdot \omega
\end{equation}
for all $\omega \in A_c^{n,n}(U \cap N_\R)$. 
By  \cite[Corollary 5.1.8]{burgos-gubler-jell-kuennemann1}, the identity \eqref{projection formula with function and form} holds for smooth functions $f$. Note that integration against the forms $\omega$ and $\trop_\infty^*(\omega)$ give  Radon measures on $U$ and $V$, respectively, with compact supports. Since we can approximate a continuous function  uniformly by smooth functions in a neighbourhood of $\supp(\omega)$, we conclude that \eqref{projection formula with function and form} holds for all continuous functions $f$ on $U$. Since a Radon measure with compact support in $U \cap N_\R$ is induced by a unique signed  Borel measure, it follows that  \eqref{projection formula with function and form} is true for all locally integrable functions $f$ on $U \cap N_\R$. 
\end{proof}

\subsection{The non-archimedean situation} \label{section: non-archimedean situation}

Let  $K$ be a non-archimedean complete field with valuation $v = -\log
| \phantom{a}|$.
First, we recall the construction of forms and currents of
Chambert-Loir and Ducros on analytic spaces over $K$.  
For more algebraic approaches, we refer to \cite{gubler-forms},
\cite{gubler-kuenne2018} and \cite{jell-2019}.
Then we will study forms and currents in the toric setting.

\begin{art} \label{Chambert-Loir-Ducros forms}
Let  $X$ be an $n$-dimensional connected good $K$-analytic space whose topology is Hausdorff (see \cite{berkovich-book} and \cite[\S 1]{berkovich-ihes}). Chambert-Loir and Ducros have introduced forms and currents on $X$ and we refer to  \cite{chambert-loir-ducros} for  details about the following facts.

The sheaf $A = A_X$ of smooth differential forms on $X$ has similar properties as in the complex case. 
It is a bigraded differential sheaf of real algebras with respect to natural differential operators $d'$ and $d''$. 
Its construction is based on tropical charts and Lagerberg forms as follows: 
A \emph{moment map} on an open subset $W$ of $X$ is a morphism
$\varphi\colon W \to \mathbb (\G_{\rm m}^r)^\an$ leading to the
tropicalization map $\trop_\varphi \coloneqq \trop_v \circ
\varphi\colon W \to \R^r$. 
There is a bigraded differential homomorphism $\trop_\varphi^*$ from
the sheaf of smooth Lagerberg forms on $\R^r$ to the sheaf of
differential forms on $W$.  
If $\trop_\varphi(W)$ is contained in the compact support $|\mathscr C|$
of a polyhedral complex $\mathscr C$ in $\R^{r}$ and 
$\alpha,\beta\in A^{p,q}(N_\R)$ have the same restriction to 
polyhedra in $\mathscr C$,
then $\trop_\varphi^*(\alpha)=\trop_\varphi^*(\beta)$ in $A^{p,q}(W)$.
For more details about the construction of $A$ we refer to \cite[\S
3.1]{chambert-loir-ducros}.

There is a unique involution $J$ on $A$ which sends $A^{p,q}$ to $A^{q,p}$ and which is compatible with the Lagerberg involution on tropical charts. 
Smooth differential forms of bidegree $(0,0)$ are functions which we call \emph{smooth functions}. 

A differential form $\omega \in A^{p,p}(V)$ on an open subset $V$ of $X$ is called \emph{positive} if there exist $m \in \N$, smooth functions $f_j\colon V \to \R_{\geq 0}$ and forms $\alpha_j \in A^{p,0}(V)$ for $j=1, \ldots, m$ such that 
\[
\omega = \sum_{j=1}^m (-1)^{\frac{p(p-1)}{2}}f_j\alpha_j \land J(
\alpha_j ).
\]

We assume now that $X$ has no boundary in the sense of \cite[\S 2.5, \S 3.1]{berkovich-book}. 
This is the case for example for an open subset in the Berkovich analytification of an algebraic variety over $K$. 
By duality there is also a bigraded sheaf $D$ of currents on $X$ with differentials $d',d''$. 
Moreover, one can integrate compactly supported forms on $X$ of bidegree $(n,n)$. 
For every open subset $V$ of $X$, we get that any $\omega \in A^{p,q}(V)$ has an associated current $[\omega] \in D^{p,q}(V)$. 
Symmetric and positive currents on $V$ are defined similarly as the corresponding notions in the tropical case. 
\end{art}

\begin{art} \label{non-archimedean tropicalization with boundary}
Let $N$ be a free abelian group of rank $n$ with dual $M$ and let $\torus$ be the multiplicative torus with character lattice $M$. 	
From now on, we consider the $\torus$-toric variety $X_\Sigma$ associated to a fan $\Sigma$ in $N_\R$. Its Berkovich analytification over $K$ is denoted by $X_{\Sigma,v}^{\rm an}$. The \emph{Kajiwara--Payne tropicalization map} $\trop_v \colon X_{\Sigma,v}^{\rm an} \to N_\Sigma$  is the unique continuous map given by the usual tropicalization map on  $\torus_v^{\rm an}$. The latter is characterized by 
$$\langle \trop_v(x), u \rangle = v \circ \chi^u(x)$$
for $x \in \torus_v^{\rm an}$ and any $u \in M$ with associated character $\chi^u$. Note that $\trop_v$ is a proper map of topological spaces. For details, see \cite{kajiwara2008} and  \cite{payne-2009a}.

We have a canonical section $\iota\colon N_\Sigma \to X_{\Sigma,v}^{\rm an}$ of $\trop_v$. 
Since $\iota$ is a homeomorphism onto a closed subset of $\Xsigmaan$, 
we may view $N_\Sigma$ as a closed subset of $X_{\Sigma,v}^{\rm an}$ which we call the \emph{canonical skeleton}. For details, we refer to \cite[\S 4.1]{bps-asterisque}.

\end{art}

\begin{art} \label{SS-invariant forms, sets and functions}
Let $\SS=\SS_v$ denote the maximal affinoid torus in $\torus_v^{\rm
  an}$ acting  on $X_{\Sigma,v}^{\rm an}$  by the  morphism
\begin{align*}
m \colon  \SS \times X_{\Sigma,v}^{\rm an} \longrightarrow X_{\Sigma,v}^{\rm an}. 
\end{align*}
Let $p_2$ the second projection of $ \SS \times X_{\Sigma,v}^{\rm
  an}$, where we take products in the category of $K$-analytic spaces.
We consider an \emph{$\SS$-invariant open subset} $V$ of
$X_{\Sigma,v}^{\rm an}$, i.e.~we require that $m^{-1}(V) =
p_2^{-1}(V)$.
A form $\alpha \in \AS^{p,q}(V)$ 
with $m^*\alpha = p_2^*\alpha$ is called \emph{$\SS$-invariant}. 
Similarly, a function $f$ on $V$ is said to be \emph{$\SS$-invariant},
if $f \circ m= f \circ p_2$.

Note that $\SS$ is a $K$-analytic group. However the underlying
topological space of $\SS$ does not inherit a group structure.
Nevertheless there are well-defined orbits \cite[Chapter 5]{berkovich-book}.
The $\SS$-orbits of $X_{\Sigma,v}^{\rm an}$ are precisely the fibers of 
the map $\trop_v$, such that 
$\trop_v \colon X_{\Sigma,v}^{\rm an} \to N_\Sigma$ is the quotient by
$\SS$ in the sense of topological spaces \cite[Proposition
4.2.15]{bps-asterisque}.
\end{art}

Using the identity on $\torus_v^{\rm an}$ as a moment map, we get from \ref{Chambert-Loir-Ducros forms} a canonical homomorphism 
\begin{equation} \label{non-archimedean pull-back on torus}
\trop_v^*\colon A_{N_\R} \longrightarrow   (\trop_v)_*(A_{\torus_v^{\rm an}})
\end{equation}
of bigraded differential sheaves of  algebras with respect to the differentials $d',d''$.

\begin{prop} \label{prop: non-archimedean pull-back on toric variety}
	The homomorphism in \eqref{non-archimedean pull-back on torus} extends uniquely to a homomorphism 
\begin{equation} \label{non-archimedean pull-back on toric variety}
\trop_v^*\colon A_{N_\Sigma} \longrightarrow   (\trop_v)_*(A_{X_{\Sigma,v}^{\rm an}})
\end{equation}	
of bigraded  sheaves of algebras {compatible}
with the differentials $d',d''$. For any open $U$
in $N_\Sigma$ and $\alpha \in A^{p,q}(U)$, we have that
$\trop_v^*(\alpha)$ is an  $\SS$-invariant form on $\trop_v^{-1}(U)$.
\end{prop}

\begin{proof}
By \cite{jell-2019}, one can construct the
sheaf of smooth forms on $X_{\Sigma,v}^\an$ using toric charts with boundaries. 
The existence of the morphism \eqref{non-archimedean pull-back on
  toric variety} is a natural consequence of this construction.
For the convenience of the reader, we include a
direct proof.

Complex forms and Lagerberg forms are determined by their restriction
to the dense stratum. 
This proves uniqueness. 
It remains to define $\trop_v^*(\alpha)\in
  A^{p,q}(\trop_v^{-1}(U))^{\SS}$ for any open subset $U$ of
$N_\Sigma$ and any $\alpha \in A^{p,q}(U)$.  
It is enough to give the definition locally
around a given point $u \in U$. There is a unique $\sigma \in \Sigma$
with $u \in N(\sigma)$.  
Let $\pi_\sigma\colon N_\R \to N(\sigma)$ be the quotient map. 
By  definition of  Lagerberg forms on $U$, there is an open
neighbourhood $U_u$ of $u$ in $U$ such that $\alpha =
\pi_\sigma^*\alpha_u$ on $U_u$ for some $\alpha_u \in A^{p,q}(U_u \cap
N(\sigma))$.
We note that $ \pi_\sigma \circ \trop_v$ is the tropicalization of a
moment map and hence we can use \ref{Chambert-Loir-Ducros forms} to
define $\omega \coloneqq \trop_v^*(\alpha)$ on $V_u \coloneqq
\trop_v^{-1}(U_u)$.
Since $\omega |_{V_u \cap \torus_v^{\rm an}}$ agrees with the
definition in \eqref{non-archimedean pull-back on torus} proving existence. 
	
	Note that $V_u$ is an $\SS$-invariant open subset of $V \coloneqq \trop_v^{-1}(U)$. On $\SS \times V_u$, we have 
	\begin{equation*} \label{identity of tropicalizations of moment maps}
	 \pi_\sigma  \circ \trop_v \circ p_2 = \pi_\sigma  \circ \trop_v \circ m.
	\end{equation*}
    These are  tropicalizations of  moment maps with  $p_2^*\omega= (\pi_\sigma  \circ \trop_v \circ p_2)^*(\alpha_u)$ and $m^*\omega=(\pi_\sigma  \circ \trop_v \circ m)^*(\alpha_u) $ on $\SS \times V_u$. This proves $\SS$-invariance of $\omega$ on $V_u$ and hence on $V$.	
\end{proof}

In the complex situation (see \ref{comparison with complex forms}),

the analogue of the morphism \eqref{non-archimedean
pull-back on torus} is an
isomorphism onto the subsheaf of $\SS$- and $F$-invariant forms of  
$(\trop_v)_*(A_{\torus_v^{\rm an}})$. Since $F$ is only
related to the  complex structure, one might hope
in the non-archimedean case that \eqref{non-archimedean pull-back on
  torus} gives an isomorphism onto the subsheaf of $\SS$-invariant
forms. The following is a counterexample.

\begin{ex} \label{counterexample for invariant smooth on torus}
Let $N\coloneqq \Z$ and $\Sigma\coloneqq \{(0)\}$
with $X_\Sigma=\torus = \mathbb G_{\rm m}$ and $N_\Sigma=\R$. 
We claim that the function $f \coloneqq \min(0, \trop_v)^2$ is an
$\SS$-invariant smooth function on $\torus_v^{\rm an}$ which is
\emph{not} contained in  $\trop_v^*(\AS^{0,0}(\R))$.

Indeed, the function $\min(0,u)^2$ is not smooth on $\R$ and hence $f
\not\in \trop_v^*(\AS^{0,0}(\R))$. Since $f$ factors through 
$\trop_v$, it is $\SS$-invariant.  To show  $f \in
\AS^{0,0}(\torus_v^{\rm an})$, note first that 
\begin{displaymath}
f(x) = (-\log \vert T(x) \vert)(-\log \vert (T-1)(x) \vert)
\end{displaymath}
for $x \in \torus_v^{\rm an}$ where $T$ is the toric coordinate on $\torus$. 
We consider the closed embedding $\varphi \colon \torus \to \torus^2,
x \mapsto (x, x-1)$. Let $u_1, u_2$ denote the coordinates on
$\trop_v((\torus_v^{\rm an})^2) = \R^2$. Then $u_1 \cdot u_2 \in
\AS^{0,0}(\trop_\varphi(\torus_v^{\rm an}))$. We conclude that
$\trop_\varphi^*(g) = f \in \AS^{0,0}(\torus_v^{\rm an})$ proving the
claim. 
\end{ex}

\begin{prop} \label{lem: support on skeleton}
	Let $X_\Sigma$ be a toric variety of dimension $n$. 
	Let $\omega$ be a $(p,q)$-form on an open subset $U$ of $N_\Sigma$
	with $\max(p,q) = n$. 
	Then $\supp(\trop_v^*(\omega)) = \iota(\supp(\omega))$,
	where $\iota \colon N_\Sigma \to X_{\Sigma,v}^{\an}$ 
	is the embedding of the canonical skeleton as in 
	\ref{non-archimedean tropicalization with boundary}. 
\end{prop}
\begin{proof}
	We have seen  $\supp(\omega) \subset N_\R$ in Remark \ref{same top dim forms} and similarly  $\supp(\trop_v^*\omega) \subset  \T_v^{\rm an}$ by \cite[Lemme 3.2.5]{chambert-loir-ducros}, so 
	we may assume $X_\Sigma = \T$. 
	Let $x \in \T_v^{\an} \setminus \iota(N_\R)$
	and let $y = \iota(\trop_v(x))$. 
	Then there exists a regular function $f$ on $\T$ such that
	$\vert f(x) \vert \neq \vert f(y) \vert$. 
	By the fundamental theorem in tropical geometry, the tropicalization $\trop_v(f)$ is a piecewise affine function on $N_\R$ whose singularity locus $V(\trop_v(f))$ is an $(n-1)$-dimensional polyhedral complex in $N_\R$ equal to $\trop_v(V(f)_v^{\rm an})$ for the zero set $V(f)$ of $f$ (see \cite[\S 3.1]{maclagan-sturmfels}).
For
\begin{align*}
\varphi \colon \T \longrightarrow \T \times \A^1, z \longmapsto (z, f(z)),
\end{align*}
we see that $\trop_\varphi(x) \neq \trop_\varphi(y)$ and $\trop_\varphi(x) \notin \trop_\varphi(\iota(N_\R)))$. 
The first projection restricts to a map $\pi \colon \trop_\varphi( \T_v^{\rm an}) \to \trop_v( \T_v^{\rm an})=N_\R$ which is injective away from $V(\trop_v(f))$. 
	Let 
	\begin{align*}
	W \coloneqq \trop_\varphi^{-1}(\trop_\varphi(\T_v^{\rm an}) \setminus \trop_\varphi(\iota(N_\R))).
	\end{align*}
	Since $\trop_v$ is a proper map and since the canonical skeleton $\iota(N_\R)$ is closed in $\T_v^{\rm an}$, the set $W$ is an open neighborhood of $x$ in  $\torus_v^{\rm an}$.
	For any $x' \in W$ and $y' \coloneqq \iota(\trop_v(x'))$, we have   
	$$\pi(\trop_\varphi(x'))=\trop_v(x')=\trop_v(y')=\pi(\trop_\varphi(y')).$$
	Since $x'\in W$, we have $x' \neq y'$ and hence the above shows that $\trop_v(x') \in V(\trop_v(f))$. This yields 
    $\trop_v(W) \subset V(\trop_v(f))$.  
	Since we have $\max(p,q) = n > \dim(V(\trop_v(f)))$, we have $\omega|_{V(\trop_v(f))} = 0$ and  
	hence $\trop_v^*(\omega)|_W = 0$. 
	This shows $x \notin \supp(\trop_v^*(\omega))$ and since $x$ was arbitrary in $\torus_v^{\rm an} \setminus \iota(N_\R)$,
	we have $\supp(\trop_v^*(\omega)) \subset \iota(N_\R)$. 
	By  \cite[Corollaire 3.2.3]{chambert-loir-ducros}, we have 
	$
	\trop_v(\supp(\trop_v^*(\omega)))=\supp(\omega),
	$
	and the claim follows.
\end{proof}

\begin{art} \label{push-forward with respect to trop}
Let $U$ be an open subset of $N_\Sigma$ and let $V \coloneqq \trop_v^{-1}(U)$. 
Then we have a linear map 
\[
\trop_{v,*} \colon D^{p,q}(V) \longrightarrow D^{p,q}(U), \quad T \longmapsto T \circ \trop_v^*.
\]
For $\alpha \in A^{r,s}(U)$ and $T \in D^{p,q}(V)$, the projection formula 
	\begin{equation} \label{projection formula non-arch}
	\trop_{v,*}(\trop_v^*(\alpha) \wedge T)= \alpha \wedge \trop_{v,*}(T).
	\end{equation} 	
is an easy formal consequence of our definitions. 
\end{art}

\section{Plurisubharmonic functions on partial compactifications} \label{sec: Plurisubharmonic functions on partial compactifications}

Let $N$ be a free abelian group of finite rank $n$.
We always consider $\Rinf\coloneqq \R \cup \{-\infty\}$ as an ordered additive 
	monoid equipped with the unique topology which turns the natural 
	bijection $\exp\colon \Rinf\stackrel{\sim}{\to} [0,\infty)\subset\R$
	into a homeomorphism. We fix a fan $\Sigma$ in $N_\R$ with associated partial compactification 
	$N_\Sigma$.
	We denote by $|\Sigma|$ the support of the fan $\Sigma$.

\subsection{Definition and basic properties of plurisubharmonic functions}\label{section-psh-functions}

\begin{definition}\label{definition-psh-open-case}
{ Let $U$ be an open subset of $N_\R$. 
A function $\varphi \colon U \to \Rinf$ is called \emph{convex} if 
\begin{align*}
\varphi\bigl( (1 - t) x + ty\bigr) \leq (1-t) \varphi(x) + t \varphi(y)
\end{align*} 
holds for all $t \in [0,1]$ and all $x,y\in U$ such that
the segment $[x,y]\coloneqq \{(1- t)x + t y \mid t \in [0,1]\}$
is contained in $U$.
A function $\varphi \colon U \to \Rinf$ is 
called \emph{finite} if $\varphi(U)\subset \R$. }
\end{definition}

For $p\in N_\R$ and $v\in |\Sigma|$, {we have shown in \cite[Lemma 3.1.4]{burgos-gubler-jell-kuennemann1} that the limit $p+\infty v\coloneqq \lim_{\mu\to\infty}p+\mu v$ exists in $N_\Sigma$ and $p+\infty v$ lies in the stratum $N(\sigma)$ for the unique cone $\sigma \in \Sigma$ containing $v$ in its relative interior. 
We will also use the \emph{compactified half line}
\[
[p,p+\infty v]\coloneqq \{p+\mu v\,|\,\mu\in [0,\infty]\}.
\]}

\begin{definition} \label{def:1}
	Let $U\subset N_\Sigma$ be an open subset. 
	A function $\varphi \colon U \to \Rinf$ is called  
	\emph{plurisubharmonic} (for short \emph{psh}) if
	\begin{enumerate}
		\item \label{item:1} 
		the function $\varphi$ is upper-semicontinuous (for short \emph{usc}),
		\item \label{item:2}
		the function $\varphi |_{U \cap N_\R}$ is convex (in the sense of
		Definition \ref{definition-psh-open-case}),
		\item \label{item:3}
		for any $p\in N_{\R}$ and  $v\in |\Sigma |$ with $[p,p+\infty v] \subset U$,
		we have
		$\varphi(p+\infty v)\le \varphi(p)$.  
	\end{enumerate}
\end{definition}

\begin{prop}\label{prop:1-3}
Let $U$ be open subset in $N_\Sigma$ and let $(\varphi_{i})_{i\in I}$ be  psh functions on $U$.
	\begin{enumerate}
		\item \label{item:prop1-3:1} 
		Plurisubharmonicity is a local property.
		\item \label{item:prop1-3:2} 
		Let $\varphi(z)=\sup_{i\in I}\varphi_{i}(z)\in [-\infty,\infty]$. 
		If we have $\varphi(z)<\infty$ for all $z\in U$ and if $\varphi\colon U\to \Rinf$ is usc 
		(both conditions hold automatically if $I$ is finite),  
		then $\varphi$ is psh.
		\item \label{item:prop1-3:3}
		If $\varphi_{i+1}(z)\le \varphi_{i}(z)$ for all $i\in \N$ and all $z\in U$, 
		then the limit $\varphi(z)=\lim_{i}\varphi_{i}(z)$ exists pointwise and 
		the function $\varphi\colon U\to \Rinf$ is psh. 
	\end{enumerate}
\end{prop}

\begin{proof}
	This is shown easily using standard properties of convex functions.
\end{proof}

Let $U\subset N_\Sigma$ be an open subset. 
We call a subset $E \subset U$ a \emph{null set} if $E \cap N_\R$ is a
null set with respect to the Lebesgue measure on $N_\R$. 
A function $\varphi\colon U\to \Rinf$ is called \emph{strongly upper semicontinuous} if  for every {null set} $E\subset U$  and every  $a\in U$, we have 
\begin{equation} \label{def:2}
\varphi(a) =  \limsup_{\substack{x \to a, \,{x \notin E}}} \varphi(x).
\end{equation}

\begin{thm} \label{thm:1}
Let $U\subset N_\Sigma$ be an open subset. {For a function $\varphi\colon U \to \Rinf $, the following conditions are equivalent:
 \begin{enumerate}
 	\item \label{psh condition} The function $\varphi$ is psh.
 	\item \label{continuity condition} The function $\varphi$ is
          continuous and $\varphi|_{U \cap N_\R}$ is convex.
 	\item \label{strongly usc condition} The function $\varphi$ is
          strongly upper semicontinuous and $\varphi|_{U \cap N_\R}$
          is convex.
\item \label{lim sup condition}
The function $\varphi|_{U\cap N_\R}$ is convex and for every point $x\in U\setminus N_{\R}$ we have
\begin{equation}\label{eq:15}
{\varphi}(x)=\limsup_{\substack{y\in U\cap N_{\R}\\y\to x}}
{\varphi}(y).
\end{equation}
 \end{enumerate} 
} 
\end{thm}

The proof will be given after a series of lemmata.

\begin{lem}\label{lemm:1} 
Let $\varphi\colon (M,\infty)\to\R$ be a convex function for some $M\in \R$ such that $\lim _{x \to \infty}\varphi(x) <\infty$. 
Then ${\varphi}$ is monotone decreasing. 
\end{lem}

\begin{proof} 
This is an easy exercise and will be left to the reader.
\end{proof}

\begin{lem}\label{lemm:3}
Let $\varphi\colon U\to {\R_{-\infty}}$ be a psh function on an open subset $U$ of $N_\Sigma$ and let $\sigma \in \Sigma$. 
We consider $p\in N(\sigma )\cap U$, $v\in \relint(\sigma)$ and $q\in U\cap N_{\R}$ such that $p=q+\infty v$ and  $[q,q+\infty v] \subset U$. 
Then the function $\varphi|_{[q,q+\infty v]}$ is continuous. 
\end{lem}

\begin{proof}
We equip $[q,p]$ with the subspace topology from $U$.
By convexity, $\varphi|_{[q,p)}$ is continuous \cite[Section 10]{rockafellar1970}. 
Therefore we only need to show the continuity at the point $p$. 
Combining upper semicontinuity and condition \ref{item:3} in Definition \ref{def:1} we deduce that
\begin{displaymath}
\varphi(p)\ge \lim_{\mu  \to \infty} \varphi(q+\mu  v) \ge \varphi(p)
\end{displaymath}
proving continuity. 
\end{proof}

\begin{lem}\label{lemm:2}
Let $\varphi$ be a psh function on an open subset $U$ of $N_\Sigma$ and let $\sigma \in \Sigma$ such that $U \cap N(\sigma)$ is connected.  
Then $\varphi|_{U\cap N(\sigma )}$ is either identically $-\infty$ or  a finite convex function.
\end{lem}

\begin{proof}
  We start by showing that $\varphi|_{U\cap N(\sigma )}$ is
    convex. Let $p,q,r\in U\cap N(\sigma )$ and $\lambda \in (0,1)$ with
    $q=\lambda p+(1-\lambda)r$ and $[p,r] \subset U$. We choose  $v\in
    \relint(\sigma )$ and
    points $p'$, $q'$ and $r'$ in $U\cap N_{\R}$ such 
    that $q'=\lambda p'+(1-\lambda)r'$,
    \begin{displaymath}
      p=p'+\infty v,\quad  q=q'+\infty v,\quad  r=r'+\infty v 
    \end{displaymath}
    and for every $\mu \ge 0$ the inclusion $[p+\mu v,r+\mu v]\subset
    U$ holds. This can easily be achieved by the openness of $U$ and
    the compactness of $[p,r]$.    
    Then
    \begin{displaymath}
      \varphi(q'+ \mu v)\le \lambda \varphi(p'+ \mu v)+(1-\lambda )\varphi(r'+ \mu v)
    \end{displaymath}
    for all $\mu\geq 0$.
    Therefore, using Lemma \ref{lemm:3}, we get
    \begin{displaymath}
      \varphi(q)=\lim_{\mu \to \infty} \varphi(q'+ \mu v) 
      \le \lim_{\mu \to \infty}\bigr(\lambda \varphi(p'+ \mu
      v)+(1-\lambda )\varphi(r'+ \mu v)\bigr)
      \le \lambda \varphi(p)+(1-\lambda )\varphi(r)
    \end{displaymath}
    showing that $\varphi|_{U\cap N(\sigma )}$ is
    convex. The claim follows from the fact that a convex function on a
    connected open set with values in $\R_{-\infty}$ is either finite
    or identically $-\infty$.  
\end{proof}

\begin{proof}[Proof of Theorem \ref{thm:1}] 
We recall first some useful facts. 
Fix $p \in N_\Sigma$, then there is a unique $\sigma \in \Sigma$ with $p \in N(\sigma)$ and we denote 
by $\pi_\sigma:N_\R \to N(\sigma)$ the quotient map.  
For $p_0 \in N_\R$ with $\pi_\sigma(p_0)=p$ and an open bounded convex neighbourhood $\Omega$ of $p_0$ {in $N_\R$},   
set  
\begin{equation}\label{eq:20}
U(\sigma ,\Omega,p_0) \coloneqq \coprod _{\tau  \prec \sigma  }\pi _{\tau }(\Omega+\sigma){\subset N_\Sigma }.
\end{equation}
It follows from the second description of the topology of $N_\Sigma$ in \cite[Remark 3.1.2]{burgos-gubler-jell-kuennemann1} that the sets $U(\sigma ,\Omega,p_0)$ form a basis of \emph{open} neighbourhoods of 
$p$ in ${N_\Sigma}$. Note that the strata $\pi _{\tau }(\Omega+\sigma)$ of the basic open sets $U(\sigma ,\Omega,p_0)$ are open convex subsets  of $N(\tau)$. We will use often that $[u,u+\infty v] \subset U(\sigma ,\Omega,p_0)$ for any $u \in U(\sigma ,\Omega,p_0) \cap N_\R$ and any $v \in \sigma$.

To prove \ref{psh condition} $\Rightarrow$
\ref{continuity condition},  we have to show continuity of a psh function
$\varphi\colon U \to \Rinf$  in $p \in U$. 
We choose an open neighbourhood $U(\sigma,\Omega,p_0)$ of $p$ in $U$ as above, fix $v\in \relint(\sigma )$  and let $\varepsilon>0$. 
For $\mu \in [0,\infty]$, let $\varphi_{\mu }\colon
\Omega\to \R_{-\infty}$, $q \mapsto \varphi(q+\mu v)$. 
By Lemmata \ref{lemm:1} and \ref{lemm:3}, we get a decreasing net $(\varphi_\mu)_{\mu \in \R_+}$ of convex functions on $\Omega$ which converges pointwise to $\varphi_\infty$. 
Since $\pi_\sigma(\Omega)$ is an open convex neighbourhood of $p=p_0+\infty v$ in $N(\sigma)$, Lemma \ref{lemm:2} shows that $\varphi_\infty= \varphi \circ (\pi_\sigma|_\Omega)$ is either a finite  convex function  or identically $-\infty$.

We assume first $\varphi(p)>-\infty$ and hence $\varphi_\infty$ is a finite function on $\Omega$. 
By \cite[Theorem 10.8]{rockafellar1970}, the convergence of
$\varphi_{\mu  }$ to $\varphi_\infty$
is uniform on 
compact subsets of $\Omega$.  
Thus we can find a compact neighborhood $W$ of $p_0$ in $\Omega$  and $M \gg 0$ such that 
\begin{equation}\label{eq:1}
\varphi(p)-\varepsilon = \varphi_\infty(p_0)-\varepsilon \le \varphi_\mu(w)=\varphi(w+\mu v) \le \varphi_\infty(p_0)+\varepsilon = \varphi(p)+\varepsilon  
\end{equation}
for all $\mu \in [M,\infty]$ and all $w\in W$. We pick an open convex neighbourhood $\Omega_0$ of $p_0$ contained in $W$ and set $V \coloneqq U(\sigma,\Omega_0+Mv,p_0+M v)$. Since $\pi_\sigma(p_0+\mu v)=p$, we see that $V$ is a basic open neighbourhood of $p$ in $U$. Continuity in $p$ follows from the claim  
\begin{equation} \label{epsilon continuity}
-\varepsilon + \varphi(p) \leq \varphi(q) \leq \varphi(p) + \varepsilon
\end{equation}
for any $q \in V$. 
There is a unique $\tau \prec \sigma$ with $q \in N(\tau)$. Using \eqref{eq:20}, there is $q_0 \in \Omega_0$ and $s \in \sigma$ such that $q=\pi_\tau(q')$ for $q' \coloneqq q_0+Mv+s$. 
We pick $v' \in \relint(\tau)$. Then $q= q'+\infty v'$ and $[q',q'+\infty v'] \subset V=U(\sigma,\Omega_0+Mv,p_0+M v)$. By Lemma \ref{lemm:3}, the restriction of $\varphi$ to  $[q',q'+\infty v']$ is continuous and hence it is enough to prove \eqref{epsilon continuity} for $q \in V \cap N_\R$. Then we have $q=q_0+Mv+s$. 
We have $[q,q+\infty v] \subset V$ and $[q_M,q_M+\infty s] \subset V$ for $q_M \coloneqq q_0+Mv \in V$.
By Lemmata \ref{lemm:1} and \ref{lemm:3}, the restrictions of $\varphi$ to $[q_M,q_M+\infty s]$ and to $[q,q+\infty v]$ are continuous and decreasing. 
Together with \eqref{eq:1} for $w=q_0$, this shows
$$\varphi(p)-\varepsilon \leq \varphi_\infty(q) = \varphi(q+\infty v)   \leq \varphi(q) = \varphi(q_0+Mv+s)  \leq \varphi(q_M) \leq \varphi(p)+\varepsilon.$$

We assume now  $\varphi(p)=-\infty$. 
By Lemma \ref{lemm:3}, there is  $M \gg 0$ such that
$\varphi_M(p_0)=\varphi(p_{0}+Mv)<-1/\varepsilon$.
Since $\varphi_M$ is convex, it is continuous. 
There is an open convex neighbourhood $\Omega_0$ of $p_0$ in $\Omega$ such that 
\begin{equation} \label{eq:1-}
\varphi_M(w)  =
\varphi(w+ M v) <-1/\varepsilon 
\end{equation}
for all $w\in \Omega_0$.
Let $q \in V \coloneqq U(\sigma,\Omega_0+ Mv,p_0+ MV)$. Then $q=\pi_\tau(q')$ for $q' \coloneqq q_M+s$,   $q= q'+\infty v'$ and $[q',q'+\infty v'] \subset V$ as in the first case. 
The same arguments show that the restrictions of $\varphi$ to $[q',q'+\infty v']$ and to $[q_M,q_M+\infty s]$ are continuous and decreasing, hence 
$$\varphi(q) = \varphi(q'+\infty v')\leq \varphi(q')= \varphi(q_M+s)\leq \varphi(q_M)=\varphi(q_0+Mv)<-1/\varepsilon$$ 
by using \eqref{eq:1-} for $w=q_0 \in \Omega$ on the right. This shows continuity of $\varphi$ in $p$.

Obviously, \ref{continuity condition} implies \ref{strongly usc
    condition} and \ref{strongly usc
    condition} implies \ref{lim sup condition}. 
It remains to prove that
\ref{lim sup condition}  yields 
\ref{psh condition} and so we assume that 
$\varphi|_{U\cap N_\R}$ is convex and \eqref{eq:15} is
satisfied. Note that \eqref{eq:15} easily yields that $\varphi$ is usc. 
It remains to check \ref{item:3} of Definition \ref{def:1}. 
We argue by contradiction and assume that there is $p' \in N_\R$ and
$v\in \relint(\sigma)$ 
for some $\sigma \in \Sigma$ such that $[p',p'+\infty v]
\subset U$ and $\varphi(p'+\infty v)> \varphi(p')$.

The function $g \colon \R_{\geq 0} \to \R$, $\mu \mapsto
\varphi(p'+\mu v)$ is convex. Using \eqref{eq:15} and $\varphi(p'+\infty v) \neq \infty$, we get $\lim_{\mu \to \infty} g(\mu)<
\infty$. 
Lemma \ref{lemm:1} yields that $g$ is monotone
decreasing. 	
Note that $p \coloneqq p'+ \infty v =\pi_\sigma(p')\in N(\sigma)$. There is $p_0 \in N_\R$  with $\pi_\sigma(p_0)=p$ and  a bounded open convex neighbourhood $\Omega$ of $p_0$ such that 
$U(\sigma,\Omega,p_0) \subset U$. There is $M \gg 0$ such that $p_M' \coloneqq p'+Mv \in U(\sigma,\Omega,p_0)$. 
By assumption, there is $\varepsilon > 0$ such that 
$\varphi(p') < \varphi(p)- \varepsilon$. Since $\varphi(p_M')\leq \varphi(p')$ and 
$\varphi|_{U\cap N_\R}$ continuous,  
there is a bounded open convex neighbourhood $\Omega_M$ of $p_M'$ in $U(\sigma,\Omega,p_0) \cap N_\R$ such that $\varphi(w)  < \varphi(p)-\varepsilon$ for all $w \in \Omega_M$.
We claim that
\begin{equation} \label{last contradiction}
\varphi(q)  < \varphi(p)-\varepsilon
\end{equation}
for all $q$ in the  open neighbourhood $V \coloneqq U(\sigma,\Omega_M,p_M')$ of $p$ in $U$. This contradicts \eqref{eq:15}. 

To prove \eqref{last contradiction}, let $q \in V$. Then there is $\tau \prec \sigma$, $q_M \in \Omega_M$ and $s \in \sigma$ such that $q = \pi_\tau(q')$ for $q' \coloneqq q_M+s$. Note that $[q_M,q_M+\infty s]\subset V$. For $v' \in \relint(\tau)$, we have $q= q'+ \infty v'$ and $[q',q'+\infty v'] \subset V$.
By  \eqref{eq:15} and  Lemma \ref{lemm:1} again, the restrictions of $\varphi$ to $[q',q'+\infty v']$ and to $[q_M,q_M+\infty s]$ are decreasing.  Using $q_M \in \Omega_M$, we get
$$\varphi(q) \leq \varphi(q')=\varphi( q_M+s) \leq \varphi(q_M)  <  \varphi(p)-\varepsilon$$
proving \eqref{last contradiction}. This finishes the proof of ~\ref{lim sup condition} $\Rightarrow$ \ref{psh condition}.
\end{proof}

We give now two results which hold similarly for complex psh functions.

\begin{cor} \label{extension of psh functions}
Let $U$ be an open subset of $N_\Sigma$ and let $\varphi$ be a psh
function on $U \cap N_\R$. If  $\varphi$ is locally bounded from above
near the boundary $U \setminus U \cap N_\R$, then $\varphi$ extends
uniquely to a psh function on $U$.	
\end{cor}

\begin{proof}
We define the extension of $\varphi$ to $a \in U \setminus N_\R$ by
\begin{equation} \label{def:2-var}
\varphi(a) =  \limsup_{\substack{x \in U \cap N_\R\\x \to a }} \varphi(x).
\end{equation}	
Using that $\varphi$ is locally bounded near the boundary, we see that
$\varphi(a)<\infty$.
Since $\varphi$ is psh on $U\cap N_\R$, it is convex on $U\cap
N_\R$. By Theorem \ref{thm:1}, we deduce that $\varphi$ is psh on
$U$. 
\end{proof}

\begin{prop} \label{maximum principle}
If $\Sigma$ is a complete fan, then any psh function on $N_\Sigma$ is constant.	
\end{prop}

\begin{proof}
Let $\varphi\colon N_\Sigma \to \Rinf$ be a psh function. 
For any line $\ell$ in $N_\R$, the restriction of $\varphi$ is a convex function. 
We may write $\ell = p+\R v$ for any $p \in \ell$ and a non-zero $v \in N_\R$. 
By condition \ref{item:3} in Definition \ref{def:1}, we have $\lim_{t \to \infty}f(p\pm tv)<\infty$. 
By Lemma \ref{lemm:1}, the convex functions $\varphi(p\pm tv)$ are
decreasing in $t \geq 0$ and hence constant.  
Since this holds for any line, we get the claim.
\end{proof}

\begin{prop} \label{functoriality of psh}
Let $L\colon N' \to N$ be a homomorphism of free abelian groups of finite rank, let $E\colon N_{\Sigma'}' \to N_\Sigma$ be an $L$-equivariant morphism of tropical toric varieties as in Definition \ref{new morphism of toric tropical varieties} and let $\varphi \colon U \to \Rinf$ be a function on an open subset $U$ of $N_\Sigma$.
\begin{enumerate}
\item \label{psh lifts}
If $\varphi$ is psh on $U$, then $\varphi \circ E$ is psh on $ E^{-1}(U)$.
\item \label{psh descends}
If $E$ is a surjective proper map and if $\varphi \circ E$ is psh on $ E^{-1}(U)$, then $\varphi$ is psh on $U$.
\end{enumerate}
\end{prop}

\begin{proof}
We set $U' \coloneqq E^{-1}(U)$. 
Let $\varphi$ be psh on $U$. 
Note that $E(N_\R')\subset N(\sigma)$  for a unique $\sigma \in \Sigma$ and hence it follows from Lemma \ref{lemm:2} that $\varphi \circ E$ is convex on $U' \cap N_\R'$. 
By Theorem \ref{thm:1}, $\varphi$ is continuous. 
It follows that $\varphi\circ E$ is continuous and hence psh.
This proves \ref{psh lifts}.

To prove \ref{psh descends},
we  observe that surjectivity of $E$ yields that $E|_{N'_{\R}}\colon N'_{\R}\to N_{\R}$ is affine and surjective. 
Since $\varphi\circ E$ is $psh$, the restriction $\varphi\circ
E|_{U' \cap N'_{\R}}$ is convex. 
Therefore $\varphi|_{N_{\R}\cap U}$ is also convex.
Using that $E$ is proper and surjective, it is clear that $\varphi$ is continuous if and only if $\varphi \circ E$ is continuous. 
By Theorem \ref{thm:1}, we get \ref{psh descends}.
\end{proof}

\subsection{Relations to Lagerberg currents and complex geometry}
For Proposition \ref{thm:2} below, we assume that the fan $\Sigma$ is
smooth to obtain a characterization of psh functions on  
$N_{\Sigma}$ in terms of classical psh functions on the complex toric manifold
$X^\an_{\Sigma,\infty}$ using the complex tropicalization map
$\trop_\infty \colon X^\an_{\Sigma,\infty} \to N_\Sigma$.

\begin{prop}\label{thm:2} We assume that $\Sigma $ is a smooth fan in $N_\R$. 
Let $U$ be an open subset of $N_\Sigma$. Then $\varphi\colon U\to
\Rinf$ is psh if and only if $\varphi\circ
\trop_\infty$ is a psh function on  $\trop_\infty^{-1}(U)$.    
\end{prop}

\begin{proof}
We note first that the statement holds for an open subset $U$ of
  $N_\R$. Indeed, if $\varphi \circ \trop_\infty$ is psh, then
  \cite[Theorem I.5.13]{demailly_agbook_2012} shows that $\varphi$ is
  psh. The converse implication follows by approximating
  $\varphi$ locally and uniformly by  smooth convex functions.

Now we allow $U$ to hit the boundary. 
Assume  that $\varphi\circ \trop_\infty $ is psh. 
The above yields that $\varphi|_{N_{\R}\cap U}$ is convex.
By definition of psh functions in
complex geometry, the function 
$\varphi\circ \trop_\infty $ is usc.
Using that $\trop_\infty$ is surjective, we get
\begin{align*}
\{x\in U\,|\,\varphi(x)\geq t\}
=\trop \bigl(\{y\in \trop_\infty^{-1}(U)\,|
\,\varphi\circ \trop (y)\geq t\}\bigr)
\end{align*}
for all $t\in \Rinf$ and these sets are closed as the map 
$\trop_\infty$ is proper.
Therefore $\varphi$ is usc.

We next prove that  \ref{item:3} in Definition \ref{def:1} is satisfied. 
Let $p\in N_{\R}$ and $v\in |\Sigma |$ 
with $[p,p+\infty v] \subset U$. 
{There is a unique cone $\sigma \in \Sigma$ with $v \in \relint(\sigma)$. 
Choose a vector $v_0 \in N$ with 
$v_{0}\in \relint(\sigma )$.  
This can always be done because the fan is rational. 
The description of the topology of $N_\Sigma$ given in \cite[Remark 3.1.2]{burgos-gubler-jell-kuennemann1} and the proof of \cite[Lemma 3.1.4]{burgos-gubler-jell-kuennemann1} show that there is $\mu_0 \geq 0$ such that 
$p+\mu_0 v+ \mu v_0 \in U$ for all $\mu \ge 0$ and that $p+\infty v = p+\mu_0 v+ \infty v_0$.
We have seen that $\varphi|_{N_{\R}\cap U}$ is convex, $\varphi$ is usc and does not take the value $\infty$, hence Lemma  \ref{lemm:1} shows that $\varphi(p)\ge \varphi(p+\mu _{0} v)$.
Thus is enough to prove that $\varphi(p+\mu _{0} v)\ge \varphi(p+\infty v)$.} 
We denote by $\mathbb{E}$ the complex unit disc and 
let $\mathbb{E}^* \coloneqq \mathbb{E} \setminus \{0\}$. 
We choose a point $x\in X^\an_{\Sigma,\infty}$ with 
$\trop_\infty(x)=p+\mu _{0}v$ and consider the map
  \begin{displaymath}
    h\colon \mathbb{E}^\ast\longrightarrow X^\an_{\Sigma,\infty}, \quad t \longmapsto v_{0}(t)\cdot x
  \end{displaymath}
where we view $N$ as 
the group of $1$-parameter subgroups of the torus $\T$. 
This map extends to a holomorphic curve 
$\widetilde h\colon \mathbb{E}\to X^\an_{\Sigma,\infty}$ 
with $\trop_\infty(\widetilde h(0))=p+\infty v$ and
  $\widetilde h(\mathbb E) \subset \trop_\infty^{-1}(U)$. 
The restriction of $\varphi\circ \trop$ to $\mathbb E$ is a psh function. Therefore
\begin{align*}
\varphi(p+\infty v) &= \varphi(\trop_\infty(\widetilde h(0)))\\
&\le \int_{0}^{1}\varphi(\trop_\infty(\widetilde h(e^{2\pi i\theta})))d\theta
= \varphi(\trop_\infty(\widetilde h(1)))=\varphi(p+\mu _{0}v).
\end{align*}

Conversely, assume that $\varphi$ is psh. 
{The beginning of the proof yields that the restriction of  $\varphi\circ \trop_\infty$ to $\trop_\infty^{-1}(U)\cap \T_\infty^{\rm an}$ is psh.
Using that $\varphi$ is continuous by Theorem  \ref{thm:1},} we deduce from  \cite[Theorem I.5.24]{demailly_agbook_2012} that $\varphi\circ\trop_\infty$ is psh.  
\end{proof}

Recall from Remark \ref{locally integrable gives current} that a function $\varphi$ on an open subset of $N_\Sigma$ with $\varphi|_{U \cap N_\R}$ locally integrable yields a Lagerberg current $[\varphi]$ on $U$. 
This applies if $\varphi|_{U \cap N_\R}$ is finite and convex.

  \begin{thm} \label{psh and positivity of Lagerberg current}
    Let $\Sigma$ be a fan and let $\varphi\colon U \to \Rinf$ be
    a function on a connected open subset $U \subset N_\Sigma$ such
    that $\varphi \not\equiv -\infty$. Then $\varphi$ is psh if and only if
    \begin{enumerate}
    \item \label{psh and positivity of Lagerberg current-i}    
    $\varphi$ is strongly upper semicontinuous;
    \item \label{psh and positivity of Lagerberg current-ii}
    $\varphi|_{U\cap N_\R}$ is  locally integrable  and
      $d'd''[\varphi]$ is a positive Lagerberg current on $U$.
    \end{enumerate}
  \end{thm}

\begin{proof}
Let $\varphi$ be a psh function.
We show that $\varphi$ satisfies \ref{psh and positivity of Lagerberg
  current-i} and \ref{psh and positivity of Lagerberg current-ii}.
We first consider the case where the fan $\Sigma$ is smooth.
Assume  that $\varphi$ is psh. 
By Theorem \ref{thm:1} the function $\varphi$ is strongly upper
semicontinuous and the restriction $\varphi|_{U\cap N_\R}$ is  locally
integrable. 
By Proposition \ref{thm:2}, the function $\trop_\infty^*(\varphi) \coloneqq \varphi \circ \trop_\infty$ is psh on $V\coloneqq \trop_\infty^{-1}(U)$ and hence $dd^c [\trop_\infty^*(\varphi)]$ is a positive current on $V$. 
Note that
\begin{equation} \label{projection formula for psh}
\trop_{\infty,\ast}[\trop_\infty^{\ast}(\varphi)]=[\varphi]
\end{equation}
by Proposition \ref{projection formula trop current cont function}. 
Using \ref{comparison with complex forms} and \eqref{projection
  formula for psh}, we conclude that  
\begin{equation} \label{projection formula combined with d'd''}
d'd''[\varphi]= \trop_{\infty,\ast}dd^c[\trop_\infty^{\ast}(\varphi)]
\end{equation}
is a positive current on $U$.

Now we consider the case of a general fan $\Sigma$.
Assume that $\varphi$ is psh. 
We choose a smooth subdivison $\Sigma'$ of $\Sigma$ in $N_\R$.
There is a unique proper surjective morphism of tropical toric varieties $g\colon N_{\Sigma'}\to N_\Sigma$ which extends the identity on $N_\R$.
By Proposition \ref{functoriality of psh} the function $\varphi'\coloneqq \varphi\circ r$ is  psh on $r^{-1}(U)$. 
As $\Sigma'$ is smooth, $\varphi'$ satisfies conditions \ref{psh and positivity of Lagerberg current-i} and \ref{psh and positivity of Lagerberg current-ii}.
Then $\varphi$ is strongly usc as well and $\varphi|_{U\cap N_\R}=\varphi'|_{U\cap N_\R}$ is locally integrable.
From $[\varphi]=r_*[\varphi']$ we conclude that $d'd''[\varphi]=r_*d'd''[\varphi']$ is positive as well.

For the converse implication, we need only that $\varphi$ satisfies  \ref{psh and positivity of Lagerberg current-i} and the weaker condition
\begin{enumerate}
\item[\textit{(ii')}] 
\textit{\label{psh and positivity of Lagerberg current-iii} $\varphi|_{U\cap N_\R}$ is  locally integrable and $d'd''[\varphi]$ is a positive Lagerberg current on $U\cap N_\R$.}
\end{enumerate}	  
As in the proof of \cite[Proposition 2.5]{lagerberg-2012},  
we construct from $\varphi|_{U\cap N_\R}$ a sequence $(\varphi_\epsilon)_\epsilon$
of convex smooth functions on $U\cap N_\R$ such that $\varphi_\epsilon$
converges weakly to a convex function $g$ on $U\cap N_\R$.
From $[\varphi|_{U\cap N_\R}]=[g]$, we get that $\varphi|_{U \cap N_\R}=g$ outside of a null set.
As $\varphi$ is strongly usc, this shows convexity of $\varphi|_{U \cap N_\R}=g$. 
Theorem \ref{thm:1} implies that $\varphi$ is psh.
\end{proof}

\subsection{Regularisation}
{Recall that $\Sigma$ is a fan and $U$ is an open subset of the partial compactification $N_\Sigma$.}
For a face $\tau$ of $\sigma \in \Sigma$, the canonical map $N(\tau)\to N(\sigma)$ is denoted by $\pi_{\sigma,\tau}$.
We say that a function $\varphi\colon U \to \R$ is \emph{constant towards the boundary} if  for each
$\sigma \in \Sigma $ and each $p\in U_{\sigma }\coloneqq U\cap
N(\sigma )$ there is a neighborhood $V$ of $p$ such that  for all 
$\tau\prec \sigma$ we have 
\begin{align} \label{compatibility along the boundary}
\text{$V_\tau = (\pi_{\sigma ,\tau }|_{V_\tau})^{-1}(V_\sigma)$ and $\varphi|_{V_\tau}=(\pi_{\sigma ,\tau }|_{V_\tau})^*   (\varphi|_{V_{\sigma }})$.} 
\end{align}
By definition, smooth functions on $U$ are constant towards the boundary.

We recall that a smoothing kernel in $N_\R$ is a non-negative smooth
function $\eta\colon N_\R \to \R$ with compact support and
$\int_{N_\R} \eta(y) dy=1$.  
Here $dy$ denotes the Haar measure on $N_\R$ such that the lattice $N$
has covolume one.

\begin{lem} \label{existence of smoothing kernels}
Let $\varphi\colon U \to \R$ be a continuous function which is constant towards the boundary and let $U'$ be a relatively compact open subset of $U$.
Then there exists a smoothing kernel $\eta$ with compact support in $N_\R$ such that the convolution 
	 \begin{align*}
      \varphi\star \eta\colon U'\cap N_\R\longrightarrow \R,\,\,
      \varphi \star \eta(x) = \int_{N_\R } \varphi(x-y) \eta(y) dy
    \end{align*}
is defined  and  extends uniquely to a smooth function $\varphi\star \eta \colon U'\to \R$.
\end{lem}

We call such an $\eta$ a \emph{smoothing kernel for $\varphi$ and $U'$}.

\begin{proof}
	Since $\varphi$ is constant towards the boundary, we can pick for any $p \in U$ an open neighbourhood $V(p)$ such that $\varphi$ satisfies \eqref{compatibility along the boundary} for $V = V(p)$. 
	We choose an open neighbourhood $V'(p)$ of $p$ which is relatively compact in $V(p)$ and hence there is a relatively compact convex subset $W(p)\subset N_\R$, symmetric with
	respect to zero and such that
	\begin{equation}\label{eq:36}
	V'(p)\cap N_{\R}+W(p) \subset V(p)\cap N_{\R}.
	\end{equation}
	Since the closure of $U'$ is compact in $U$, it can be covered
        by open  subsets $V'(p)$ with $p$ ranging over a finite $I
        \subset U$. Then $W \coloneqq \bigcap_{p \in I} W(p)$ is a
        relatively compact convex and symmetric subset of $N_\R$ with
        $U'\cap N_{\R}+W \subset U\cap N_{\R}$ by \eqref{eq:36}. For a
        smoothing kernel $\eta$ with compact support in $W$, we
        conclude   that the convolution $\varphi\star
	\eta$  is well defined on $U'\cap N_\R$.
    Since $\eta$ is smooth, it is clear that  $\varphi\ast
	\eta$ is smooth on $U'\cap N_{\R}$.
	
We have to check that $\varphi \ast \eta$ can be extended to a smooth function on $U'$.
We can check this locally. 
Hence it is enough to check for any $p \in I$ that  $\varphi \star \eta$ extends from $V'(p) \cap N_\R$ to a smooth function on $V'(p)$.
There is a unique $\sigma \in \Sigma$ with $p \in N(\sigma)$.
Let $\pi_\sigma\colon N_\R \to N_\sigma$ be the quotient map. 
For $x, x' \in V'(p) \cap N_\R$ with $\pi_\sigma(x)=\pi_\sigma(x'$) and any $y \in W(p)$, we have $\pi_\sigma(x-y)=\pi_\sigma(x'-y)$. Since $\varphi$ satisfies \eqref{compatibility along the boundary} on $V(p)$, we get $\varphi(x-y)=\varphi(x'-y)$ and hence 	
\[
\varphi \star \eta(x) = \int_{N_\R } \varphi(x-y) \eta(y) dy = \int_{N_\R } \varphi(x'-y) \eta(y) dy = \varphi \star \eta(x').
\]
Since $V'(p)$ is dense, we deduce that $\varphi \star \eta$ has a unique extension from $V'(p) \cap N_\R$ to a function on $V'(p)$ satisfying \eqref{compatibility along the boundary}. 
By definition, this extension is smooth.
    
Uniqueness is clear by density of $U' \cap N_\R$ in $U'$ and continuity.
\end{proof}

\begin{lem} \label{monotone property}
Let $\varphi$  be {a} continuous real function on  $U$ 
which is constant towards the boundary.  
Let $\eta$ be a smoothing kernel for  $\varphi$ and {a} relatively compact open subset $U'$ of $U$.
If $\varphi$ is psh in $U$, then $\varphi\star \eta$ is psh in $U'$ and $\varphi\star
  \eta \ge \varphi$ in $U'$.
\end{lem}

\begin{proof}
If $\varphi$ is psh, then $\varphi|_{ U\cap N_{\R}}$ is convex. This readily implies that $\varphi\star \eta \ge \varphi$ and that $\varphi\star \eta $ is convex on $U'\cap N_{\R}$. By Lemma \ref{existence of smoothing kernels}, we know that $\varphi\star \eta$ is continuous on $U'$ and hence 
    Theorem \ref{thm:1} shows that it is psh.
\end{proof}

\begin{lem} \label{sequence of smoothing kernels}
Let $\varphi$ be a continuous real function on  $U$
and let $U'$ be a relatively compact open subset of $U$. 
Then there exists a sequence of smoothing kernels 
$(\eta_{k})_{k\ge 1}$  for $\varphi$ and $U'$
such that $(\varphi\star \eta_{k})_{k\ge 1}$ converges pointwise to
$\varphi$ {on $U'$}. If $\varphi$ is psh, then the convergence is monotone decreasing. 
\end{lem}

\begin{proof}
Let $\eta$ and $W$ be as in the proof of Lemma \ref{existence of smoothing kernels}.
For  $k\ge 1$, we consider {the smoothing kernels $\eta_{k}(x)=k^n\eta(kx)$}.
Since $W$ is convex and symmetric,
    the support of  $\eta_{k}$
    is also contained in $W$. On $N_\R$, the measures $\eta_{k}dx $
    converge weakly to the Dirac delta measure $\delta _{0}$ centered
    at zero. 
{Since $\varphi$ is constant towards the boundary,} $(\varphi\star \eta_{k})_{k\ge 1}$ converges pointwise to
    $\varphi$ on $U'$. If $\varphi$ is psh, then Lemma \ref{monotone property} shows that $\varphi\star \eta_{k}$ is psh and $\geq \varphi$, hence $\varphi\star \eta_{k}$ and $\varphi$ restrict to convex functions on $U'\cap N_{\R}$ and  we get
\begin{align*}
(\varphi\star\eta_{k})(x)&=\int_{\R^n } \varphi(x-y) k^n\eta(ky)dy 
= \int_{\R^n } \varphi\Bigl(x-\frac{y}{k}\Bigr) \eta(y) dy\\
&\le \int_{\R^n } \left(\frac{1}{k}\varphi(x)+\frac{k-1}{k}\varphi\Bigl(x-\frac{y}{k-1}\Bigr)\right)\eta(y) dy\\ 
&=\frac{1}{k}\varphi(x) +\frac{k-1}{k}(\varphi\star\eta_{k-1})(x) \le(\varphi\star \eta_{k-1})(x) 
\end{align*}
for $x \in U' \cap N_\R, k \geq 2$.   Hence the convergence is monotonically decreasing on $U'$.
\end{proof}

The next Theorem gives \emph{local regularization} of psh functions. 
In \S \ref{global-regularization}, we will discuss \emph{global regularization} of $\theta$-psh functions.

\begin{thm} 
\label{thm:3} 
Let $\Sigma$ be  a fan and let $U$ be an open subset of $N_\Sigma$. 
A function
$\varphi \colon U \rightarrow \Rinf $ is psh if and only if 
 $\varphi$ is locally a decreasing limit of smooth psh functions.
\end{thm}
\begin{proof}
By Proposition \ref{prop:1-3}, any decreasing limit of psh functions is psh. Conversely, we assume that $\varphi$ is psh. The claim is local and so we may assume $U$ connected. The claim is obvious for $\varphi \equiv -\infty$ and so we may assume that $\varphi$ is not identically equal to $-\infty$.
The difficulty is that our definition of a smooth function on $U$ means that it is a smooth function on $U \cap N_\R$ 
which is constant towards the boundary, 
therefore a function $\varphi$ with $\varphi\circ \trop$  smooth is not necessarily smooth on $U$.

By Proposition \ref{prop:1-3} the function $\varphi$ can be approximated monotonically {decreasing by the} finite psh functions $\varphi_M \coloneqq \max(\varphi, -M)$ for $M \in \N$.

Now we show that a finite function $\varphi$ is locally a decreasing limit of continuous psh functions which are constant towards the boundary. 
So let us check this in  $p\in  U \cap N(\sigma )$. Passing to an open neighbourhood, we may assume that $\Sigma $
contains a single maximal cone $\sigma $ {and that $U$ satisfies the first condition in \eqref{compatibility along the boundary} for $V(p)=U$.}
Fix $\varepsilon > 0$ and let $\varepsilon_\sigma \coloneqq \varepsilon$.
For each 
$\tau \in \Sigma$ with $\tau \prec \sigma$, we  fix $\varepsilon_\tau>0$
satisfying {$\varepsilon_{\tau_1} < \varepsilon_{\tau_2}$}
whenever $\tau_1 \prec \tau_2$. 
Since we are assuming that $\Sigma $ contains a single maximal
  cone, there is a unique continuous extension
 $\pi_\tau\colon N_\Sigma \to \overline{N(\tau)}$ of the quotient map
 $N_\R \to N(\tau)$.  
Using that $\pi_\tau(U)=U \cap N(\tau)$, we define
\begin{align*}
g \colon U \longrightarrow \R,\,\,\,
x \longmapsto \max_{\tau \prec \sigma} 
\varphi(\pi_{\tau}(x)) + \varepsilon_\tau.
\end{align*}

Since $\pi_\tau$ is an equivariant morphism of tropical toric varieties and $\varphi$ is psh, Proposition \ref{functoriality of psh} shows that $\varphi\circ \pi_\tau$ is psh and hence $g$ is psh by Proposition \ref{prop:1-3}.
{Using that $\varphi$ is decreasing towards the boundary, we have $|\varphi-g| \leq\varepsilon$.} 
{Since $\varphi$ is a finite function which is continuous by Theorem \ref{thm:1}, the choice of the family $\varepsilon_\tau$ shows that $g$ is constant towards the boundary.}   
Choosing
$\varepsilon _{k}$ and $\varepsilon _{\sigma ,k}$ as above
 converging monotonically to zero for $k \to \infty$, we obtain a sequence
of continuous psh functions which are constant towards the boundary and  converge monotonically decreasing to $\varphi$. 

Now assume that $\varphi$ is a finite psh function which is constant towards the boundary. We prove the claim for $\varphi$ locally in $p \in U$. We pick a relatively compact open neighbourhood $U'$ of $p$ in $U$. Then
applying the sequence of smoothing kernels $\eta_{k}$ from Lemma \ref{sequence of smoothing kernels}
 we see
that $\varphi$ is a decreasing limit of smooth psh functions on $U'$.

Finally, applying a standard diagonal argument {based on Dini's theorem},
we deduce from the above steps that locally any psh function $\varphi$ 
is a limit of a decreasing sequence of smooth psh functions.
\end{proof}

\section{Bedford Taylor calculus on a partial compactification}
\label{sec:bedf-tayl-calc}
In this section, $\Sigma$ denotes a smooth fan in $N_\R$ for a free abelian group $N$ of rank $n$.
The goal  is to develop a Bedford--Taylor calculus on
the partial compactification $N_{\Sigma }$ using the Bedford Taylor
calculus on the complex manifold $X_\Sigma^{\an}$. 
{At the end, functoriality will help us to construct Bedford--Taylor products for psh functions also for non-smooth fans.}

\subsection{Locally bounded case} \label{locally bounded case}

Let $U$ be an open subset of $N_{\Sigma }$. 
{Let $D_{\cl,+}^{p,p}(U)$ denote the space of positive Lagerberg currents on $U$ of bidegree $(p,p)$ which are 
closed with respect to the differentials $d'$ and $d''$.}
Let $u_{1},\ldots,u_{q}$ be smooth psh functions on $U$ and let $T\in D^{p,p}(U)_{\cl,+}$ with $p+q\le n$. 
{From the calculus of smooth Lagerberg forms and currents on $U$ (see \S \ref{section: tropical situation} and \cite[Section 3]{burgos-gubler-jell-kuennemann1}),} we get {products}
\begin{align}\label{eq:38}
u_{1}d'd''u_{2}\wedge\ldots \wedge d'd''u_{q}\wedge T&\in D^{p+q-1,p+q-1}(U),\\
d'd''u_{1}\wedge\ldots \wedge d'd''u_{q}\wedge T&\in D_{\cl,+}^{p+q,p+q}(U).\label{eq:39}
\end{align}
The second type of currents is called \emph{Monge-Amp\`ere currents}. 
We want to extend the construction of the products  {\eqref{eq:38} and \eqref{eq:39}} to locally bounded psh functions. 
By Theorem \ref{thm:1}, a psh function is locally bounded if and only if it is finite continuous.

\begin{rem} \label{compatibility with complex MA}
For smooth psh functions $u_1,\ldots,u_q$ on $U$ and a current $T$ as above, we write $v_{j}=\trop_\infty^{\ast}(u_{j})$ for $j=1,\ldots,q$.
By Proposition \ref{thm:2}, the functions $v_{j}$ are psh on $V= \trop_\infty^{-1}(U)$. 
By the main theorem in \cite{burgos-gubler-jell-kuennemann1} recalled in \ref{comparison with complex forms}, there is a unique {closed positive current $S \in D^{p,p}(V)^{\SS,F}$ with $\trop_{\infty,\ast}(S)=T$.}
The projection formula \cite[Proposition 5.1.7]{burgos-gubler-jell-kuennemann1} and \eqref{eq:3} imply
{\begin{align*}
u_{1}d'd''u_{2}\wedge\ldots \wedge d'd''u_{q}\wedge T
&= \trop_{\infty,\ast} (v_{1} dd^c v_{2} \wedge \ldots \wedge dd^c v_{q}\wedge S),\\
d'd''u_{1}\wedge\ldots \wedge d'd'' u_{q}\wedge T
&=\trop_{\infty,\ast}( dd^c v_{1} \wedge \ldots \wedge dd^c v_{q}\wedge S).
\end{align*}}
\end{rem}

\begin{thm}\label{thm:7}
{For any open subset $U$ of $N_\Sigma$,  locally bounded  psh functions $u_{1},\ldots,u_{q}$ on $U$ and $T\in D^{p,p}(U)_{\cl,+}$, there are unique Lagerberg currents} 
\begin{align*}
u_{1}d'd''u_{2}\wedge \ldots \wedge d'd''u_{q} \wedge T & \in D^{p+q-1,p+q-1}(U), \\
d'd''u_{1} \wedge \ldots \wedge d'd''u_{q} \wedge T & \in D_{\cl,+}^{p+q,p+q}(U)
\end{align*}
such that these products are given locally in $U$ and satisfy the following two properties:
\begin{enumerate}
\item \label{item:16} 
If $u_1, \ldots, u_q$ are smooth, then the
product currents agree with \eqref{eq:38} and
\eqref{eq:39}.
\item \label{item:17}  
If there are psh functions  $\{u_{j,k}\}_{k\ge 0}$ converging locally
uniformly to $u_{j}$ for  
$j=1,\ldots,q$ and if there are closed positive
Lagerberg currents  $\{T_{k}\}_{k\ge 0}$ converging weakly to $T$,  
then we have the weak convergence of Lagerberg currents
\begin{align*}
u_{1,k} d'd''u_{2,k} \wedge \ldots \wedge d'd''u_{q,k} \wedge T_{k} 
&\xrightarrow{k\to \infty} u_{1} d'd''u_{2} \wedge\ldots \wedge d'd''u_{q} \wedge T, \\
d'd''u_{1,k} \wedge \ldots \wedge d'd''u_{q,k} \wedge T_{k}& \xrightarrow{k\to \infty} 
d'd''u_{1} \wedge \ldots \wedge d'd''u_{q} \wedge T.
\end{align*}
\end{enumerate}
\end{thm}

\begin{proof}
For $j=1,\ldots,q$, Proposition \ref{thm:2} shows that  $v_{j}\coloneqq \trop_\infty^{\ast}(u_{j})$ is psh on $V \coloneqq \trop_\infty^{-1}(U)$. 
As seen in Remark \ref{compatibility with complex MA},  there is a unique {closed positive $S \in D^{p,p}(V)^{\SS,F}$ with $\trop_{\infty,\ast}(S)=T$.}
In  \cite[\S III.3]{demailly_agbook_2012}
there is an inductive definition of 
{$v_{1} dd^c v_{2} \wedge \ldots \wedge dd^c v_{q} \wedge S
\in D^{p+q-1,p+q-1}(V)$ and 
$dd^c v_{1} \wedge \ldots \wedge dd^c v_{q} \wedge S 
 \in D_{\cl,+}^{p+q,p+q}(V)$.}
Then we define
\begin{align*}
u_{1}d'd''u_{2}\wedge\ldots \wedge d'd''u_{q}\wedge T
& \coloneqq \trop_{\infty,\ast}(v_{1} dd^c v_{2} \wedge \ldots \wedge dd^c v_{q}\wedge S)   \\
d'd''u_{1}\wedge\ldots \wedge d'd'' u_{q}\wedge T
& \coloneqq \trop_{\infty,\ast}( dd^c v_{1} \wedge \ldots dd^c
 v_{q} \wedge S).
\end{align*}
This definition is local in $U$ as the complex construction is local in $V$. By Remark \ref{compatibility with complex MA}, we get \ref{item:16}. 
By Theorem \ref{thm:1}, any 
psh function on $U$ is continuous
and hence \ref{item:17} follows from the corresponding fact on complex manifolds given in \cite[Corollary III.3.6]{demailly_agbook_2012}.
It follows from the regularization in Theorem \ref{thm:3} that these
two properties and locality in $U$ characterize the product currents
uniquely. 
\end{proof}

\subsection{The unbounded case} \label{unbounded case} 

We will now consider a version of the previous theorem for unbounded
psh-functions. 
Similarly as in \cite[\S 3.4]{demailly_agbook_2012}, we will replace uniform convergence by pointwise monotone convergence 
and the sequence $(T_k)$ by one current $T$. 
Additionally, the loci of unboundedness of the psh functions and $\supp(T)$ have to
intersect properly.

We still assume that $\Sigma$ is a smooth fan in $N_\R$ and that $U$ is an open subset of $N_\Sigma$. 
A \emph{stratum} of $U$ is a connected component of $U \cap N(\sigma)$ for some $\sigma \in \Sigma$. 
A \emph{strata subset} of $U$ is a union of strata of $U$.
Given a psh function $u$ on  $U$, we will write
\begin{displaymath}
 L(u) \coloneqq \{x\in U\,\mid\, u(x)=-\infty\}.
\end{displaymath}
{By Theorem \ref{thm:1} and Lemma \ref{lemm:2}, the set $L(u)$ is a closed strata subset of $U$.}

In the following definition, the closed strata subsets $D_1, \ldots, D_q$
should morally be viewed as supports of divisors.

\begin{definition} \label{proper intersection}
Let $B$ be a closed strata subset of $U$  and let $D_1, \ldots, D_q$ be closed strata subsets of $U$ of codimension at least one.
{For $d \in \N$}, we say that \emph{$D_1, \ldots, D_q$ intersect $B$ 
{$d$-properly}} if for any subset $I \subset  \{1,\ldots,q\}$
	we have
	\begin{equation} \label{proper intersection eq}
	\dim\left( B\cap \bigcap_{j\in I}D_{j}\right) \le d-|I|
	\end{equation}
\end{definition}

Note that for $I=\emptyset$, 
this gives that the dimension of $B$ is bounded by $d$.

\begin{thm}\label{thm:8}
{Let $D_1, \ldots , D_q$ be closed strata subsets of $U$ of codimension at least one which intersect the closed strata subset $B$ of $U$ $(n-p)$-properly for some $p \in \N$.}  
For $j\in\{1,\ldots,q\}$, let $u_{j}$ be a psh function on $U$ with $L(u_{j}) \subset D_{j}$ and let {$T \in D_{\cl,+}^{p,p}(U)$} with $\supp(T) \subset B$. 
Then there are product currents
\begin{align*}
u_{1}d'd''u_{2}\wedge\ldots \wedge d'd''u_{q}\wedge T
& \in D^{p+q-1,p+q-1}(U),  \\
d'd''u_{1} \wedge \ldots \wedge d'd''u_{q} \wedge T 
& \in D_{\cl,+}^{p+q,p+q}(U)
\end{align*}
that agree with the ones in Theorem \ref{thm:7} when the functions $u_{j}$ are locally bounded. 
The formation of the product currents is local in $U$. 
If there are decreasing sequences $\{u_{j,k}\}_{k\ge 0}$ of psh functions converging pointwise   to $u_{j}$ for $j=1,\ldots,q$, then we have 
\begin{align*}
u_{1,k} d'd''u_{2,k} \wedge \ldots \wedge d'd''u_{q,k} \wedge T
&\xrightarrow{k\to \infty} u_{1} d'd''u_{2} \wedge \ldots \wedge d'd''u_{q} \wedge T,  \\
d'd''u_{1,k} \wedge \ldots \wedge d'd''u_{q,k} \wedge T
& \xrightarrow{k\to \infty} d'd''u_{1} \wedge \ldots \wedge d'd''u_{q} \wedge T
\end{align*}
as weak convergence of currents. 
These conditions determine the product currents  uniquely. 
\end{thm}

\begin{proof}
Let $V$, $v_{j}$ and $S$ be as in the proof of Theorem \ref{thm:7}. 
For every $I\subset \{1,\ldots,q\}$, the set
\begin{math}
B\cap \bigcap_{j\in I}D_{j}
\end{math}
is a closed stata subset of $U$. 
If $N(\sigma)$ is a stratum of $N_{\Sigma }$, then
\begin{displaymath}
\dim_{\R}(\trop^{-1}(N(\sigma)))=2\dim_{\C}(\trop^{-1}(N(\sigma)))=2\dim_{\R}(N(\sigma))
\end{displaymath}
and hence our Assumption \eqref{proper intersection eq} yields
\begin{displaymath}
\dim_{\R} \trop^{-1} \left( B\cap \bigcap_{j\in I} D_{j} \right) \le 2n-2p-2|I|<2n-2p-2|I|+1. 
\end{displaymath}
Therefore, the functions $v_{j}$ and the current {$S$} satisfy the hypotheses of 
\cite[Theorem III.4.5]{demailly_agbook_2012}, {where the role of $p$ and $n-p$ is interchanged}. 
It follows that the currents 
\begin{align*}
v_{1}dd^c v_{2}\wedge \ldots \wedge dd^c v_{q}\wedge S
&\in D^{p+q-1,p+q-1}(V),  \\ 
dd^c v_{1}\wedge\ldots \wedge dd^c v_{q}\wedge S
&\in D_{\cl,+}^{p+q,p+q}(V) 
\end{align*}
are well defined. 
As in Theorem \ref{thm:7}, we define the product currents on $U$ as the
direct image of the corresponding currents in $V$ with respect to $\trop \colon V \to U$. 
The continuity result is a consequence of the weak continuity of $\trop_*$ and 
\cite[Theorem III.4.5]{demailly_agbook_2012}. 
Uniqueness of the product currents 
follows again from  the regularization in Theorem \ref{thm:3} and
  from locality in $U$.  
Since the current $dd^c v_{1}\wedge\ldots \wedge dd^c v_{q}\wedge S$ is closed and positive, the same is true 
for $d'd''u_{1} \wedge \ldots \wedge d'd''u_{q} \wedge T$ by \cite[Proposition 5.1.13]{burgos-gubler-jell-kuennemann1}.
\end{proof}

\begin{rem} \label{comparison of complex BT product and tropical BT product}
	In the above proof, we have seen that the tropical Bedford--Taylor products  of the psh functions $u_1,\ldots, u_q$ and  $T \in D_{\cl,+}^{p,p}(U)$ satisfying the hypotheses in Theorem \ref{thm:8}  are compatible with the Bedford--Taylor products of the psh functions $v_j \coloneqq u_j \circ \trop_\infty$ on the complex manifold $V = \trop_\infty^{-1}(U)$ and the unique closed positive $S \in D^{p,p}(V)^{\SS,F}$ with $\trop_{\infty,\ast}(S)=T$  in the sense that we have
	\begin{align*}
	u_{1}d'd''u_{2}\wedge\ldots \wedge d'd''u_{q}\wedge T
	& = \trop_{\infty,\ast}(v_{1} dd^c v_{2} \wedge \ldots \wedge dd^c v_{q}\wedge S)\\
	d'd''u_{1}\wedge\ldots \wedge d'd'' u_{q}\wedge T
	& = \trop_{\infty,\ast}( dd^c v_{1} \wedge \ldots dd^c
	v_{q} \wedge S).
	\end{align*}  
\end{rem}

\subsection{Functoriality}

Let $N,N'$ be free abelian groups of finite rank $n$ and $n'$
respectively.
As in {Definition \ref{new morphism of toric tropical varieties}}, we consider a homomorphism $L\colon N' \to N$  and a proper  $L$-equivariant morphism  $E\colon N'_{\Sigma'} \to N_\Sigma$ of tropical toric varieties for fans $\Sigma$ and $\Sigma'$ in $N_\R$ and $N_\R'$, respectively. 
Let $U$ be an open subset of $N_\Sigma$ and $U' \coloneqq E^{-1}(U)$.

\begin{lem} \label{projection formula smooth}
For $T \in D^{r,s}(U')$ and $\alpha \in \AS^{p,q}(U)$ 
we have
\begin{align*}
\alpha \wedge E_* T  = E_* ( E^* \alpha \wedge T) \in D^{r+n-n'+p, s+n-n'+q}(U).
\end{align*}
\end{lem}
\begin{proof}
This follows directly from testing against a form $\beta \in
A^{n'-r-p,n'-s-q}_c(U)$.
\end{proof}

By Proposition \ref{functoriality of psh}, the pull-back of a psh function on $U$ with respect to $E$ is psh.

\begin{prop} \label{projection formula for locally bounded psh}
In the above setting,  assume that $\Sigma,\Sigma'$ are smooth fans. Let $T$ be a closed positive current on $U'$ and 
let $u_1,\ldots,u_q$ be locally bounded psh-functions on $U$. 
Then 
\begin{align*}
E_* \left(E^*(u_1) d'd'' E^*(u_2) \wedge \ldots \wedge E^*(u_q) \wedge T \right) = 
u_1 d'd'' u_2 \wedge \ldots \wedge u_q \wedge E_*T
\end{align*}
\end{prop}
\begin{proof}
When the $u_i$ are smooth, this follows from Lemma \ref{projection formula smooth}. 
In general, we use Theorem \ref{thm:3}  to approximate the $u_i$ 
{locally by  decreasing sequences of smooth psh functions}.
Then we use locality and continuity of the product in Theorem \ref{thm:7} and of $E_*$. 
\end{proof}

\begin{rem} \label{projection formula for unbounded psh}
In the same setting, the projection formula in Proposition \ref{projection formula for locally bounded psh} holds more generally for unbounded psh-functions $u_1,\ldots,u_q$ on $U$ if we assume $L(u_j)	\subset D_j$ and $\supp(T) \subset B$ for strata subsets $D_j$ of codimension at least one in $U$ and a strata subset $B$ of $U'$ such that $E^{-1}(D_1), \ldots, E^{-1}(D_q)$ intersect $B$ $(m-p)$-properly.

Indeed, then $E_*(T) \in D^{n-m+p,n-m+p}(U)$ and $D_1, \ldots, D_q$ intersect the closed stata subset $E(B)$ $(m-p)$-properly. The projection formula follows as above by using Theorem \ref{thm:8}.
\end{rem}

\begin{rem}  \label{Bedford-Taylor for non-smooth}  
We can use functoriality to define Bedford--Taylor products also in the case of a non-smooth fan $\Sigma$ in $N_\R$. Let $u_1, \ldots, u_q$ be locally bounded psh functions on the open subset $U$ of $N_\Sigma$. By toric resolution of singularities, there is a smooth fan $\Sigma'$ in $N_\R$ refining $\Sigma$. The associated morphism $E\colon N_{\Sigma'}\to N_\Sigma$ is a proper $\id_{N}$-equivariant morphism of tropical toric varieties and we define the product currents by
\begin{align*}
u_{1}d'd''u_{2}\wedge\ldots \wedge d'd''u_{q}  \coloneqq  E_* \left(E^*(u_1) d'd'' E^*(u_2) \wedge \ldots \wedge E^*(u_q)  \right)
& \in D^{q-1,q-1}(U),  \\
d'd''u_{1} \wedge \ldots \wedge d'd''u_{q} \coloneqq
E_* \left(E^*(u_1) d'd'' E^*(u_2) \wedge \ldots \wedge E^*(u_q) \right) & \in D_{\cl,+}^{q,q}(U).
\end{align*}
It follows from Proposition \ref{projection formula for locally bounded psh} that the definition of the product currents does not depend on the choice of $\Sigma'$.  Obviously, the above product currents are still local in $U$ and are continuous along monotonically decreasing sequences of psh functions. Moreover,  the projection formula in Proposition \ref{projection formula for locally bounded psh} holds also for not necessarily smooth fans.
\end{rem}

\section{Semipositive metrics and $\theta$-psh functions} 
\label{section: semipositive metrics and theta-psh functions}

We recall {properties of} semipositive metrics on line bundles over complex manifolds or non-archimedean analytic spaces and {discuss} the equivalent concept of $\theta$-psh functions for a closed first Chern current $\theta$ of the  line bundle. 
We  also introduce $\theta$-psh functions on tropical toric varieties. 
For toric varieties, {tropical $\theta$-psh functions} serve as a link between invariant $\theta$-psh functions in the complex and the non-archimedean situation.
We prove correspondence theorems which will later be applied to solve Monge--Amp\`ere equations.

\subsection{The complex setting} \label{subsection complex setting}

Let $L$ be a holomorphic line bundle on a complex manifold $X$. 
We first introduce singular psh metrics on $L$ and $\theta$-psh functions in this general setting.
Then we restrict to the toric setting.

\begin{art} \label{recall metric}
A {\it continuous metric $\metr$ on $L$} is given by 
 continuous functions 
$-\log\|s\|\colon U \to \R$ for all frames (i.e.~nowhere vanishing sections) $s \in H^0(U,L)$ on open subsets $U$ of $X$ such that for frames $s$ on $U$ {and $s'$ on $U'$} we have $\log \|s\| - \log \|s'\|= \log |s/s'|$ {on $U\cap U'$}.
	
A {\it singular metric} of $L$ is defined similarly, 
but without assuming continuity and with allowing the function $-\log\|s\|$ to take the value $-\infty$ {as well}. 
\end{art} 

\begin{definition} \label{psh metric for complex} 
A singular metric $\metr$ of $L$ is called {\it psh} 
if for all local frames $s$ of $L$ on any open subset $U$ of $X$, the function $-\log \|s\|{\colon U\to \Rinf}$ is psh. 
\end{definition}

\begin{rem} \label{Lelong's characterization}
If $X$ is connected, then a singular metric $\metr$ of $L$ is psh if and only if 
 either the function $-\log \|s\|$ is  identically $-\infty$ 
or  the function $-\log \|s\|$ is strongly usc as defined in \eqref{def:2'}, locally integrable and the first Chern current  
$c_1(L,\metr)\coloneqq {dd^c}[-\log \|s\|]$ is positive for all
{local} frames $s$ {of $L$} {\cite[Th\'eor\`eme II.3]{lelong68:psh}.} 
Note that a smooth metric is psh if and only if the first Chern form
is a positive form.
Some authors would call the latter form semipositive, 
but we use here the positivity notions of forms and currents given in 
\cite[\S III.1.A, \S III.1.B]{demailly_agbook_2012} following also our conventions from 
Section \ref{sec:posit-real-compl}. 
\end{rem}

\begin{definition} \label{theta psh}
We fix a continuous reference metric $\metr_0$ on $L$ and we set
$\theta \coloneqq c_1(L,\metr_0)$ for the associated first Chern
current. 
A {\it $\theta$-psh function} $\varphi$ on $X$ is a function $\varphi \colon X \to \Rinf$ such that for every connected open subset $U$ of $X$, we have either $\varphi|_U \equiv -\infty$ or $\varphi|_U$ is strongly usc, locally integrable and $dd^c [\varphi|_U]+ \theta$ is a positive current on $U$.
\end{definition}

{Observe that if $\metr_0$ is a smooth metric, then $\theta$ is a smooth
$(1,1)$-form on $X$.}

\begin{rem} \label{psh vs theta psh in complex}
Note that every singular metric $\metr$ on $L$ induces a  function 
\[
\varphi {\coloneqq - \log( \|\phantom{s}\|/\|\phantom{s}\|_0)}\coloneqq - \log( \|s\|/\|s\|_0)\colon X \to \Rinf
\]
independent of the choice of a local frame $s$ of $L$. 
Obviously, the singular metric $\metr$ is psh if and only if the
function $\varphi$ is $\theta$-psh. 
\end{rem}

\begin{rem} \label{Zhang semipositive metrics for complex}
For a line bundle $L$ on a complex {smooth} proper variety $Y$, 
there is another semipositivity notion for metrics on $L$ introduced
by Zhang \cite{zhang-95}: 
Let $\Yan$ be the complex analytification of $Y$.
A metric $\metr$ on $\Lan$ is called {\it semipositive}, if $\metr$ is
a uniform limit of smooth psh metrics  
$\metr_k$ on $\Lan$, i.e.~the Chern forms  $c_1(h^*\Lan,h^*\metr_k)$
are assumed to be positive forms.
Obviously, a semipositive metric is continuous.
It is easy to see 
that for every  semipositive metric $\metr$ on $\Lan$, the first Chern current $c_1(L,\metr)$ is a positive current on $\Xan$ and 
hence $\metr$ is a psh metric as in Definition \ref{psh metric for complex}. 
The converse requires existence of global regularization and 
holds for $L$ ample \cite[Theorem 4.6.1]{maillot-thesis}. 	
\end{rem}

\begin{art} \label{toric semipositive metrics for complex}
{Let $N$ be a lattice with dual lattice $M$.
Let $\Sigma$ be a smooth fan in $N_\R$ with associated smooth complex toric variety $X_\Sigma$. 
Recall from \cite[Definition 3.3.4]{bps-asterisque}, that a \emph{toric line bundle $L$ on $X_\Sigma$} is given by a line bundle $L$ on $X_\Sigma$ together with a fixed trivialization of the fiber of $L$ over the origin of the generic torus $\T$ of $X_\Sigma$.
A meromorphic section $s$ of a toric line bundle is called \emph{toric} if 
it is regular and nowhere vanishing on $\T$ and the fixed trivialization is induced by  $s$.
The choice of a toric section of a toric line bundle corresponds
precisely to the choice of a $\T$-linearization of the toric line bundle
\cite[Remark 3.3.6]{bps-asterisque}.}

{Let $s$ be a {toric meromorphic section} of a {toric line bundle} $L$ on $X_\Sigma$. 
Associated with $L$ and $s$ there is a piecewise linear function $\Psi=\Psi(L,s)\colon \Sigma\to \R$, whose construction we now recall: 
The invariant Cartier divisor of $s$ is given on the affine open subset $U_\sigma=\Spec\, \C[\sigma^\vee\cap M]$ associated 
to the cone $\sigma \in \Sigma$ by a character $\chi^{-m}$ and 
then $\Psi$ is given on $\sigma$ by the corresponding linear form  $m \in M$.} 

{For a continuous metric $\metr$ on $\Lan$, we call $g \coloneqq -\log \|s\|$ the \emph{associated Green function for $\div(s)$}. 
A metric on $\Lan$ is called \emph{toric} if it is invariant under pull-back with respect to the action of the compact torus $\SS$ in $\torusan$.}

If $\Sigma$ is complete, then a result of Burgos, Philippon and Sombra 
\cite[Theorem 4.8.1]{bps-asterisque} shows
that {for a fixed toric meromorphic section $s$ of $L$} the invariant continuous semipositive toric metrics on $\Lan$ 
are in bijection to the concave functions $\psi$ on $N_\R$ 
with $\psi = \Psi + O(1)$. 
The correspondence is given by 
\begin{equation}\label{archimedean-toric-corr}
\psi \circ \trop = \log \|s \|.
\end{equation}
We will generalize this {correspondence} below. 
The \emph{canonical metric} on a {nef} toric line bundle $L$
on the proper toric variety $X_\Sigma$ corresponds to the choice $\psi =\Psi$. 
This canonical metric {does not depend on the choice of $s$ and} is continuous but  not necessarily smooth. 
\end{art}

\subsection{The Lagerberg setting} \label{subsection: Lagerberg setting}

Let $\Sigma$ be a fan  in $N_\R$ for a lattice $N$ of rank $n$. 
Our goal is to transfer the notions from the previous subsection to an open subset $U$ of the partial compactification $N_\Sigma$. 
As usual, we set $M \coloneqq \Hom_\Z(N,\Z)$.

In the tropical world, there is no perfect analogue of line bundles
and so we prefer to establish the analogue of $\theta$-psh
functions. According to \ref{toric semipositive metrics for complex}, 
a good way to replace the holomorphic line bundle is to
fix a function $\Psi\colon |\Sigma| \to \R$ 
which is piecewise linear with respect to the fan $\Sigma$ 
which means that $\Psi$ is linear on each cone $\sigma$ of $\Sigma$. 
We also require that the slope of $\Psi|_\sigma$ is integral, 
i.e. $\Psi|_\sigma = m_\sigma$ for some $m_\sigma \in M$. 
This makes sure that there is an associated line bundle $L$ on the complex toric variety $X_\Sigma$ 
by the reverse of the construction in  \ref{toric semipositive metrics for complex}. 
This will be important for comparing to semipositive metrics in the complex setting. 
Note that $\Psi$ might be seen as a tropical Cartier divisor on $N_\Sigma$ and 
linear equivalence is given by adding linear functions from $M$
{\cite{allermann-rau-2010}}.

The role of continuous metrics is played {in the following} by Green functions. 

\begin{definition} \label{tropical Green function} 
A \emph{{continuous Green function} for $\Psi$} is a function $g\colon U \cap N_\R \to \R$ such that for all $\sigma \in \Sigma$ and all $m_\sigma\in M$ with $\Psi|_\sigma=m_\sigma|_\sigma$, there exists an open neighbourhood $\Omega_\sigma$ of $U\cap N(\sigma)$ in $U$ and a continuous function $h_\sigma\colon \Omega_\sigma\to\R$ which extends the function $g+m_\sigma|_{\Omega_\sigma\cap N_\R}$. 	
We say that  $g$ is a \emph{smooth Green function for $\Psi$} if we may choose $h_\sigma$ smooth for all $\sigma \in \Sigma$.

A \emph{{continuous singular Green function} for $\Psi$} is
  defined as above, but 
  allowing the continuous function $h_\sigma$ to take the value
  $-\infty$. {Note that the function $g$ is still assumed to have
finite values on $U \cap N_\R$. Therefore $h_{\sigma }$ can only take the value
$-\infty$  on $U\setminus N_{\R}$.}
\end{definition}

\begin{rem} \label{smooth and canonical tropical Green functions}
A partition of unity argument shows that there exists a smooth Green function for $\Psi$. 
In case of a complete fan, a canonical {continuous} Green function is
given by $g \coloneqq -\Psi|_{U\cap N_\R}$.
\end{rem}

We use in the following that a
	real valued continuous or   
	more generally locally integrable function $f$ on $U \cap N_\R$ induces a 
	Lagerberg current $[f] \in D^{0,0}(U)$, see Remark \ref{locally integrable gives current}.

\begin{definition} \label{Lagerberg theta}
Let $g$ be a continuous singular Green function for $\Psi$. 
For a cone $\sigma \in \Sigma$, we pick an open subset $\Omega_\sigma$ and
$m_\sigma \in M$ as in Definition \ref{tropical Green function} and
define \emph{the first Chern current} $c_1(\Psi, g)$ on
$\Omega_\sigma$ by  $c_1(\Psi, g) \coloneqq d'd''[g + m_\sigma]$.  
Since the open sets $\Omega_\sigma$ cover $U$ and since the first Chern currents agree on overlappings, we get a well-defined Lagerberg current $c_1(\Psi, g)\in D^{1,1}(U)$. 
\end{definition}

It is easy to see that 
$c_1(\Psi, g)$ is a $d'$- and $d''$-closed symmetric Lagerberg current in $D^{1,1}(U)$ and does not depend on the choice of the  linear functions $m_\sigma$.

{We fix a reference {continuous} Green function $g_0$ for $\Psi$ and we set $\theta \coloneqq c_1(\Psi, g_0){\in D^{1,1}(U)}$.}

\begin{definition} \label{theta psh for Lagerberg}
A function $\varphi \colon U \to \Rinf$ is called \emph{$\theta$-psh}
if it is strongly upper semicontinuous (see \eqref{def:2})  
and for any connected component $W$ of $U$, 
the function $\varphi$ either restricts to a
locally integrable function  $W \cap N_\R \to \R$  and 
$d'd''[\varphi] + \theta$ is a positive Lagerberg current on $W$ or
$\varphi$ is identically $-\infty$ on $W$. 
\end{definition}

The following characterization of $\theta $-psh functions uses 
the open sets $\Omega_\sigma$ and $m_\sigma \in M$ from
Definition \ref{tropical Green function} for the reference
{continuous} Green function $g_0$ instead of $g$.

\begin{prop}\label{lemm:4} 
A function $\varphi\colon U \to \Rinf$  is $\theta$-psh if and only if 
$\varphi+g_0+m_\sigma|_{\Omega_\sigma \cap N_\R}$ extends to a psh
function on $\Omega_\sigma$ for all $\sigma \in \Sigma$.
Hence any $\theta$-psh function is continuous.
\end{prop}

\begin{proof}   
We may assume that $U$ is connected and that  $\varphi$ is not
identically $-\infty$ on $U$. 
By Definition \ref{tropical Green function}, the function $g_0+ m_\sigma$ extends to a continuous function $h_\sigma\colon \Omega_\sigma \to \R$ for any $\sigma \in \Sigma$. 
On $\Omega_\sigma$, we have 
\begin{displaymath}
d'd''[\varphi+h_\sigma]=d'd''[\varphi+g_0+m_\sigma] = d'd''[\varphi]+\theta. 
\end{displaymath}	
Since the open subsets $\Omega_\sigma$ cover $U$, it follows from Theorem \ref{psh and positivity of Lagerberg current} that $\varphi$ is $\theta$-psh if and only if  $\varphi+h_\sigma$ is psh on $\Omega_\sigma$ for every $\sigma \in \Sigma$. 
Now Theorem \ref{thm:1} yields continuity. 
\end{proof}

  \begin{prop}\label{prop:4'}
    A function $\varphi\colon U \to \Rinf$ is $\theta $-psh if and
    only if the restriction of
    $\varphi + g_{0}$ to $U\cap N_{\R}$ is convex and for every point $x\in
            U\setminus N_{\R}$ we have
            \begin{equation*}
              {\varphi}(x)=\limsup_{\substack{y\in U\cap N_{\R}\\y\to x}}{\varphi}(y).
            \end{equation*} 
  \end{prop}

  \begin{proof}
    Follows from Theorem \ref{thm:1}, Theorem \ref{psh and positivity of Lagerberg
      current} and Proposition \ref{lemm:4}.
  \end{proof}

Next, we deal with functoriality of $\theta$-psh functions. 

\begin{rem} \label{functoriality and theta-psh}
Let $L\colon N' \to N$ be a homomorphism of free abelian groups of finite rank.
Let $E\colon N_{\Sigma'}' \to N_\Sigma$ be an $L$-equivariant morphism of tropical toric varieties as in {Definition \ref{new morphism of toric tropical varieties}}. 
Then $\Psi' \coloneqq \Psi \circ L$ is a function on $|\Sigma'|$ which is piecewise linear with respect to $\Sigma'$. 
There exists a unique cone $\sigma \in \Sigma$ such that $E(N_\R')\subset N(\sigma)$.  Adding to $\Psi$  a linear function in $M$, we may assume that $\Psi|_\sigma=0$. 
Then $g_0' \coloneqq g_0 \circ E$ is a {continuous} Green function for $\Psi'$ on $U' \coloneqq E^{-1}(U)$ and we set $\theta' \coloneqq c_1(\Psi',g_0')=E^*(\theta)$. 
If we use the characterization of psh functions in Proposition \ref{lemm:4}, we can conclude from Proposition \ref{functoriality of psh} the following  properties of functions $\varphi\colon U \to \Rinf$: 	
\begin{enumerate}
\item \label{theta-psh lifts}
If $\varphi$ is $\theta$-psh on $U$, then $\varphi \circ E$ is $\theta'$-psh on $U'$.
\item \label{theta-psh descends}
If $E$ is a surjective proper map and if $\varphi \circ E$ is $\theta'$-psh on $U'$, then $\varphi$ is $\theta$-psh on $U$.
\end{enumerate}
\end{rem}

\begin{prop} \label{descend of psh}
	Let $E:N_{\Sigma'} \to N_\Sigma$ be a proper surjective $L$-equivariant morphism of tropical toric varieties,  $\theta' \coloneqq E^*(\theta)$ and $U' \coloneqq E^{-1}(U)$ as above. If $\varphi':U' \to \Rinf$ is a $\theta'$-psh function, then there is a unique $\theta$-psh function $\varphi:U \to \Rinf$ with $\varphi' = \varphi \circ E$.
\end{prop}

\begin{proof}
Let  $x \in U$. Since $E$ is $L$-equivariant, the fiber $E^{-1}(x)$ is
isomorphic to a partial compactification $N_{\Sigma''}$. Using that
$E$ is proper, we conclude that the fiber is compact and hence
$\Sigma''$ is a complete fan. Let $\sigma \in \Sigma$ be the unique
cone with $x \in N(\sigma)$. It follows from Proposition \ref{lemm:4}
that $\varphi'+g_0\circ E + m_\sigma \circ E$ extends to a psh
function $\psi$ on neighbourhood of $E^{-1}(x)$.
By Proposition \ref{functoriality of psh}, the pull-back of
$\psi$ to $N_{\Sigma''}$ is psh and hence constant by
Proposition \ref{maximum principle}. We conclude that
$\varphi'$ is constant on each fiber and hence there is a
unique $\varphi:U \to \Rinf$ with $\varphi'=\varphi \circ
E$. By Remark \ref{functoriality and theta-psh} \ref{theta-psh
  descends}, the function $\varphi$ is $\theta$-psh. 
\end{proof}

\begin{rem} \label{toric resolution} 
If $\Sigma$ is not a smooth fan, then we can apply toric resolution of singularities \cite[Theorem 11.1.9]{cox-little-schenck} to get a smooth fan $\Sigma'$ 
{which is a subdivision of} $\Sigma$. 
The induced morphism $E\colon  N_{\Sigma'} \to N_\Sigma$ of tropical toric varieties is $\id_N$-equivariant and proper. 
By Remark \ref{functoriality and theta-psh},  the function $\varphi\colon U \to \Rinf$ is $\theta$-psh if and only if the function $\varphi' \coloneqq \varphi \circ E$ is $E^*(\theta)$-{psh} on $U' \coloneqq E^{-1}(U)$. 
Moreover, we can use Proposition \ref{descend of psh} to descend a $E^*(\theta)$-psh function on $U'$ to a $\theta$-psh function on $U$.
\end{rem}
 
Bedford--Taylor theory extends to $\theta$-psh functions as follows. 
Similarly as in Theorem \ref{thm:8}, we consider closed strata subsets $D_1,\ldots,D_q$ of $U$ of codimension at least one intersecting a closed strata subset $B$ of $U$ $(n-p)$-properly for some $p \in \N$.

\begin{thm} \label{MA-currents for theta-psh}
Assume that $\Sigma$ is smooth. For $j=1,\ldots,q$, let $g_j$ be a {continuous} Green function for the piecewise linear function $\Psi_j$ with respect to the fan $\Sigma$ and let $\theta_j \coloneqq c_1(\Psi_j,g_j)$. For  $\theta_j$-psh functions $\varphi_j$ with $L(\varphi_j) \subset D_j$ and any $T \in D_{\cl,+}^{p,p}(U)$ with $\supp(T) \subset B$,  there are unique product currents
\begin{align*}
(\varphi_1+g_1)(d'd''\varphi_{2}+\theta_2)\wedge\ldots \wedge (d'd''\varphi_{q}+\theta_q)\wedge T
& \in D^{p+q-1,p+q-1}(U),  \\
(d'd''\varphi_{1}+\theta_1)\wedge\ldots \wedge (d'd''\varphi_{q}+\theta_q) \wedge T 
& \in D_{\cl,+}^{p+q,p+q}(U) 
\end{align*}
such that the products are local in $U$, depend only on the factors and not on the choice of $\varphi_j,\Psi_j,g_j,\theta_j$, and agree with the product currents in Theorem \ref{thm:8} in case that all $\theta_j=0$. All properties of  product currents for psh functions seen in Section \ref{sec:bedf-tayl-calc} extend to the above product currents.
\end{thm}

\begin{proof}
Let $\sigma \in \Sigma$ and for $j=1,\ldots,q$ let $m_{\sigma,j} \in M$
such that $\Psi_j|_\sigma=m_{\sigma,j}|_\sigma$. There is an open
neighbourhood $\Omega_\sigma$ of $U \cap N(\sigma)$ in $U$ such that
for  $j=1,\ldots,q$ we have a continuous map  $h_{j,\sigma}:\Omega_\sigma \to \R$ extending $g_j+m_{\sigma,j}|_{\Omega_\sigma \cap N_\R}$. By definition of $\theta_j$-psh, the function $u_j \coloneqq \varphi_j+h_{j,\sigma}$ is psh on $\Omega_\sigma$. Using Theorem \ref{thm:8}, we define product currents on $\Omega_\sigma$ by 
\begin{align*} 
(\varphi_1+g_1)(d'd''\varphi_{2}+\theta_2)\wedge\ldots \wedge (d'd''\varphi_{q}+\theta_q)\wedge T
& \coloneqq u_{1}d'd''u_{2}\wedge\ldots \wedge d'd''u_{q}\wedge T   ,  \\
(d'd''\varphi_{1}+\theta_1)\wedge\ldots \wedge (d'd''\varphi_{q}+\theta_q) \wedge T 
& \coloneqq d'd''u_{1} \wedge \ldots \wedge d'd''u_{q} \wedge T.
\end{align*}
Since a continuous function on $\Omega_\sigma$ satisfies $d'd''[f]=0$ if and only if $f$ is an affine bounded function on each connected component of $\Omega_\sigma$, we deduce that the above definition of the product currents on $\Omega_\sigma$ depends only on the factors and not on  $\varphi_j$, $g_j$ and not on the choice of $m_{\sigma,j}$. Since the products in Theorem \ref{thm:8} are local and since the $\Omega_\sigma$ cover $U$, we get globally defined product currents. All the required properties are obvious from the corresponding properties for products of psh functions in Section \ref{sec:bedf-tayl-calc} as they can be checked locally. Uniqueness is clear by construction. 
\end{proof}

\begin{rem} \label{MA-currents for theta-psh in non-smooth case}
If $\Sigma$ is not necessarily smooth and $\Psi_j,g_j,\theta_j$ are as in Theorem \ref{MA-currents for theta-psh}, then for locally bounded $\theta_j$-psh functions $\varphi_j$ on $U$, there are unique product currents 
\begin{align*}
(\varphi_1+g_1)(d'd''\varphi_{2}+\theta_2)\wedge\ldots \wedge (d'd''\varphi_{q}+\theta_q)
& \in D^{q-1,q-1}(U),  \\
(d'd''\varphi_{1}+\theta_1)\wedge\ldots \wedge (d'd''\varphi_{q}+\theta_q) 
& \in D_{\cl,+}^{q,q}(U) 
\end{align*}
such that the products are local in $U$, depend only on the factors and not on $\varphi_j$ and $\theta_j$, and agree with the product currents in Remark \ref{Bedford-Taylor for non-smooth} in case that all $\theta_j=0$.
This follows again by applying locally Remark \ref{Bedford-Taylor for non-smooth}.  
All the properties of the product currents of psh functions seen in Section \ref{sec:bedf-tayl-calc} extend to the products above.	
\end{rem}

{Now we want to compare the Lagerberg setting with the complex situation. 
We assume that $\Sigma$ is a smooth fan.
By \cite[\S 3.3]{bps-asterisque}, the piecewise linear function $\Psi$ induces a toric line bundle $L$ {on $X_\Sigma$ together} with a toric meromorphic section $s$  on $X_\Sigma$. 
Let $V \coloneqq \trop_\infty^{-1}(U)\subset X_{\Sigma}^\an$. 
The reference {continuous} Green function $g_0$ gives rise to a {unique} continuous metric $\metr_0$ on $\Lan|_V$ {such that}  
\begin{equation} \label{induced reference metric on complex}
\|s\|_0 \coloneqq \exp \circ (-g_0).
\end{equation}

Let ${\theta_\infty} \coloneqq c_1(L,\metr_0)\in D^{1,1}(V)$
be the associated first Chern current . 
Choosing $m_\sigma\in M$ with $\Psi|_\sigma{=m_\sigma|_\sigma}$, the function $g_0+m_\sigma$ extends from $U \cap N_\R$ to a continuous real function $h_\sigma$ on a neighbourhood $\Omega_\sigma$ of $U \cap N(\sigma)$ for any $\sigma \in \Sigma$. 
This corresponds to a change to a nowhere vanishing section over   $W_ \sigma \coloneqq  \trop_\infty^{-1}(\Omega_\sigma)$ and hence 
 $\theta_\infty$ is given on $W_\sigma$ by 
\begin{equation} \label{induced Theta on complex}
\theta_\infty =  dd^c[h_\sigma\circ \trop_\infty]=dd^c [(g_0+
m_\sigma)\circ \trop_\infty].
\end{equation}}
For $\theta \coloneqq c_1(\Psi, g_0)$, Proposition \ref{projection formula trop current cont function} and \eqref{eq:3} yield that
\begin{equation} \label{normalizations}
\trop_{\infty,*}(\theta_\infty)=  \theta.
\end{equation}

\begin{thm} \label{main comparison thm of psh}
Let  $\Sigma$ be a smooth fan,  $U$  an open subset of $N_\Sigma$ and
$V \coloneqq
\trop_\infty^{-1}(U)$. In the above notation, the following conditions
are equivalent for a
function $\varphi\colon U \to \Rinf$:
\begin{enumerate}
\item\label{item:20} The function $\varphi$ is $\theta$-psh on $U$.
\item\label{item:21} The function $\varphi\circ \trop_\infty$ is
  $\theta_\infty$-psh on $V$.
\item\label{item:22} The singular metric $\metr_\varphi \coloneqq
  e^{-\varphi\circ \trop_\infty} \metr_0$ on $\Lan|_V$ is psh. 
\end{enumerate}	
This induces bijections between the set of $\theta$-psh functions on $U$, 
the set of $\SS$-invariant $\theta_{\infty}$-psh functions on $V$ and the set of singular psh toric metrics of $\Lan|_V$. 
\end{thm}

\begin{proof}
By Remark \ref{psh vs theta psh in complex}, the equivalence of \ref{item:21} and \ref{item:22} is obvious. 
The equivalence of \ref{item:20} and \ref{item:21} can be checked locally in $U$. 
Using Proposition \ref{lemm:4}, 
we easily reduce the claim to the case of psh functions.
Then  \ref{item:20} is equivalent to \ref{item:21} by
Proposition \ref{thm:2}. 

By \cite[Remark 3.1.3]{burgos-gubler-jell-kuennemann1}, the tropicalization map induces the
identification
$\Xsigmaan/\SS = N_\Sigma$ and 
hence the $\SS$-invariant functions $V \to \Rinf$ correspond
bijectively to the functions $U \to \Rinf$.  
This makes the remaining claims obvious.
\end{proof}

\begin{rem} \label{comparison thm for continuous psh}
Clearly, {a} function $\varphi{\colon U\to \Rinf}$ is  continuous real valued if and only if $\varphi\circ \trop_\infty$ 
is continuous real valued on $V$ {and the latter holds}
if and only if $\metr_\varphi$ is a continuous metric on $\Lan|_V$. 	
It follows that Theorem \ref{main comparison thm of psh} is a generalization of the archimedean part 
of a result of Burgos, Philippon and Sombra   \cite[Theorem 4.8.1]{bps-asterisque} 
where the continuous case for a complete fan $\Sigma$ and for $U=N_\Sigma$ was handled.
\end{rem}

\begin{rem} \label{inverse map on Theta psh functions}
We give  a canonical map from the set of 
$\theta_\infty$-psh functions on $V$ to the set of $\theta$-psh
functions on $U$, which induces the inverse of the construction in
Theorem \ref{main comparison thm of psh}:

Let $\Phi\colon V \to \Rinf$ be any $\theta_\infty$-psh function on $V
\coloneqq \trop_\infty^{-1}(U)$. We denote by $\Phi^{\rm av}$ the
$\SS$-invariant function on $V$ obtained from $\Phi$
averaging over the fibers of $\trop_{\infty}$ with respect to the probability
Haar measure on $\SS$. Since $\theta_\infty$ is $\SS$-invariant by the
construction in \eqref{induced reference metric on complex}, we
conclude that $\Phi^{\rm av}$ is an $\SS$-invariant $\theta_\infty$-psh
function on $V$. 
Hence, there is a function $\varphi\colon U
\to \Rinf$ with $\Phi^{\rm av} = \varphi \circ \trop_\infty$.
By Theorem \ref{main comparison thm of psh}, the function $\varphi$ is
$\theta$-psh.
\end{rem}

\begin{rem} \label{smooth comparison}
Suppose that we have chosen a smooth reference Green function $g_0$ and 
hence $\theta$ is a smooth Lagerberg form. 
Then our normalizations of $\trop_\infty^*$ yield that 
\begin{align*}
\trop_\infty^*(\theta)=  \theta_\infty.
\end{align*}
Note that if $\varphi$ is a smooth $\theta$-psh function on $U$, 
then $\varphi \circ \trop_\infty$ is a smooth $\theta_\infty$-psh function on $V$ and 
hence $\metr_\varphi$ is a smooth psh metric on $\Lan|_V$. 
However, the converse is not true, i.e.~a smooth psh toric metric on $\Lan|_V$ is not necessarily 
of this type as the smooth functions on $U$ have to be constant towards the boundary. 	
\end{rem}

\begin{rem} \label{generalization for any piecewise linear Psi}
We have assumed through  Subsection \ref{subsection: Lagerberg setting} 
that the piecewise linear function $\Psi$ on the fan $\Sigma$ has integral slopes. 
The reason was to get an induced  line bundle $L$ on $X_\Sigma$. 
If one is willing to abandon this geometric point of view working exclusively 
with $\theta_\infty$-psh functions on $V$ for $\theta_\infty$ defined
by \eqref{induced Theta on complex},
then the equivalence of \ref{item:20} and \ref{item:21} in Theorem
\ref{main comparison thm of psh}
holds also without the assumption that the slopes of $\Psi$ are integral.
\end{rem}

\subsection{The non-archimedean setting} \label{non-archimedean setting} 
We consider a non-archimedean field $K$ with valuation $v \coloneqq -\log | \phantom{a}|$. 
We will {study}  psh metrics and $\theta$-psh functions in non-archimedean geometry. 
Let  $X$ be a boundaryless separated  $K$-analytic space of pure dimension $n$. {Observe that $X$ is then a good $K$-analytic space whose topology is Hausdorff \cite[4.2.4.2]{temkin2015}.} 
We define \emph{continuous metrics} and \emph{singular metrics}  on a line bundle $L$ over $X$ 
 as in the complex case in   \ref{recall metric}. We will use forms and currents on $X$ introduced in \S \ref{section: non-archimedean situation}.

\begin{art} \label{metrics in non-arch}
{As in complex geometry, a \emph{continuous metric} on $L$ is given by
  continuous functions $-\log\|s\|\colon U \to \R$ for all frames $s$
  on open subsets $U$ of $X$ such that for {frames $s$ on $U$ and $s'$
    on $U'$ we have $\log\|s\|-\log\|s'\|=\log|s/s'|$ on $U\cap
    U'$}. A \emph{singular metric} is defined similarly, but without
  assuming continuity and with allowing the function $-\log\|s\|$ to
  take the value $-\infty$. A \emph{continuous  singular metric} is a
  singular metric such that the functions $-\log\|s\|\colon U \to
  \Rinf$ are continuous for all frames $s$ of $L$ on open subsets $U$
  of $X$.}  
\end{art}

\begin{art} \label{recall Chambert-Loir and Ducros}
A {Borel} measurable function $f\colon X \to \R \cup \{\pm \infty\}$ is called \emph{locally integrable}, if $\int_{{X}} |f | \eta < \infty$ for all positive forms $\eta \in A_c^{n,n}(X)$. 
{Hence a locally integrable function $f$ yields a} well-defined current $[f]$ on $X$.  
A singular metric $\metr$ on $L$ is called \emph{locally integrable} if the functions $-\log\|s\|\colon U \to \Rinf$ are locally integrable for all frames $s$ of $L$ on open subsets $U$ of $X$.

For a locally integrable singular metric $\metr$ on $L$, the \emph{first Chern current} $c_1(L,\metr) \in D^{1,1}(X)$ 
is defined similarly as in complex geometry. 
Locally for a frame $s$ of $L$ over $U$, it is given by
$c_1(L,\metr)|_U= d'd''[-\log \|s\|]$.
For a smooth metric, we get a smooth Chern form.
\end{art}

Pluripotential theory in non-archimedean geometry is not 
{yet as well developed } as in complex geometry. 
In \cite[{\S 5}]{chambert-loir-ducros}, a local approach to continuous finite psh functions is given in terms of positive currents. 
We generalize this {approach} slightly below by allowing the continuous functions to take also the value $-\infty$. 
We will apply this generalization in the toric setting and there we
expect that invariant psh functions are always of this form for all
reasonable notions of psh functions.
We refer to \cite{bfj-singular} for a global approach to $\theta$-psh functions in case of a discretely valued field of residue characteristic zero where the functions are only assumed to be usc.

\begin{definition}\label{theta psh functions in non-arch}
  We fix a continuous reference metric $\metr_0$ on $L$ and set
  $\theta \coloneqq c_1(L,\metr_0)$ as a current on $X$. {Let
    $U\subset X$ 
    be an open set.} A continuous
  function $f\colon U \to \Rinf$ is called \emph{$\theta$-psh} on $U$
  if for every connected component $C$ of $U$, the restriction $f|_C$
  is either identically $-\infty$ or locally integrable and
  $d'd''[f|_C]{+\theta}$ is a positive current on $C$. 
		
  We say that a function $f\colon X \to \Rinf$ is \emph{psh} if
  it is $\theta$-psh for $\theta=0$.
\end{definition}

\begin{art} \label{psh metric for non-arch} 
A continuous singular metric $\metr$ of $L$ is called {\it psh} if for
any  frame $s$ of $L$ on an open  $U \subset X$ 
the function $-\log \|s\|$ is psh on $U$.  
A smooth metric $\metr$ is psh if and only if $c_1(L,\metr)$ is a
positive form.
Similarly to Remark \ref{psh vs theta psh in complex}, 
our choice of a reference metric $\metr_0$ leads to a bijection of the set of continuous singular psh metrics of $L$ 
onto the set of continuous $\theta$-psh functions $X \to \Rinf$. 
\end{art}

\begin{rem} \label{locally psh}
Following \cite[\S 6.3]{chambert-loir-ducros},  a (continuous) metric on $L$ is called 
\emph{locally psh approximable} if it is locally a uniform limit of smooth psh metrics on $L$. 
It is obvious that a locally psh approximable metric is psh in the sense of \ref{psh metric for non-arch}.
\end{rem}
	
\begin{rem} \label{Zhang semipositive for non-archimedean}
If $L$ is a line bundle on  a proper variety $Y$ over $K$, 
then Zhang \cite{zhang-95} introduced \emph{semipositive metrics} of $L$ as 
uniform limits of metrics of $L$ induced by semipositive model metrics (see \cite{gubler-martin} for details). 
Here, we call $\metr$ a \emph{model metric} on $L$ if there is a $k \in \N \setminus \{0\}$ 
such that $\metr^{\otimes k}$ is a metric of $L^{\otimes k}$ induced
by a line bundle $\Lcal$ 
on a proper scheme ${\Ycal}$ over the valuation ring $\kcirc$ such that 
\begin{align*}
({Y},L^{\otimes k}) = ({\Ycal} \otimes_\kcirc K, \Lcal|_{{\Ycal} \otimes_\kcirc K}).
\end{align*} 
The model metric is called \emph{semipositive} if $\Lcal$ is nef.   By construction, semipositive metrics are continuous on ${Y^\an}$ 
where the latter  denotes the analytification of ${Y}$ as a Berkovich space. 
We note that ${Y^\an}$ is a good $K$-analytic space whose topology is Hausdorff.
\end{rem}

\begin{art} \label{toric semipositive in non-arch}
Let $X_\Sigma$ be a toric variety over $K$ induced by a fan $\Sigma$  in 
$N_\R$ for the lattice $N$. 
In \ref{non-archimedean tropicalization with boundary}, we have defined the tropicalization map 
$\trop_v\colon \Xsigmaan \to N_\Sigma$ 
in the non-archimedean setting.  
Again, a toric line bundle $L$ with toric meromorphic section $s$ is given by a piecewise linear function $\Psi$ 
on the fan $\Sigma$ such that $\Psi$ has integral slopes. 
Here and in the following, we always refer to \cite{bps-asterisque} for more details about the toric setting.

Let $\torus$ be the dense torus of $\Xsigmaan$ with character lattice $M=\Hom_\Z(N,\Z)$. 
Then we have the compact torus     
\begin{align*}
\SS \coloneqq \SS_v \coloneqq \{x \in \torusan \mid |\chi^m(x)|=1 \; \forall m \in M\}
\end{align*}
in $\torusan$. 
We note that $N_\Sigma$ is the \emph{canonical skeleton} of $\Xsigmaan$ given as follows. 
The map $\trop_v \colon \Xsigmaan \to N_\Sigma$ has a canonical section $\iota \colon N_\Sigma \to \Xsigmaan$ which identifies $N_\Sigma$ homeomorphically with a closed subset of $\Xsigmaan$ and then $\trop_v$ is a strong deformation retraction.
 
For any open subset $U$ of $N_\Sigma$ and $V \coloneqq \trop^{-1}(U)$, we get {from \ref{SS-invariant forms, sets and functions}} that 
\begin{equation} \label{invariant continuous functions}
\trop_v^*\colon  C(U) \longrightarrow C(V)^{\SS}, 
f \longmapsto f \circ \trop_v
\end{equation}
is an isomorphism from the space of continuous functions on $U$ onto the space of $\SS$-invariant 
continuous functions on $V$ and the inverse is given by restriction to  $U=V \cap N_\Sigma$. 
Obviously, the same holds for the spaces of usc functions on $U$ and $V$.
    
If $\Sigma$ is a complete fan and if the valuation on $K$ is non-trivial, 
then the $\SS$-invariant  semipositive toric metrics on $L$ are in bijection to the concave functions $\psi$ on $N_\R$ 
with $\psi = \Psi + O(1)$ by a result of Burgos, Philippon and Sombra 
(see \cite[Theorem 4.8.1]{bps-asterisque} for discrete valuations and  \cite[Theorem II]{gubler-hertel} for
algebraically closed non-trivially valued non-archimedean fields 
which implies the general case by \cite[Lemma 3.3]{gubler-martin}). 
The correspondence is given by 
\begin{align}\label{non-archimedean-toric-corr}
\psi \circ \trop_v = \log \|s \|.
\end{align}
The \emph{canonical metric} on a nef toric line bundle $L$
on the proper toric variety $X_\Sigma$ corresponds to the choice $\psi \coloneqq \Psi$ and again {this metric} is  not necessarily smooth.  
\end{art}

\begin{rem} \label{projection formula for forms in non-arch}
Let $\Sigma$ be any fan in $N_\R$. For an open subset $U$ of $N_\Sigma$, we set $V \coloneqq \trop_v^{-1}(U)$. 
It follows from the definition of integration of top dimensional forms on $V$  that  
\begin{equation} \label{current of integration in non-arch}
\trop_{v,*}(\delta_V)=\delta_U
\end{equation} 
for the currents of integration over $V$ and over $U$.
If $\omega$ is any smooth Lagerberg form of type $(p,p)$ on $U$, 
then $\omega$ is positive on $U$ if and only if $\trop_v^*(\omega)$ is positive on $V$. 
To prove this, we note  $\omega$ positive obviously implies $\trop_v^*(\omega)$. 
On the other hand, assume that $\trop_v^*(\omega)$ is positive. 
By autoduality of positivity {\cite[Corollary 2.2.5]{burgos-gubler-jell-kuennemann1}}, it is enough to show that the current $[\omega]$ 
is positive which follows from the other direction and \eqref{current of integration in non-arch}. 	
\end{rem}

\begin{art}\label{hypotheses}
In the following, we fix a 
 fan $\Sigma$ in $N_\R$ for a
lattice $N$ of rank $n$. Let $(L,s)$ be a toric line bundle on $X_\Sigma$ inducing the piecewise linear function $\Psi$ as in 
\ref{toric semipositive in non-arch}. 
Let $U$ be an open subset of $N_\Sigma$ and let $V \coloneqq \trop_{v}^{-1}(U)$. 
{As before Definition  \ref{theta psh for Lagerberg}, we fix a reference {continuous} Green function $g_0$ for $\Psi$ and set $\theta \coloneqq c_1(\Psi, g_0){\in D^{1,1}(U)}$.}
Our choices also determine a continuous reference metric $\metr_0$ on $\Lan|_V$ by requiring that {$-\log \|s\|_0=g_0$}. 
Then we set $\theta_v \coloneqq c_1(L,\metr_0)$ as the corresponding
first Chern current on $V$. 
\end{art}

\begin{prop} \label{push-forward of psh in non-arch} Under the
    hypotheses \ref{hypotheses}, let $\varphi_v$ be a $\theta_v$-psh
    function on $V$ and let $\varphi$ be the restriction of
    $\varphi_v$ to the canonical skeleton $V \cap
    N_\Sigma=U$.
    \begin{enumerate} \item \label{invariance of Theta}
      We have $\trop_{v,*}(\theta_v)=\theta$.
    \item \label{case of smooth reference metric for non-arch}	
      If $\metr_0$ is a smooth metric, then $\theta \in A^{1,1}(U)$ and 
      $\trop_v^*(\theta)=\theta_v$.
    \item \label{restriction of Theta-psh to canonical skeleton}
      The function $\varphi$ is  $\theta$-psh on $U$.
    \item \label{pushforward and Theta for non-arch}
      If
        $\varphi_v|_C\not\equiv-\infty$ for any connected
        component $C$ of $V$, then
      $[\varphi]=\trop_{v,*}([\varphi_v])$.
    \end{enumerate}
\end{prop}

\begin{proof}
If the metric $\metr_0$ is smooth, then \ref{case of smooth reference metric for non-arch} follows easily from the definitions and in this special case we deduce \ref{invariance of Theta} from \ref{case of smooth reference metric for non-arch} and \eqref{current of integration in non-arch}. 
In general, $\theta$ has locally a finite continuous potential $\rho$ with $d'd''[\rho]=\theta$. 
By \cite[Corollary 3.2.13]{burgos-gubler-jell-kuennemann1}, we have locally a uniform approximation of $\rho$ by smooth functions and so we deduce  \ref{invariance of Theta} from the smooth case.

To prove the remaining claims, 
we may assume that $\varphi_v$ is not identically $-\infty$ on any
connected component of $V$.  
Since $\varphi$ is the restriction of $\varphi_v$, it is clear that
$\varphi$ is {continuous}. 
For $\alpha \in A_c^{n,n}(U)$, the $(n,n)$-form $\trop_v^*(\alpha)$ has support in the canonical skeleton $V \cap N_\Sigma$ 
by Proposition \ref{lem: support on skeleton}
and hence the associated measure also has support in $V \cap N_\Sigma$. We get 
\begin{align*}
\int_V \varphi_v \trop_v^*(\alpha) = \int_{V \cap N_\Sigma} \varphi \trop_v^*(\alpha) 
= \int_{V \cap N_\Sigma} \trop_v^*(\varphi \alpha)
= \int_V \trop_v^*(\varphi \alpha) = \int_U \varphi \alpha 
\end{align*}
using \eqref{current of integration in non-arch} as above for the last equality. 
We get $[\varphi]=\trop_*([\varphi_v])$ proving \ref{pushforward and Theta for non-arch}. Note that
\begin{align*}
  d'd''[\varphi] + \theta
  = d'd''\trop_{v,*}[\varphi_v]+\trop_{v,*}(\theta_v)
  = \trop_{v,*}(d'd''[\varphi_v]+ \theta_v).
\end{align*}
Since $\trop_{v,*}$ maps positive currents to positive Lagerberg currents 
(use duality in Remark \ref{projection formula for forms in non-arch}), 
Theorem \ref{psh and positivity of Lagerberg current}
yields that $d'd''[\varphi] + \theta$ is positive and hence
  $\varphi$ is $\theta$-psh.
\end{proof}

\begin{thm} \label{main correspondence theorem for semipositiv in non-arch}
Under the hypotheses \ref{hypotheses}, the following conditions are
equivalent 
for a function $\varphi\colon  U \to \Rinf$.
\begin{enumerate}
\item\label{item:23} 
The function $\varphi$ is $\theta$-psh on $U$.
\item\label{item:24}  
The function $\varphi \circ \trop_v$ is {a continuous $\theta_v$-psh
  function $V\to \Rinf$.}
\item\label{item:25}  
{$\metr_\varphi \coloneqq e^{-\varphi \circ \trop_v}\metr_0 $ is a continuous singular psh metric on $\Lan|_V$.}
\end{enumerate} 
This induces bijections between the set of 
$\theta$-psh
functions $ U \to \Rinf$, the set of $\SS$-invariant continuous
$\theta_v$-psh functions $V\to \R$ and the set
of toric continuous singular psh metrics of $\Lan|_V$. 
Moreover, the following are equivalent:
\begin{enumerate}[label = \it{(\roman*')},ref = \it{(\roman*')}]
\item\label{item:26} The function $\varphi$ is finite and $\theta$-psh on $U$. 
\item\label{item:27} The function $\varphi \circ \trop_v$ is finite, 
  $\theta_v$-psh and continuous on $V$.
\item\label{item:28} The  metric $\metr_\varphi \coloneqq e^{-\varphi
    \circ \trop_v}\metr_0 $ on $\Lan|_V$ is a continuous psh  metric.
\item\label{item:29} The metric $\metr_\varphi$ is locally psh
  approximable on $\Lan|_V$.
\end{enumerate}
If the valuation on $K$ is non-trivial,  $\Sigma$ is complete and  $U = N_\Sigma$ (hence $V=\Xsigmaan$), then {\ref{item:26}--\ref{item:29}} are also equivalent to:
\begin{enumerate}[resume*] 
\item\label{item:30}  $\metr_\varphi$ is a continuous semipositive metric on $\Lan$.
\end{enumerate}
\end{thm}
\begin{proof}
It follows from Remark \ref{psh metric for non-arch} that our choice
of a continuous reference metric  
$\metr_0$ on $\Lan$ induces a bijection from the set of continuous
$\theta_v$-psh functions $V\to \Rinf$  
onto the set of continuous singular psh metrics on $\Lan|_V$ proving  the equivalence
of \ref{item:24} and \ref{item:25} 
as well as the equivalence of \ref{item:27} and \ref{item:28}. 
We will show now the equivalence of \ref{item:23} and \ref{item:24}
which also proves the 
equivalence of \ref{item:26} and \ref{item:27}.

Assume first that $\varphi \circ \trop_v$ is a continuous
$\theta_v$-psh function $V \to \Rinf$.  
Then Proposition \ref{push-forward of psh in non-arch} shows that 
$\varphi = \varphi \circ \trop_v|_{V\cap N_\Sigma}$ is $\theta$-psh on $U=V \cap N_\Sigma$. 
Conversely, we assume that $\varphi$ is $\theta$-psh on $U$. 
We want to prove that $\varphi \circ \trop_v$ is $\theta_v$-psh on $V$. 
Similarly as in the proof of Theorem \ref{main comparison thm of psh},
this can be checked locally and again we just have to prove that
$\varphi \circ \trop_v$ is psh on $V$ for any psh function $\varphi$ on
$U$. Since $\varphi$ is  continuous by Theorem \ref{thm:1} and since $\trop_v$ is continuous, it is clear
that $\varphi \circ \trop_v$ is continuous. We may assume that $U$ is connected
and that $\varphi$ is not identically $-\infty$. Then $\varphi$ is a
finite continuous function on $U \cap N_\R$. Since any $(n,n)$-form on
$V$ has  support in $\torusan \cap V$ as the boundary is of lower
dimension (see \cite[Lemme 3.2.5]{chambert-loir-ducros} or
\cite[Corollary 5.12]{gubler-forms}), the continuity of the function
$\varphi \circ \trop_v \colon \torusan \cap V \to \R$ yields that  
$\varphi \circ \trop_v$ is locally integrable (see \cite[Proposition 6.13]{gubler-kuenne2017}). 
As we may argue locally in $U$, 
the regularization in Theorem \ref{thm:3} 
shows that 
we may even assume that the function $\varphi$ is  a decreasing limit of smooth psh functions $\varphi_n$ on $U$. 
Since $\trop_v^*(d'd''(\varphi_n))=d'd''(\varphi_n \circ \trop_v)$ and 
since the pull-back of positive Lagerberg forms with respect to $\trop_v$ is positive 
(see Remark \ref{projection formula for forms in non-arch}), 
we conclude that $\varphi_n \circ \trop_v$ is psh. 
By the dominated convergence theorem, we deduce that $d'd''[\varphi \circ \trop_v]$ is a positive current on $V$.
This proves the equivalence of \ref{item:23} and \ref{item:24}.

If $\varphi$ is a finite continuous function in the above argument, 
then Dini's theorem shows that $\varphi$ is locally the uniform limit of the smooth psh functions 
$\varphi_n$ which shows that \ref{item:26} yields \ref{item:29}. 
We have noted in Remark \ref{locally psh} that \ref{item:29} yields
\ref{item:28}. We conclude that \ref{item:26} to \ref{item:29} are all
equivalent.

The functions $\varphi \circ \trop_v$ are
$\SS$-invariant which in turn is equivalent for the metrics
$\metr_\varphi$ to be toric. Conversely, it follows from 
\ref{toric semipositive in non-arch} that any $\SS$-invariant continuous
function $\varphi_v\colon V \to \Rinf$ is of the form $\varphi_v = \varphi \circ
\trop_v$ for a unique continuous function $\varphi\colon U \to \Rinf$. 
By Proposition \ref{push-forward of psh in non-arch} again, it follows
that $\varphi$ is $\theta$-psh. 
We get a bijection from the set of $\theta$-psh functions on $U$ 
onto the set of $\SS$-invariant continuous $\theta_v$-psh functions $V\to \Rinf$.

We assume  that  the valuation $v$ is non-trivial, $\Sigma$ is complete and $U=N_\Sigma$. 
It follows from the above that continuity of $\varphi$ and $\metr_\varphi$ are equivalent. 
For a continuous  function $\varphi\colon U \to \R$, we deduce from the result of Burgos, Philippon and Sombra mentioned in 
\ref{toric semipositive in non-arch} that $\metr_\varphi$ is semipositive if and only if 
$\psi = \log \|s\|_\varphi = -g_0 - \varphi$ 
is concave on $N_\R$.
{The latter condition} is equivalent for $\varphi$ to be $\theta$-psh by Proposition \ref{prop:4'}.
This proves the equivalence of \ref{item:26} and \ref{item:30}.  
\end{proof}

\begin{rem} \label{generalization for any piecewise linear Psi in non-arch}
Similarly as in Remark \ref{generalization for any piecewise linear Psi}, 
the equivalences of \ref{item:23} with \ref{item:24} and of
\ref{item:26} with \ref{item:27} hold more generally  
for any piecewise linear function $\Psi$ on the fan $\Sigma$ without
assuming integral slopes. 
\end{rem}

\subsection{Global regularization}\label{global-regularization} Our goal is to show a global regularization result for $\theta$-psh functions. 
In this subsection, we fix a piecewise linear function $\Psi$ on the fan $\Sigma$ of $N_\R$ with integral slopes. So far, we have used Green functions only as a reference to define $\theta$. We will show first that the concept of $\theta$-psh functions is equivalent to considering convex singular Green functions for $\Psi$. The latter have the advantage that we can omit a reference Green function and the corresponding $\theta$ which means that we do not have to impose regularity conditions on $\theta$ later during the regularization.

\begin{prop} \label{equivalence theta-psh and convex green functions}
	Let $g_0$ be a {continuous} Green function for $\Psi$ on
        a {connected} open subset
        $U$ of $N_\Sigma$  and let $\theta \coloneqq
        c_1(\Psi,g_0)$. Then the map $\varphi \to \varphi+g_0$ is an
        isomorphism from the cone $\{\varphi\colon U \to \Rinf \mid
        \text{$\varphi$ $\theta$-psh, $\varphi \not \equiv
          -\infty$}\}$ onto the cone of 
        continuous singular Green
        functions $g\colon U \cap N_\R \to \R$ for $\Psi$ {that are convex}.
\end{prop}

\begin{proof}
Since $g_0$ is a continuous Green function for $\Psi$, it is clear for a function $\varphi\colon U \to \Rinf$ with $\varphi \not \equiv -\infty$ that $\varphi$ is continuous if and only if $\varphi + g_0$ is a {continuous} singular Green function for $\Psi$. If such a $\varphi\colon U \to \Rinf$ is continuous, then it follows from 
Proposition \ref{prop:4'} that $\varphi$ is a $\theta$-psh function if and only if $(\varphi + g_0)|_{U \cap N_\R}$ is convex.
\end{proof}

\begin{definition} \label{rational piecewise affine functions}
A \emph{rational piecewise affine function} $f$ on $N_\R$ is given by a finite polyhedral complex $\Pi$ with support equal to $N_\R$ and with vertices in $N_\Q$ such that for every $\tau \in \Pi$, there are $m_\tau \in M_\Q$ and $\gamma_{\tau} \in \Q$ with $f = m_\tau + \gamma_\tau$ on $\tau$, where $M \coloneqq \Hom_\Z(N,\Z)$. 
\end{definition}

We start with global regularization by rational piecewise affine functions.

\begin{prop} \label{global regularization by rational piecewise affine functions}
Let $\Psi$ be a piecewise linear concave function on the complete fan $\Sigma$ with integral slopes. 
Then every convex continuous singular Green function for $\Psi$ is the point\-wise limit of a decreasing sequence of rational piecewise affine convex continuous Green functions for $\Psi$.
\end{prop}

\begin{proof} 
By definition of a Green function and compactness of $N_\Sigma$, the function $g+\Psi$ is bounded from above on $N_\R$.  
For $k \in \N$, let $g_k \coloneqq \max(g,-\Psi-k)$. 
Since $g$ is a convex continuous singular Green function for $\Psi$ and $-\Psi-k$ is a convex continuous Green function for $\Psi$, it is clear that $g_k$ is a convex continuous Green function for $\Psi$ with $g_k+\Psi$ bounded. 
Since $g_k + \Psi$ is bounded, it follows from \cite[Propositions 2.5.23, 2.5.24]{bps-asterisque} that the convex function $g_k$ is the uniform limit of a sequence of convex rational piecewise affine functions $(g_{k,j})_{j\in \N}$ on $N_\R$. 
Obviously, we have $g_{k,j}+\Psi$ bounded for all $k,j$ and  hence $g_{k,j}$ is a continuous Green function for $\Psi$.
As $(g_k)_{k \in \N}$ is a decreasing sequence converging pointwise to $g$, and $(g_{k,j})_{j\in \N}$ is a sequence converging uniformly to $g_{k}$, one can find  $j_{k} \in \N$
and {$\varepsilon _{k} \in \R_{>0}$} such that the sequence $(g_{k,j_{k}}+\varepsilon _{k})_{k\in \N}$ of
rational piecewise affine continuous Green  functions for $\Psi $ is decreasing and converges pointwise to $g$.
\end{proof}

Now we obtain our main global regularization result by smooth Green functions for $\Psi$. 

\begin{thm}  \label{global regularization theorem} 
Let $\Psi$ be a piecewise linear concave function on the complete fan $\Sigma$ of $N_\R$ with integral slopes and let $g$ be a convex singular Green function for $\Psi$. Then $g$ is the pointwise limit of a decreasing sequence of convex smooth Green functions for $\Psi$. 
\end{thm}

\begin{proof}
Using piecewise linear regularization in Proposition \ref{global regularization by rational piecewise affine functions}, it is enough to show that every piecewise affine convex continuous Green function $g$ for $\Psi$ is a uniform limit of convex smooth Green functions. 
Such a function $g$ can be written as 
\[
g = \max(a_1, \ldots, a_p)
\]
for affine functions $a_1,\ldots, a_p$ on $N_\R$. 
For any $\varepsilon > 0$, we replace $\max$ by the regularized maximum $M_\varepsilon$ from \cite[Lemma I.5.18]{demailly_agbook_2012} to obtain a function $g_\varepsilon$. 
The properties of the regularized maximum show that $g_\varepsilon$ is a smooth convex function on $N_\R$ with $|g-g_\varepsilon| \leq \varepsilon$. 
This shows that $g_\varepsilon$ is a convex continuous Green function for $\Psi$. 
Since $g$ is piecewise affine and a continuous Green function for $\Psi$, it is clear that $g+\Psi$ is constant towards the boundary which  implies that $g_\varepsilon$ is a smooth Green function in the sense of Definition \ref{tropical Green function}.  
\end{proof}

\begin{cor} \label{global regularization for theta-psh}
Let $\Psi$ be a piecewise linear concave function on the complete fan $\Sigma$ with integral slopes and let $g_0$ be a smooth Green function for $\Psi$ with $\theta \coloneqq c_1(\Psi,g_0)$. Then  any $\theta$-psh function on $N_\Sigma$ is the pointwise limit of a decreasing sequence of smooth $\theta$-psh functions.
\end{cor}

\begin{proof}
	This follows from Proposition \ref{equivalence theta-psh and convex green functions} and Theorem \ref{global regularization theorem}.
\end{proof}

\section{Toric Monge--Amp\`ere equations}\label{section: non-archimedean BT and MA}

In this section, we compare the tropical Bedford--Taylor product with
the Bedford--Taylor product defined by Chambert--Loir and Ducros on
the corresponding toric variety over any non-archimedean field $K$  with valuation $v \coloneqq -\log | \phantom{a}|$.
Using Bedford--Taylor products one can define the
   Monge--Amp\`ere operators and look at
  the Monge--Amp\`ere equations.
We will extend the correspondence theorems to
Monge--Amp\`ere equations comparing the complex solutions with the
tropical solutions and with the non-archimedean solutions.

\subsection{Bedford--Taylor theory on a non-archimedean toric variety} \label{subsection: Bedford--Taylor theory on a non-archimedean toric variety}

Chambert-Loir and Ducros have introduced a wedge product of first Chern
currents of locally psh approximable metrized line bundles on a 
{boundaryless separated equidimenional} 
Berkovich space, similar to the Bedford--Taylor product in complex
differential geometry. The product is constructed first in the case of
smooth metrics where it is  the wedge product of smooth first Chern
forms.  Psh approximable metrics are locally uniform limits of smooth
psh metrics and then the Bedford--Taylor product is obtained locally
as a weak limit of currents from the smooth case (see \cite[\S 5.6, \S
6.4]{chambert-loir-ducros}).

In this subsection we will show that, for toric metrics, the above
product in the non-archimedean case corresponds to our Bedford--Taylor
product on the tropicalization.
{In the case of a smooth fan, as a consequence of Remark \ref{comparison of complex BT product and tropical BT product}, we obtain a precise
correspondence to the complex Bedford--Taylor product on the
associated complex toric variety.}

We keep the assumptions and notations from Theorem \ref{main
  correspondence theorem for semipositiv in non-arch}, i.e.
$X_\Sigma$ is  any toric variety
over a non-archimedean field $K$ with associated fan $\Sigma$ in
$N_\R$, also $U$ is an open subset of $N_\Sigma$ and $V \coloneqq
\trop_v^{-1}(U)$. For $j=1, \ldots, p$, we consider a toric line bundle
$L_j$ on $X_\Sigma$ with non-trivial toric meromorphic section $s_j$
and corresponding piecewise linear function $\Psi_j$ on $\Sigma$.
{We fix a {continuous} Green function $g_{j,0}\colon U \cap
  N_\R\to \R$ for $\Psi_j$ (see  Definition \ref{tropical Green
    function})} and denote by  $\theta _{j}=c_{1}(\Psi_{j}
,\psi_{j,0})$ the first Chern current from Definition \ref{Lagerberg
  theta}.

\begin{thm} \label{comparision to non-archimedean Bedford-Taylor} For
  $j=1, \ldots, q$, let $\metr_j$ be the locally psh approximable toric
  metric on $L_j^{\rm an}$ corresponding to the continuous
  $\theta_j$-psh function $\varphi_j\colon U \to \R$ by Theorem
  \ref{main correspondence theorem for semipositiv in non-arch}. Then
\[
\trop_{v,*}\bigl(c_1(L_1,\metr_1) \wedge \ldots \wedge c_1(L_q,\metr_q)\bigr) =(d'd''[\varphi_1] + \theta_1)\wedge \ldots \wedge ( d'd''[\varphi_q]+ \theta_q)
\]
as positive Lagerberg currents on $U$. 
On the left hand side, we use the Bedford--Taylor product on the Berkovich space $V$ introduced by Chambert--Loir and Ducros and on the right hand side, we use the product introduced in {Remark \ref{MA-currents for theta-psh in non-smooth case}}. 
\end{thm}

\begin{proof}
The claim is local in the target $U$ of $\trop_v$, hence, by using Proposition \ref{lemm:4} and the multilinearity of the Bedford-Taylor product we may assume that $\theta = 0$ and that the functions $\varphi_{i}$ are finite continuous psh functions on $U$. 
It remains to show that the identity
\begin{equation} \label{Bedford-Taylor comparision for non-archimedean psh} 
\trop_{v,*}\left(d'd''[\varphi_1 \circ \trop_v]\wedge \ldots \wedge d'd''[\varphi_q \circ \trop_v] \right) = d'd''[\varphi_1] \wedge \ldots \wedge d'd''[\varphi_q]
\end{equation} 
of Lagerberg currents on $U$ holds in this case. Arguing locally on $U$, the regularization in Theorem \ref{thm:3} shows that we may assume that the continuous psh functions $\varphi_j$ are decreasing limits of smooth psh functions on $U$. 
By Dini's theorem and arguing again locally in $U$, we may assume that the limit is uniform. 
Since the Bedford--Taylor type product of Chambert--Loir and Ducros on the left hand side of \eqref{Bedford-Taylor comparision for non-archimedean psh} is continuous with respect to uniform convergence of psh metrics and weak limits of currents \cite[Corollaire 5.6.5]{chambert-loir-ducros} and since Theorem \ref{thm:7} gives a similar kind of continuity on the right hand side of \eqref{Bedford-Taylor comparision for non-archimedean psh}, it is enough to prove the identity \eqref{Bedford-Taylor comparision for non-archimedean psh} for smooth psh functions $\varphi_j$ on $U$. 
Then the claim follows from the compatibility of $\trop_v^{\ast}$ with $d$, $d'$ and the product of smooth forms, the projection formula \eqref{projection formula non-arch} and equation \eqref{current of integration in non-arch}.
\end{proof}

\subsection{The Monge-Amp\`ere equation}
\label{sec:monge-ampere-equat}
In this subsection, we recall some facts about the complex, the non-archimedean, and the real Monge-Amp\`ere equation.

\begin{rem}\label{rem:3}   
 Let $X$ be either a complex manifold or a boundaryless separated Berkovich analytic
 space of {pure} dimension $n$ and $V\subset X$ an open subset. 
Then to each top degree positive current $T\in D^{n,n}(V)$ we can
associate a positive Radon measure $\lambda (T)$ on $U$ and conversely
every positive Radon measure defines a positive current
\cite[Proposition 5.4.6]{chambert-loir-ducros}.
To avoid cumbersome notation,  we will identify the spaces of positive Radon measures and that of top degree positive currents. 
Similarly, if $\Sigma $ is a 
fan, $N_{\Sigma }$ is the corresponding partial compactification and $U\subset N_{\Sigma }$ is an open subset, then we can identify the space of positive Lagerberg currents on $U$ with the space of positive Radon measures on $U$. 
\end{rem}

We will introduce the Monge--Amp\`ere measure in the  three situations mentioned above. 
First, we do it from the local perspective restricing our attention to  continuous metrics. 
Then we will also allow certain singular metrics in a global compact setting. 

\begin{art} \label{complex MA-operator for continuous}
Let $L$ be a holomorphic line bundle on a complex manifold $X$ of pure dimension $n$. Let $\metr$ be a continuous psh metric on $L$. Then Bedford--Taylor theory gives $c_1(L,\metr)^{\wedge n} $ as a positive $(n,n)$ current. We view it as a positive Radon measure on $X$ which we call the \emph{associated Monge--Amp\`ere measure}. 

Fix a  continuous reference metric $\metr_0$ on $L$ with first Chern current $\theta \coloneqq c_1(L,\metr_0)$.
Using the Bedford--Taylor products for locally bounded psh functions from \eqref{second Bedford-Taylor product},
we
define the Monge--Amp\`ere measure
\begin{equation} \label{definition of MA in terms of metric}
(dd^c \varphi+\theta)^{\wedge n}  \coloneqq c_1(L,\metr_\varphi)^{\wedge n} 
\end{equation}
for any continuous or more generally bounded $\theta$-psh function as the corresponding  bounded metric $\metr_\varphi \coloneqq e^{-\varphi}\metr_0$ on $L$ is psh.
\end{art} 

\begin{art} \label{non-arch MA for locally psh}
Let $L$ be a line bundle on a boundaryless separated Berkovich analytic space $X$ of pure dimension $n$ over a non-archimedean field $K$ with valuation $v \coloneqq -\log | \phantom{a}|$. 
For a locally psh-approximable metric $\metr$ on $L$,  the \emph{Monge--Amp\`ere measure} $c_1(L,\metr)^{\wedge n}$ is defined as a positive Radon measure by using Theorem \ref{comparision to non-archimedean Bedford-Taylor}. 

Fixing a continuous reference metric $\metr_0$ on $L$ with associated first Chern current $\theta$, we define  $(d'd''\varphi+\theta)^{\wedge n}$ for a $\theta$-psh function $\varphi$ corresponding to a locally psh-approximable metric $\metr_\varphi$ similarly as in \eqref{definition of MA in terms of metric}.
\end{art}

\begin{art} \label{tropical MA for continuous}
Let $\Sigma$ be a fan in $N_\R$ for a free abelian group $N$ of rank $n$ and let $U$ be an open subset of $N_\Sigma$. 
In this tropical toric situation, we stick  to $\theta$-psh functions, where $\Psi$ is a piecewise linear function 
on $\Sigma$ with integral slopes and $\theta \coloneqq c_1(g_0,\Psi)$ for a fixed {continuous} Green function $g_0$ for $\Psi$ (see \S \ref{subsection: Lagerberg setting}). 
For a $\theta$-psh function $\varphi\colon U \to \R$, Remark \ref{MA-currents for theta-psh in non-smooth case} gives $(d'd'' \varphi+\theta)^{\wedge n}$ as a positive Lagerberg current of type $(n,n)$ on $U$ which we identify again with a positive Radon measure on $U$ called the \emph{Monge--Amp\`ere measure of $\varphi$}. 
\end{art}

\begin{rem} \label{locality and continuity of MA}
The constructions of the Monge--Amp\`ere measures in \ref{complex MA-operator for continuous}--\ref{tropical MA for continuous} are continuous along uniformly converging sequences (\cite[Corollary I.3.6]{demailly_agbook_2012}, \cite[Corollaire 5.6.5]{chambert-loir-ducros}, Theorem \ref{thm:7} and Remark \ref{MA-currents for theta-psh in non-smooth case}) and local in $X$ resp.~$U$. 
\end{rem}

\begin{rem} \label{compatibility of MA-measures}
If $U$ is an open subset of $N_\Sigma$ for a smooth fan $\Sigma$ in $N_\R$ and if $\varphi$ is a $\theta$-psh function on $U$ as in \ref{tropical MA for continuous}, then $\varphi_\infty \coloneqq \varphi \circ \trop_\infty$ is a $\theta_\infty$-psh function on $\trop_\infty^{-1}(U)$ by Theorem \ref{main comparison thm of psh} and we have the compatiblity
\begin{equation*} \label{formula for compatibility of MA measures}
\trop_{\infty,*}(dd^c \varphi_\infty + \theta_\infty)^{\wedge n} = (dd^c \varphi+\theta)^{\wedge n}
\end{equation*}
of Monge--Amp\`ere measures using Remark \ref{comparison of complex BT product and tropical BT product} and Theorem \ref{MA-currents for theta-psh}. 
If $v$ is the valuation of a non-archimedean field $K$, then the same compatiblity holds replacing $\infty$ by $v$. 
This is based on Theorem \ref{main correspondence theorem for semipositiv in non-arch} and Theorem \ref{comparision to non-archimedean Bedford-Taylor}, and does not need that $\Sigma$ is smooth.
\end{rem}

We now discuss the global Monge-Amp\`ere equation starting with the complex case.  

\begin{art} \label{global complex MA operator and Ecal}
Let $(X,\omega)$ be a connected compact K\"ahler manifold of dimension $n$. 
Recall that an \emph{$\omega$-psh function} is a {strongly} usc
function $u\colon X \to \Rinf$ which is either identically $-\infty$
or is locally integrable with $dd^c [u] + \omega$ a positive current.
Let $\mu$ be a positive Radon measure on $X$ such that
\begin{displaymath}
\mu(X) = \int_{X}\omega ^{\wedge n}. 
\end{displaymath}
We are interested in solutions of the \emph{Monge-Amp\`ere equation}
\begin{equation}\label{eq:4}
(\omega + dd^{c}u)^{\wedge n}=\mu 
\end{equation}
with $u$ an $\omega$-psh function. 
First, we will specify the class of Radon measures and the class of $\omega$-psh functions that are allowed in that equation. 

Locally, $\omega$ is of the form $dd^c w$ for a smooth real function $w$ and hence \ref{complex MA-operator for continuous} gives a well-defined Monge--Amp\`ere measure $(\omega + dd^{c}u)^{\wedge n}$ for any  bounded $\omega$-psh function $u$ on $X$.
Guedj and Zeriahi  \cite{guedj-zeriahi2007} introduced the
space $\Ecal(X,\omega )$ of \emph{$\omega$-psh functions $u$ with full mass $\int_X \omega^n$} on $X \setminus \{u = - \infty\}$. 
More precisely, for an $\omega$-psh function $u$ on $X$ one defines the \emph{non-pluripolar Monge--Amp\`ere operator} by a strong convergence of measures 
\begin{equation} \label{non-pluripolar MA operator}
\mu _{u} \coloneqq \lim_{k\to \infty}1_{\{u > -k\}}\bigl(\omega + dd^c(\max(u,-k))\bigr)^{\wedge n}
\end{equation}
using that $\max(u,-k)$ is  bounded. 
By Stokes's theorem on $X$, one has $\mu _{u} (X)\leq \int_{X}\omega ^{n}$. 
Hence $\mu_u$ is a positive Radon measure on $X$. 
Guedj and Zeriahi define 
\[
\Ecal(X,\omega ) \coloneqq \Bigl\{u:X \to \Rinf \,\Big|\, \text{$u$ is $\omega$-psh and $\mu _{u} (X)=\int_{X}\omega ^{n}$}\Bigr\}.
\]
For $u \in \Ecal(X,\omega)$, the \emph{Monge-Amp\`ere measure} is defined by $(dd^c u+\omega)^{\wedge n} \coloneqq \mu_u$. 
The continuous and more generally the (locally) bounded $\omega$-psh functions on $X$ are included in $\Ecal(X,\omega )$ and the above Monge--Amp\`ere measure agrees with the previously constructed Monge--Amp\`ere measures on these subspaces.

Guedj and Zeriahi introduced also the subspace $\Ecal^{1}(X,\omega )$
of $\Ecal(X,\omega)$ of $\omega$-psh functions of finite energy.
It is given by all $u \in \Ecal(X,\omega)$ 
such that $u\in L^{1}((\omega +dd^{c}u)^{\wedge n})$. 
The main result of \cite{guedj-zeriahi2007}
is the existence of solutions of the Monge--Amp\`ere equation in the
space $\Ecal(X,\omega)$ and uniqueness up to adding constants in the
subspace $\Ecal^{1}(X,\omega )$.
Later, Dinew proved in \cite{dinew2009} uniqueness of the solution up to adding constants in $\Ecal(X,\omega)$.
\end{art}

\begin{thm}[\cite{guedj-zeriahi2007,dinew2009}]\label{thm:4}
Let $X$ be a complex K\"ahler manifold and $\omega $ a K\"ahler form on $X$. 
Let $\mu $ be a positive Radon measure on $X$ that does not charge any pluripolar set and satisfies $\mu (X)=\int_{X}\omega^{n}$. 
Then there exists a function $u\in \Ecal(X,\omega )$ such that
\begin{displaymath}
  (\omega +dd^{c}u)^{\wedge n}=\mu .
\end{displaymath}
Moreover, if $u_{1}$ and $u_{2}$ are two solutions in $\Ecal(X,\omega )$, then $u_{1}-u_{2}$ is constant. 
\end{thm}

\begin{art} \label{global non-arch MA operator and Ecal}
In the case of a non-archimedean field $K$, we stick 
to a  projective geometrically integral $K$-variety $Y$ of dimension $n$ with an ample line bundle $L$.
We follow  the global approach to pluripotential theory on $\Yan$   by
Boucksom, Favre and Jonsson in \cite{bfj-singular} and
\cite{bfj-solution} with generalizations in
\cite{bj:singular_metrics}. 
We refer to \cite[Definition 2.1]{bj:singular_metrics} for the
definition of a \emph{Fubini--Study--metric}. 
In the non-trivially valued case, the metric $\metr$ of $L$ is a Fubini--Study metric if and only if $\metr^{\otimes m}$ is a model metric associated to a globally generated model of $L$, see \cite[Theorem 5.14]{boucksom-eriksson2018}. 
Fubini--Study metrics are globally  decreasing (uniform) limits of smooth psh metrics (see \cite[Corollaire 6.3.4]{chambert-loir-ducros}, \cite[Lemma 2.9]{bj:singular_metrics}) and there is always a Fubini--Study metric of $L$. 
In the following, we fix a Fubini--Study metric $\metr_0$ of $L$ as a reference metric.

A possibly singular metric $\metr$ on $L$ is called \emph{semipositive} if it
is the pointwise limit of an increasing net of Fubini--Study metrics
of $L$.  
 As in \cite[Definition 5.1]{bj:singular_metrics}, we exclude the
 metric which is identically equal to $\infty$ outside the zero
 section (note that in \emph{loc.cit.~}~a logarithmic notion of
 metrics is used).
It follows from \cite[Theorem 7.8]{boucksom-eriksson2018} that for a
non-trivially valued field and a for a continuous metric on $\Lan$,
the above semipositivity notion agrees with the one introduced by
Zhang (see Remark \ref{Zhang semipositive for non-archimedean}).

The {Monge--Amp\`ere measure} of a Fubini--Study metric $\metr$ is defined by \ref{non-arch MA for locally psh}. 
It is shown in \cite[\S 6.2]{bj:singular_metrics} that there is a unique extension of the Monge--Amp\`ere measure to the space of (locally) bounded metrics which is continuous along increasing nets of singular semipositive metrics. 
This positive Radon measure has total mass $\deg_L(Y)$ (note that the measures were normalized in \cite{bj:singular_metrics} to obtain
probablity measures). 

The Monge--Amp\`ere measure $c_1(L,\metr)^{\wedge n}$ for (locally) bounded metrics $\metr$ on $L$ satisfies  the locality principle from \cite[Theorem 5.1]{bfj-singular}. 
In \emph{loc.cit.}, this is shown for a discretely valued field $K$ of residue characteristic zero, but the arguments extend to the general case using that the Monge--Amp\`ere operator for Fubini--Study metrics is local in the analytic topology. 
For any semipositive metric $\metr$ on $L$, we define the
\emph{non-pluripolar Monge--Amp\`ere operator}
\begin{equation} \label{non-arch non-pluripolar MA-operator}
\mu_{\metr} \coloneqq \lim_{k \to \infty} 1_{\{\metr < e^k\metr_0 \}} c_1(L,\min(\metr,e^k\metr_0))^{\wedge n}
\end{equation}
similary as in \eqref{non-pluripolar MA operator} using the Monge--Amp\`ere measure for (locally) bounded metrics on the right hand side.  Locality shows as in \cite[\S 6.3]{bfj-solution} that this is a limit of positive Radon measures. Moreover, the total mass of the limit is at most $\deg_L(Y)$. Similarly as in the complex case, we define 
$\Ecal(L)$ to be the set of singular semipositive metrics $\metr$ on
$L$ with $\mu_\metr(\Yan)= \deg_L(Y)$ and for such a metric we define
the \emph{Monge--Amp\`ere measure} 
$$c_1(L,\metr)^{\wedge n} \coloneqq
\mu_\metr.$$
In analogy with \ref{global complex MA operator and Ecal}, we define
$\Ecal^{1}(L)$ to be the set of metrics $\metr$ in $\Ecal(L)$ such
that $-\log(\metr/\metr_0)$ is integrable with respect to
$c_1(L,\metr)^n$. In \cite[\S 6.1]{bj:singular_metrics},
$\Ecal^{1}(L)$ is defined as the set of possibly singular metrics of
$L$ with finite energy. Using that the summands of the energy  in
\cite[\S 3.8]{bj:singular_metrics} are decreasing, the properties of
the Monge--Amp\`ere operator for metrics of finite energy given in
\cite[\S 6.2]{bj:singular_metrics} show that the definitions
agree. Note that $\Ecal^1(L)$ contains all singular (locally) bounded
semipositive metrics and hence all continuous semipositive metrics.

 A positive Radon measure  $\mu$ on $\Yan$ is called of \emph{finite energy} if 
\begin{equation} \label{measure of finite energy} 
{\int_\Yan -\log\frac{\metr_{\phantom{0}}}{\metr_0}\,d \mu} > - \infty
\end{equation}
for all $\metr \in  \Ecal^{1}(L)$. 
This is independent of the choice of $L$ and $\metr_0$ \cite[Corollary 7.10]{bj:singular_metrics}. 
We define $\Mcal^r(\Yan)$ as the space of positive Radon measures $\mu$ {on $\Yan$} with $\mu(\Yan)=r$. 
\end{art}

\begin{thm}[\cite{bj:singular_metrics}]\label{thm:9} 
Let $K$ be a field that is complete with respect to a discrete or trivial valuation and such that the residue field has characteristic zero, $Y$ a smooth projective variety over $K$ and $L$ an ample line bundle on $Y$ of degree $r$. 
Then the Monge-Amp\`ere measure induces a bijection $\Ecal^{1}(L)/\R\longrightarrow \Mcal^{r}(Y^{\an})$.
\end{thm}

The case where the residue field of $K$ has positive or mixed characteristic remains open.  

\begin{art} \label{real setting for MA}
As we will see in the next section, for toric varieties, it is also useful to study the real Monge-Amp\`ere equation.
The real Monge--Amp\`ere problem that is relevant for us is the second boundary problem that deals with the asymptotic behavior of the solutions.  
Let $N$ be a free abelian group of finite rank $n$ and let $M \coloneqq \Hom_\Z(N,\Z)$. 
We choose a basis $u_1,\ldots, u_n$ of $M$ leading to coordinates on $N_\R$ and to an identification $N_\R \simeq \R^n$. Let $\lambda$ be the corresponding Lebesgue measure.
For a $C^{2}$ function $\varphi$ on $N$, the Monge-Amp\`ere measure is given by
\begin{equation}\label{real-ma}
\MA(\varphi)=n!\det\bigl((\varphi_{ij})_{1 \leq i,j \leq n}\bigr)\lambda,\qquad \varphi_{ij}=\frac{\partial^{2} \varphi}{\partial u_{i}\partial u_{j}}.
\end{equation}
This real Monge-Amp\`ere operator can be extended to arbitrary convex functions on $N_{\R}$. 
It is then a positive Radon measure, not necessarily absolutely continuous with respect to $\lambda$.
\end{art}

\begin{rem}\label{rem:4} 
Through the identification of Remark \ref{rem:3}, the Monge--Amp\`ere operator for convex functions is given by
\begin{displaymath}
\MA(\varphi)= (d'd''\varphi)^{\wedge n}.
\end{displaymath}
This is a direct computation in the case of a $C^{2}$ function and follows in general from regularization using the continuity of both sides for decreasing sequences of convex functions.
\end{rem}

Each convex function $\varphi$ on $N_\R$ has a \emph{stability set} $\Delta (\varphi) \subset M_\R$ defined as the convex set
\begin{displaymath}
  \Delta (\varphi)=\{x\in M_{\R}\mid \text{the map }u\mapsto \varphi(u)-\langle
  u,x\rangle \text{ is bounded below}\}.
\end{displaymath}

The following result solves the \emph{second boundary problem.}

\begin{thm}\label{thm:6}
Let $\Delta \subset M_\R$ be a \emph{convex body}, i.e.~a compact convex set with non empty interior $\Delta ^{\circ}$.
Let $\mu $ be a positive Radon measure {on $N_\R$} such that $\mu (N_{\R})=n!\vol(\Delta )$. 
Then there exists a convex function $\varphi\colon N_{\R} \to \R$ such that
\begin{gather*}
\MA(\varphi) = \mu  \quad {\rm and}  \quad 
\Delta ^{\circ} \subset \Delta (\varphi) \subset \Delta.
\end{gather*}
Moreover, if $\varphi_{1}$ and $\varphi_{2}$ are two such solutions, then $\varphi_{1}-\varphi_{2}$ is constant.
\end{thm}

The second boundary problem was originally solved by Pogorelov \cite{pogorelov1964}.
In Bakelman's book  \cite[Theorem 17.1]{bakelman94:_convex_analy}, this was shown for positive Radon measures $\mu$ which are absolutely continuous with respect to the Lebesgue measure. 
Existence of a solution in the more general situation of Theorem \ref{thm:6} was shown by Berman and Berndtsson {\cite[Theorem 2.19]{berman-berndtsson2013}}. 
We add a proof for uniqueness in Appendix \ref{sec:proof-theor-refthm:6}.

The next result will come in handy when we want to understand the solutions of the second boundary value problem as singular psh metrics on toric varieties.

Recall that given a compact convex set $\Delta \subset M_\R$, its \emph{convex support function} is the function $\Phi _{\Delta }\colon N_{\R}\to \R$ defined as
\begin{displaymath}
\Phi_{\Delta } (u)\coloneqq \sup_{x\in \Delta } \langle x,u\rangle. 
\end{displaymath}

\begin{lemma}\label{lemm:8}
Let $\Delta \subset M_\R$ be a compact convex set and $\varphi\colon
N_\R \to \R$ a convex function. Then $\Delta (\varphi) \subset
\Delta$ if and only if $\varphi-\Phi_{\Delta }$ is bounded from above.
\end{lemma}

\begin{proof}
Let $\varphi^{\ast}\colon \Delta (\varphi)\to \R$ be the
\emph{Legendre dual} to $\varphi$ given by
$\varphi^{\ast}(x)=\sup_{u\in N_{\R}} \langle x,u\rangle -\varphi(u)$.
Then $\varphi^{\ast}$ is a convex function and $\varphi$ is the Legendre dual of $\varphi^{\ast}$ \cite[Corollary 12.2.1]{rockafellar1970}:
\begin{displaymath}
\varphi(u) = \sup_{x\in \Delta (\varphi)}\langle x,u\rangle -\varphi^{\ast}(x).
\end{displaymath}
In particular, we have $\varphi(0)= \sup_{x\in \Delta (\varphi)}
(-\varphi^{\ast}(x)) \in \R$. Assume that $\Delta (\varphi) \subset
\Delta$. Then
  \begin{multline*}
    \varphi(u) = \sup_{x\in \Delta (\varphi)}\langle x,u\rangle
    -\varphi^{\ast}(x) \le \sup_{x\in \Delta (\varphi)}\langle x,u\rangle
    +\varphi(0)\le \sup_{x\in \Delta }\langle x,u\rangle
    +\varphi(0)  =\Psi _{\Delta }(u)+\varphi(0).
  \end{multline*}
  Therefore $\varphi(u)-\Phi_{\Delta }(u)\le \varphi(0)$ for any $u
  \in N_\R$. Conversely, if $\varphi-\Phi_{\Delta }$ is bounded from
  above and $x\in \Delta (\varphi)$, then $\varphi(u)-\langle u,x
  \rangle$ is bounded below. Hence
  \begin{displaymath}
    \Phi_{\Delta }(u)-\langle u,x \rangle=
    (\varphi(u)-\langle u,x \rangle) - (\varphi(u)-\Phi_{\Delta }(u)) 
  \end{displaymath}
  is also bounded below. So $x\in \Delta (\Phi _{\Delta })=\Delta $. 
\end{proof}

\subsection{Toric Monge-Amp\`ere equations} \label{subsection: toric MA equations}

In this subsection we will compare the real, the complex and
the non-archimedean Monge--Amp\`ere equations for 
toric varieties. 
First, we will compare Monge-Amp\`ere equations in the local case and then in the global case. 
The local case will be applied in the next sections to solve Monge--Amp\`ere equations on abelian varieties over non-archimedean fields and the global case leads to a solution of toric Monge--Amp\`ere equations over non-archimedean fields at the end of this subsection.
  
Let $\Sigma$ be a fan in $N_\R$ for a lattice $N$ of rank
$n$. 
Let $K$ be a non-archimedean field. 
Let us denote the valuation of $K$ by $v$. 
We denote by  $X_{\Sigma,\infty}^{\an}$ (resp.~$X_{\Sigma,v}^{\an}$) the complex (resp.~Berkovich) analytification of the associated toric variety over $\C$ (resp.~over $K$) with tropicalization map $\trop_\infty$ (resp.~$\trop_v$) onto
 $N_\Sigma$. 
In the complex case, we assume for simplicity that $\Sigma$ is smooth in order to have $X_{\Sigma,\infty}^{\an}$ as a complex manifold.
Let $U$ be an open subset of $N_\Sigma$ and let $V_\infty \coloneqq \trop_\infty^{-1}(U)$ (resp.~$V_v \coloneqq \trop_v^{-1}(U)$).

Let $\Psi$ be a piecewise linear function on $\Sigma$ with integral slopes, 
{which comes as in \ref{toric semipositive metrics for complex} (resp.~\ref{toric semipositive in non-arch}) from}  a toric line bundle $L_\infty^{\rm an}$ on $X_{\Sigma,\infty}^{\rm an}$ (resp.~$L_v^{\rm an}$ on $X_{\Sigma,v}^{\rm an}$). 
We choose a continuous Green function $g_0\colon U \cap
N_\R\to {\R}$ for $\Psi$ with  Lagerberg current  $\theta=c_{1}(\Psi ,g _{0})$ on $U$ as in \S \ref{subsection: Lagerberg setting}. 
As in \eqref{induced reference metric on complex} and in \ref{hypotheses}, the Green function $g_0$ induces a reference metric on $L_\infty^{\rm an}$ (resp. $L_v^{\rm an}$) and we denote the corresponding first Chern current by $\theta_\infty$ (resp. $\theta_v$).

The Monge--Amp\`ere measures introduced in \ref{complex MA-operator for continuous}, \ref{non-arch MA for locally psh} and \ref{tropical MA for continuous} are related as follows.

\begin{prop} \label{corresponding Monge-Ampere measures}
  Let $\varphi\colon U \to \R$ be a $\theta$-psh function and let  
  $\metr_\infty$ (resp.~$\metr_v$) be the corresponding toric continuous psh
  metric on $L_\infty^{\rm 
  an}|_{V_\infty}$ (resp. $L_v^{\rm an}|_{V_v}$) (see Theorems
\ref{main comparison thm of psh} and \ref{main correspondence theorem
  for semipositiv in non-arch}). Then   
we have the identities
\begin{equation*} 
(\trop_\infty)_* (c_1(L_\infty^{\rm an}|_{V_\infty},\metr_\infty)^{\wedge n}) 
= (d'd''\varphi + \theta)^{\wedge n} 
	=(\trop_v)_*(c_1(L_v^{\rm an}|_{V_v},\metr_{v})^{\wedge n})     
\end{equation*}
of positive Radon measures on $U$.	
Observe that for the left-hand equality {we are assuming} that the fan
$\Sigma$ is smooth whereas for the right-hand equality we are not. 
\end{prop}

\begin{proof}
The complex case follows from  the projection formula for $\theta$-psh functions with respect to $\trop_\infty$, see Remark \ref{comparison of complex BT product and tropical BT product} and Theorem \ref{MA-currents for theta-psh}. 
The non-archimedean case follows  from Theorem \ref{comparision to non-archimedean Bedford-Taylor} and the definition of the Bedford--Taylor product on $U$. 
\end{proof}

\begin{art} \label{invariant toric measures}
Recall that the $\SS_\infty$-invariant positive  Radon measures on
$V_\infty$ are in canonical bijection with the positive Radon
measures on $U$ {\cite[Corollary
  5.1.17]{burgos-gubler-jell-kuennemann1}}. We consider a positive
Radon measure on $U$ and we denote the corresponding
$\SS_\infty$-invariant positive Radon measure on $V_\infty$ by $\mu_\infty$, i.e. we have 
\begin{equation} \label{corresponding measures}
\trop_*(\mu_\infty)= \mu.
\end{equation}
In the non-archimedean situation, we have seen in Remark \ref{toric
  semipositive in non-arch} that we may view $N_\Sigma$ as the
canonical skeleton in $X_{\Sigma,v}^{\rm an}$ and hence $U$ may be
viewed as a subset of $V_v$. Let $\mu_v$ be the image measure
of $\mu$ to $V_v$ with respect to this inclusion. Since the
  inclusion is proper as  a section of the proper map $\trop_v$, we deduce that
$\mu_v$ is a positive Radon measure.
\end{art}

\begin{prop} \label{local MA equation} 
Under the above hypotheses, let $\varphi\colon U \to \R$ be a
$\theta$-psh function.
Let $\metr_\infty$ be the corresponding continuous toric psh metric on $L_\infty^{\rm an}|_{V_\infty}$ {and let $\varphi_\infty \coloneqq \varphi \circ \trop_\infty$ be the corresponding $\theta_\infty$-psh function on $V_\infty$ from Theorem \ref{main comparison thm of psh}. 
Similarly, we pick a corresponding continuous toric psh metric $\metr_v$ on $L_v^{\rm an}|_{V_v}$ and a corresponding  $\theta_v$-psh function on $V_v$  from Theorem \ref{main correspondence theorem for semipositiv in non-arch}.}
For the measures $\mu$, $\mu_\infty$ and $\mu_v$ from \ref{invariant toric measures}, the following are
equivalent: 
\begin{enumerate}
\item\label{item:31} The complex Monge--Amp\`ere equation
$c_1(L_\infty^{\rm an}|_{V_\infty},\metr_\infty)^{\wedge n}=
\mu_\infty$ is satisfied on $V_{\infty}$.
\item\label{item:31'} 
The complex Monge--Amp\`ere equation $(dd^c \varphi_\infty + \theta_\infty)^{\wedge n}=\mu_\infty$ is satisfied on $V_{\infty}$.
\item\label{item:32} 
The tropical Monge--Amp\`ere equation $(d'd''\varphi + \theta)^{\wedge
  n}=\mu$ is satisfied on $U$.
\item\label{item:33} 
The non-arch.~Monge--Amp\`ere equation $c_1(L_v^{\rm an}|_{V_v},\metr_{v})^{\wedge n}=\mu_v$ is satisfied on $V_{v}$.
  \item\label{item:33'} 
The non-arch.~Monge--Amp\`ere equation $(dd^c \varphi_v + \theta_v)^{\wedge n}=
\mu_v$ is satisfied on $V_{v}$.  
\end{enumerate}
The equivalences of \ref{item:32}, \ref{item:33} and \ref{item:33'} do not require the fan $\Sigma$ to be smooth.
\end{prop}

\begin{proof}
Clearly, \ref{item:31} is equivalent to \ref{item:31'} and
\ref{item:33} is equivalent to \ref{item:33'}.  
Proposition \ref{corresponding Monge-Ampere measures} yields
\begin{displaymath}
(\trop_\infty)_*\bigl(c_1(L_\infty^{\rm an}|_{V_\infty},\metr_\infty)^{\wedge n}\bigr) = (d'd''\varphi + \theta)^{\wedge n}.
\end{displaymath}
Therefore \ref{item:32} follows  from \ref{item:31} by
 \eqref{corresponding measures}. Conversely, if \ref{item:32}
is satisfied, then, equation \ref{item:31} is
satisfied after applying $\trop_{\ast}$. Since the metric
$\metr_{\infty}$ is toric, the current $c_1(L_\infty^{\rm
    an}|_{V_\infty},\metr_\infty)^{\wedge n}$ is
  $\SS_{\infty}$-invariant. Since $\mu _{\infty}$ is also
  $\SS_{\infty}$-invariant, \ref{item:31} is
satisfied by \cite[Corollary
  5.1.17]{burgos-gubler-jell-kuennemann1}.  

Again by Proposition \ref{corresponding Monge-Ampere measures}, we have 
\begin{displaymath}
(\trop_v)_*\bigl(c_1(L_v^{\rm an}|_{V_v},\metr_v)^{\wedge n}\bigr) = (d'd''\varphi + \theta)^{\wedge n}.
\end{displaymath}
Note that $\mu_v$ agrees with $\mu$ on the canonical skeleton
$N_\Sigma$. As $\trop_v$ is the identity on the canonical skeleton, we
have $(\trop_v)_*(\mu_v)=\mu$. We conclude that \ref{item:33} yields
\ref{item:32}.

Now we  show that \ref{item:32} implies \ref{item:33}. Arguing
as above, it is
enough to show that the non-archimedean Monge--Amp\`ere measure
$c_1(L_v^{\rm an}|_{V_v},\metr_v)^{\wedge n}$ has support in the
canonical skeleton $V_v \cap N_\Sigma = U$. We may argue locally and
so, by Lemma \ref{lemm:4}, after replacing $\varphi$ by 
$\varphi+g_0+m_\sigma$, we may assume that $\varphi$ is psh on $U$. 
As a consequence the continuous function $\varphi$ is by Theorem \ref{thm:3} a decreasing limit of smooth psh-functions $\varphi_j$ on $U$. 
By Dini's theorem, the limit is uniform and so by construction the
Bedford--Taylor product	$c_1(L_v^{\rm an}|_{V_v},\metr_v)^{\wedge n}$
on $V_v$ is the weak limit of the currents associated to the smooth
Lagerberg forms $\trop_v^*\bigl((d'd''\varphi_j)^{\wedge
  n}\bigr)\in A^{n,n}(V_v)$. 
By Lemma \ref{lem: support on skeleton} we know that such an
$(n,n)$-form has support in the canonical skeleton $V_v \cap N_\Sigma
= U$ and hence the same is true for the associated Radon
measure. Passing to the weak limit of currents, we prove the
claim.
\end{proof}

\begin{lem}\label{no-charge-pluriplolar-sets-toric}
Let $U$ be an open subset of $N_{\Sigma }$ and assume that the fan $\Sigma$ is smooth.
Let $\mu $ be a positive Radon measure on $U$ such that $\mu (U\cap (N_{\Sigma }\setminus N_{\R}))=0$. 
Let $\mu _{\infty}$ denote the invariant positive Radon measure on
$V\coloneqq\trop_\infty^{-1}(U)$ defined by $\mu $.  
Then $\mu_\infty$ does not charge pluripolar subsets of $X^{\an}_{\Sigma,\infty }$.
\end{lem}

\begin{proof}
By construction, $\mu _{\infty}$ does not charge $V\setminus \T_{\infty}^{\an}$.
Hence it is enough to check that $\mu _{\infty }$ does not charge any
pluripolar subset $D$ of $V \cap \T_{\infty}^{\an}$. {Let $D$ be
  such a pluripolar subset.}
Then Josefson's Theorem (see \cite{josefson1978} or
\cite[Theorem 4.7.4]{klimek1991}) yields a psh function 
$u\colon \T_{\infty}^{\an}\to \Rinf$ not identically $-\infty$ 
such that $D\subset\{u=-\infty\}$. Let $U_{0}$ be a bounded convex open
subset of $N_{\R}$ and $V_{0}=\trop_{\infty}^{-1}(U_{0})$. Since a countable
union of such subsets $U_{0}$ covers $N_{\R}$, 
in order to show $\mu_\infty(D)=0$, is enough to show that
$\mu_{\infty}(\{u=-\infty\}\cap V_{0})=0$.
Assume that $\mu_{\infty}(\{u=-\infty\}\cap V_{0})>0$. 
It follows from \cite[Corollary I.5.14]{demailly_agbook_2012} that the function
\[
\rho:U_{0}\longrightarrow \R,\quad
r\longmapsto \int_{\trop_\infty^{-1}(\{r\})}u(t)dP_r(t)
\]
is convex on the convex open subset $U_{0}$ of $N_\R$ where $P_r$ is
the unique $\SS_\infty$-invariant probability measure on
$\trop_\infty^{-1}(\{r\})$. 
As $U_{0}$ is bounded, 
we have $\int_{V_0}ud\mu_\infty= \int_{U_0} \rho d\mu \in \R$.
However, this integral is $-\infty$ by our assumption  
$\mu_\infty(\{u=-\infty\}\cap V_{0})>0$, which gives the desired contradiction. 
\end{proof}

We switch to the global case. 
In the following, we assume that $\Sigma$ is a complete fan in $N_\R$. We will see in this toric setting that the global approaches to psh functions descibed in \ref{global complex MA operator and Ecal} and \ref{global non-arch MA operator and Ecal} have natural combinatorial analogues. 

\begin{art} \label{global toric case}
We  consider a concave 
piecewise linear function $\Psi$ on $\Sigma$ with integral slopes and
with dual polytope  $\Delta=\{x\in M_\R\,|\,\Psi(u)\leq\langle
x,u\rangle\,\forall u\in N_\R\}$ of dimension $n$. 
This means that the induced toric line bundle $L$ is big and semiample on the proper toric variety $X_\Sigma$. 
Moreover, we have $\deg_L(X)=n! \cdot \vol(\Delta)$ using the normalized volume such that $\vol(N_\R/N)=1$. 
Then the {canonical} Green function $g_0\coloneqq -\Psi$ is convex and we set $\theta \coloneqq c_1(\Psi,g_0)$. The \emph{non-pluripolar Monge--Amp\`ere operator} of  a $\theta$-psh function $\varphi:N_\Sigma \to \Rinf$ is defined by
\begin{equation} \label{tropical non-pluripolar MA operator}
\mu _{\varphi} \coloneqq \lim_{k\to \infty}1_{\{\varphi > -k\}}(\theta + d'd''\max(\varphi,-k))^{\wedge n} 
\end{equation}
as a strong limit of measures. We omit here and in the following  the case $\varphi \equiv - \infty$.
\end{art}
 
\begin{prop} \label{properties of tropical MA}
Under the hypotheses above, the following properties hold.	
\begin{enumerate}
\item \label{tropical MA is Radon}
$\mu_\varphi$ is a positive Radon measure of total mass $\mu_\varphi(N_\Sigma)\leq n! \cdot \vol(\Delta)$.
\item \label{tropical non-pluripolar MA for bounded}
If $\varphi$ is bounded, then $\mu_\varphi$ is equal to the Monge--Amp\`ere measure $(d'd''\varphi+\theta)^{\wedge n}$ from \ref{tropical MA for continuous} and we have $\mu_\varphi(N_\Sigma)= n! \cdot \vol(\Delta)$.
\item \label{tropical MA and dense orbit}
$\mu_\varphi$ is the image measure of the Monge--Amp\`ere measure $(\theta|_{N_\R} + dd^c(\varphi|_{N_\R}))^{\wedge n}$ on $N_\R$ with respect to the inclusion $j:N_\R \to N_\Sigma$.
\end{enumerate}
\end{prop}

\begin{proof}
The first claim in \ref{tropical non-pluripolar MA for bounded} is obvious from the definition in \eqref{tropical non-pluripolar MA operator}. Then the second claim in  \ref{tropical non-pluripolar MA for bounded} follows directly from Proposition \ref{corresponding Monge-Ampere measures} if $\Sigma$ is smooth and $L$ is ample using the corresponding claim in the complex case. In general, the second claim follows from this special case by passing to a projective toric desingularization and writing the pull-back of $L$ as a limit of ample $\Q$-line bundles. This proves \ref{tropical non-pluripolar MA for bounded}. 

For unbounded $\varphi$, we deduce from  \ref{tropical non-pluripolar MA for bounded} that $\mu_\varphi(N_\Sigma)\leq \vol(\Delta)$ and hence the Borel measure $\mu_\varphi$ on $N_\Sigma$ is locally finite. This proves \ref{tropical MA is Radon}.

It is clear that $\mu_\varphi$ and $(\theta|_{N_\R} + dd^c(\varphi|_{N_\R}))^{\wedge n}$ agree on the bounded subsets $\{\varphi>-k\}$  of $N_\R$ for every $k\in \N$. Since these subsets form a covering of $N_\R$, the measures agree on $N_\R$ and we get \ref{tropical MA and dense orbit}.
\end{proof}

\begin{definition} \label{tropical Ecal and MA}
  We define $\Ecal(N_\Sigma,\theta)$ as the set of $\theta$-psh
  functions $\varphi:N_\Sigma \to \Rinf$ with $\mu_\varphi(N_\Sigma)=
  n! \cdot \vol(\Delta)$ and $\Ecal^1(N_\Sigma,\theta) \coloneqq
  \{\varphi \in \Ecal(N_\Sigma,\theta) \mid \varphi \in
  L^1(\mu_\varphi)\}$. For $\varphi \in \Ecal(N_\Sigma,\theta)$, we
  define the \emph{Monge--Amp\`ere measure} $(d'd''\varphi +
  \theta)^{\wedge n} \coloneqq \mu_\varphi$.
\end{definition}

\begin{rem} \label{compare global approaches for complex toric}
To compare with the global complex approach in \ref{global complex MA
  operator and Ecal}, we assume that $\Sigma$ is a smooth fan and that
$\Psi$ is a strictly concave function on the complete fan $\Sigma$
with integral slopes. This means that the corresponding toric line
bundle $L$ is ample. Let
$\theta_\infty$ be the first Chern current of a fixed canonical metric of the toric line bundle $L_\infty^{\rm an}$.  There is also a smooth Fubini--Study metric on $L_\infty^{\rm an}$ such that the associated first Chern form is a K\"ahler form $\omega$ on the complex manifold $X_{\Sigma,\infty}^{\rm an}$. These two reference metrics lead to an isomorphism from the space of $\theta_\infty$-psh  functions to the space of $\omega$-psh functions. Let $u$ be the $\omega$-psh function corresponding to a $\theta$-psh function $f$, then we define the \emph{non-pluripolar Monge--Amp\`ere operator of $f$} by $\mu_f \coloneqq \mu_u$ 
and then we can subsequently adjust all the definitions in \ref{global complex MA operator and Ecal} replacing $\omega$ by $\theta$. 

Let $\varphi$ be a $\theta$-psh function on $N_\Sigma$. By Theorem \ref{main comparison thm of psh}, the corresponding $\SS_\infty$-invariant function $\varphi_\infty \coloneqq \varphi \circ \trop_\infty$ on $X_{\Sigma,\infty}^{\rm an}$ is $\theta_\infty$-psh. It follows immediately from the definitions of the non-pluripolar Monge--Amp\`ere operators that $\mu_\varphi = \trop_{\infty,*}(\mu_{\varphi_\infty})$. By Theorem \ref{main comparison thm of psh}, the map $\varphi \to \varphi_\infty$ induces canonical isomorphisms between $\Ecal(N_\Sigma,\theta)$ (resp. $\Ecal^1(N_\Sigma,\theta)$) and the space of $\SS_\infty$-invariant functions in $\Ecal(X_{\Sigma,\infty}^{\rm an},\theta_\infty)$ (resp. $\Ecal^1(X_{\Sigma,\infty}^{\rm an},\theta_\infty)$).
\end{rem}

We now give two proofs of the existence and uniqueness of solutions of the tropical toric  Monge--Amp\`ere equation. 
The first one uses  convex analysis through Theorem \ref{thm:6} and the second one uses the comparison with the complex situation.

\begin{lemma}\label{lemm:9}
Let $\Delta _{1}\subset \Delta _{2}$ be bounded convex sets with the same Lebesgue measure. 
Then the interior $\Delta _{2}^{\circ}$ of $\Delta_2$ in $N_\R$ is contained in $\Delta_{1}$. 
\end{lemma}
\begin{proof}
Assume that $\Delta _{2}^{\circ}\not \subset \Delta_{1}$. 
Then there is a point $p\in \Delta _{2}^{\circ}\setminus \Delta_{1}$ and a ball $B(p,\varepsilon )$ centered at $p$ which is contained in $\Delta _{2}^{\circ}$. 
By the separation theorem for convex sets, there is a non-zero affine function $H$ such that $H(p)\ge 0$ and $H(x)\le 0$ for all $x\in \Delta _{1}$. 
The set $B(p,\varepsilon )\cap \{H >0 \}$
has positive Lebesgue measure and is contained in $\Delta_{2}^{\circ}\setminus \Delta _{1}$. 
Therefore the Lebesgue measure of   $\Delta _{2}$ is strictly bigger than the Lebesgue measure of $\Delta _{1}$.\
\end{proof}

\begin{thm}\label{missing-label}
Let $\Sigma $, $\Psi $, $\Delta $ and $\theta $ be as in \ref{global toric case}. 
Let $\mu $ be a positive Radon measure on $N_{\Sigma }$ such that
\begin{displaymath}
    \mu (N_{\R})=\mu (N_{\Sigma })=n!\vol(\Delta ).
\end{displaymath}
Then there exists a $\theta $-psh function $\varphi\in \Ecal(N_\Sigma,\theta)$ such that $(d'd''\varphi + \theta)^{\wedge n}=\mu $. 
Moreover, if $\varphi_{1}$ and $\varphi_{2}$ are two such functions, then $\varphi_{1}-\varphi_{2}$ is constant. 
\end{thm}

\begin{proof}[Proof using convex analysis]
Since $\mu (N_{\R})=\mu (N_{\Sigma })=n!\vol(\Delta )$, the measure $\mu $ is the image measure of a Radon measure $\mu _{0}$ in $N_{\R}$ such that $\mu _{0}(N_{\R})=n!\vol(\Delta )$. 
By Theorem \ref{thm:6} there exist a convex function $f\colon
  N_{\R}\to \R$ such that $\MA(f)=\mu _{0}$ and $\Delta
  ^{\circ}\subset \Delta (f)\subset \Delta $. 
  We define the function
  $\varphi\colon N_{\Sigma }\to \Rinf$ by
  \begin{displaymath}
    \varphi(x)=\limsup_{\substack{y\to x\\y\in N_{ \R}}} (f(y)+\Psi(y)). 
  \end{displaymath}
  Note that the values of this function belong to $\Rinf$ because, by
  Lemma \ref{lemm:8} the function $f+\Psi $ is bounded from above. By
  Proposition \ref{prop:4'}, the function $\varphi$ is $\theta
  $-psh. By Proposition \ref{properties of tropical MA}, the
  measure $\mu _{\varphi}$ is the image measure of $(\theta|_{N_\R} +
  dd^c(\varphi|_{N_\R}))^{\wedge n}=\MA(f)$ and hence $\mu _{\varphi}=\mu
  $. We conclude that $\mu _{\varphi}(N_{\Sigma })=\mu (N_{\Sigma
  })=n!\vol(\Delta )$ proving $\varphi\in \Ecal(N_\Sigma,\theta)$. 

For  uniqueness, assume that $\varphi_{1}$ and $\varphi_{2}$ are $\theta $-psh functions in $\Ecal(N_\Sigma,\theta)$ which solve the Monge--Amp\`ere equation. 
For $i=1,2$, write $f_{i}={\varphi_{i}|_{N_\R}}-\Psi $. 
Since
$\varphi_{i}$ is $\theta $-psh, we get a convex function $f_{i}$ on $N_\R$. 
Since $\varphi_{i}$ is continuous on the compact set $N_{\Sigma }$ and does not attain the value $+\infty$, it is bounded above. 
By Lemma \ref{lemm:8}, we have $\Delta (f_{i})\subset \Delta $. 
Since $\varphi_{i}\in \Ecal(N_\Sigma,\theta)$, the Lebesgue measure of $\Delta (f_{i})$, that agrees with the total mass of the Monge--Amp\`ere measure of $f_{i}$, agrees with the Lebesgue measure of $\Delta $. 
By Lemma \ref{lemm:9}, we have $\Delta^\circ \subset \Delta(\varphi_i)$. 
It follows from Theorem \ref{thm:6} that 
$f_{1}-f_{2}$ is constant.
\end{proof}
\begin{proof}[Proof using the complex Monge--Amp\`ere equation]
  We assume first that $\Sigma$ is a smooth fan and that $\Psi $ is
  strictly concave in $\Sigma $.  
Then $X_{\Sigma,\infty}^{\rm an}$ is a complex manifold and $\Psi$ induces a toric meromorphic section $s$ of the holomorphic ample toric line bundle $L_\infty^{\rm an}$  on $X_{\Sigma,\infty}^{\rm an}$ again by \ref{toric semipositive metrics for complex}. 
The canonical Green function $g_0=-\Psi $ induces the canonical metric
$\metr_{\infty,{\rm can}}$ on the toric line bundle $L_\infty^{\rm
  an}$ and we set $\theta_\infty \coloneqq c_1(L_\infty^{\rm an},
\metr_{\infty,{\rm can}})$.
Let $\mu _{\infty}$ be as in \ref{invariant toric measures}. 
Introducing a K\"ahler form $\omega$ as in \ref{compare global
  approaches for complex toric} and applying Theorem \ref{thm:4}, we
deduce that there is $\varphi_\infty \in \Ecal(X_{\Sigma,\infty}^{\rm
  an},\theta_\infty)$, unique up to adding a constant, solving the
complex Monge--Amp\`ere equation
\[
(dd^c \varphi_\infty + \theta_\infty)^{\wedge n} = \mu_\infty.
\]
Uniqueness shows that $\varphi_\infty$ is $\SS_\infty$-invariant. 
By Theorem \ref{main comparison thm of psh}, there is a  $\theta$-psh function $\varphi$ with $\varphi_\infty = \varphi \circ \trop_\infty$. 
By Remark \ref{compare global approaches for complex toric}, we have
$\varphi \in \Ecal(N_\Sigma,\theta)$ and  \eqref{corresponding
  measures} yields $(dd^c \varphi + \theta)^{\wedge n} = \mu$ and
$\varphi$ is uniquely determined by these conditions up to adding a
constant.

If $\Psi $ is not strictly concave, but only concave, and the polytope 
$\Delta $ is full dimensional, then the line $L$ is no longer ample
but it is still big and nef. Then we use the
generalization of  Theorem \ref{thm:4}  in \cite[Theorem 
A]{boucksometal10:_monge_amper_big} and proceed in the same way as
before. 

If $\Sigma$ is not smooth, we choose a projective toric
desingularisation $X_{\Sigma'}\to X_\Sigma$ corresponding to a smooth
fan $\Sigma'$ subdividing $\Sigma$.  Then the pull back $L'$ of $L$ to
$X_{\Sigma'}$ is still big and nef. Let
$\theta_\infty'$ be the first Chern current of the canonical metric of
$(L')_\infty^{\rm an}$ and let $\theta' \coloneqq c_1(\Psi,-\Psi)$ the
canonical first Chern current on $N_{\Sigma'}$.  Applying again
\cite[Theorem A]{boucksometal10:_monge_amper_big}, we deduce that
there is a $\theta_\infty'$-psh function $\varphi_\infty'$ on
$X_{\Sigma',\infty}^{\rm an}$, unique up to adding a constant, solving
the complex Monge--Amp\`ere equation
$(dd^c \varphi_\infty' + \theta_\infty')^n = \mu_\infty'$ for the
invariant measure $\mu_\infty'$ on $X_{\Sigma',\infty}^{\rm an}$
induced by $\mu$. Here, we note that $\mu$ does not charge the
boundary and hence may be seen as a Radon measure on
$N_{\Sigma'}$. As above, we get a $\theta'$-psh function $\varphi'$ on $N_{\Sigma'}$
solving the real Monge--Amp\`ere equation
$(dd^c \varphi' + \theta')^{\wedge n} = \mu$ and which is unique up to
adding a constant. It follows from Remark \ref{toric resolution} that
$\varphi'$ descends to a $\theta$-psh function $\varphi$ and since the
Monge--Amp\`ere measures and $\mu$ do not charge the boundary, we see
that $\varphi$ solves $(dd^c \varphi + \theta)^{\wedge n} = \mu$, and
is unique up to adding a constant.
\end{proof}

\begin{art}  \label{compare global approaches for non-arch toric}
To compare with the global approach in \ref{global non-arch MA operator and Ecal} for a non-archimedean field $K$ with valuation $v$, we do not have to assume that the fan $\Sigma$ is smooth. 
We fix as before a  {continuous} Green function $g_0$ for the polarization function $\Psi$ of the toric line bundle $(L,s)$ and put $\theta=c_1(\Psi,g_0)$. 
For a $\theta$-psh function $\varphi\colon N_\Sigma \to \Rinf$, {Theorem \ref{main correspondence theorem for semipositiv in non-arch}} gives a corresponding toric continuous singular psh metric $\metr_v$ on $L_v^{\rm an}$. 
We claim that $\metr_v$ is also {semipositive} in the
sense of the global approach of Boucksom--Jonsson explained in
\ref{global non-arch MA operator and Ecal}.

Indeed, we apply Proposition \ref{global regularization by rational
  piecewise affine functions} to the convex Green function $g\coloneqq
\varphi+g_0$.
We obtain a decreasing sequence $(g_n)_{n\in \N}$ of rational piecewise affine convex Green functions with pointwise limit $g$.
Let $\metr_{v,n}$ be the metric on $L_v^{\rm an}$ corresponding to the Green function $g_n$. 
Then $\metr_{v,n}$ is by construction a Fubini--Study metric, 
$\metr_v$ is the increasing limit of the metrics $\metr_{v,n}$, and hence the metric $\metr_v$ is {semipositive} in the sense of  \ref{global non-arch MA operator and Ecal}.

For non-pluripolar Monge--Amp\`ere operators, we have $\mu_\varphi = \trop_{v,*}(\mu_{\metr_v})$. 
The map $\varphi \mapsto \metr_v$ induces canonical isomorphisms between $\Ecal(N_\Sigma,\theta)$ (resp.~$\Ecal^1(N_\Sigma,\theta)$) and the space of  toric metrics in $\Ecal(L_v^{\rm an})$ (resp. $\Ecal^1(L_v^{\rm an})$). 
Since $\mu_{\metr_v}$ does not charge the boundary $X_{\Sigma,v}^{\rm an} \setminus \T_v^{\rm an}$, we have the useful formula
\begin{equation} \label{non-archimedean torus formula}
\mu_{\metr_v}=j_{v}\bigl(c_1(L^{\an}_{v}|_{\T_v^{\rm an}},\metr_v)^{\wedge n}\bigr)
\end{equation}
where we use the image measure with respect to the inclusion $j_v:\T_v^{\rm an} \to X_{\Sigma,v}^{\rm an}$. Applying Proposition \ref{local MA equation} to the dense torus $\T$ and using that the real and the non-archimedean non-pluripolar Monge--Amp\`ere operators do not charge the boundaries, we see that $\mu_{\metr_v}=\mu_{\varphi,v}$, i.e.~$\mu_{\metr_v}$ is the push-forward measure of $\mu_\varphi$ with respect to the inclusion $N_\Sigma \to X_{\Sigma,v}^{\rm an}$ of the canonical skeleton.
\end{art}

We illustrate the theory in the case of a projective line.

\begin{ex} \label{projective line}
  Let $X_\Sigma = \P^1$, then the associated tropical toric
  variety $N_\Sigma$ is $[-\infty,\infty]$. 
  We pick the piecewise linear function $\Psi(u)={\min(u,0)}$ on $\R$.  
The associated toric divisor on $\P^1$ is $[\infty]$.
It induces the very ample toric line bundle $L=\Ocal_{\P^1}(1)$. 
The dual polytope is $\Delta = [0,1]$ and we have $\deg_L(X)=\vol(\Delta)=1$. 
Then the canonical Green form is the convex function $g_0 \coloneqq
{\max(-u,0)}$ on $\R$ and
$\theta \coloneqq c_1(\Psi,g_0)$ is the Dirac current $\delta_{\{0\}}$
on $[-\infty,\infty]$.
	In the following, we fix a real parameter $0<\alpha<1$ and consider the function
\begin{displaymath}
\varphi(u) \coloneqq \varphi_{\alpha }(u)\coloneqq     
\begin{cases}
	1 & \text{if $u \leq 0$}, \\
	1-u & \text{if $0 \leq u \le 1$},\\
	\frac{1-u^{\alpha }}{\alpha } & \text{if $1 < u$},\\
	-\infty & \text{if $u=\infty$.}
	\end{cases}
	\end{displaymath}
    We also allow $\alpha =0$ by using
    $\varphi_0(u)\coloneqq -\log(u)$ for $1 < u$. The unbounded function $\varphi$ is $\theta$-psh and solves the real Monge--Amp\`ere equation  $d'd''\varphi + \theta = \mu$ for the positive Radon measure 
    \begin{equation} \label{example definition of mu}
    \mu \coloneqq \mu_\alpha \coloneqq (1-\alpha)1_{\{u\ge 1\}}\frac{d'u\wedge d'' u}{u^{2-\alpha }}
\end{equation}
on $[-\infty,\infty]$. 
Note here that $\mu([-\infty,\infty])=1$ and hence $\varphi \in \Ecal(N_\Sigma,\theta)$. 
Basic analysis shows that we have $\varphi \in \Ecal^1(N_\Sigma,\theta)$ (i.e.~$\varphi$  integrable with respect to $\mu$) if and only if $\alpha <1/2$.
    
Now we switch to the complex setting. 
We fix a toric coordinate $z$ on $\P_\C^1$ with $\trop_\infty(z)=-\log|z|$. 
Then $\SS_\infty$ is the unit circle. The canonical metric on $L_\infty^{\rm an}$ induced by $\Psi$ has first Chern current $\theta_\infty$ induced by the Haar probability measure on $\SS_\infty$. 
Note that $\theta_\infty$ is the positive $(1,1)$-current with $\trop_{\infty,*}(\theta_\infty)=\theta$ illustrating \cite[Theorem 7.1.5]{burgos-gubler-jell-kuennemann1}. 
The $\SS_\infty$-invariant positive Radon measure $\mu_\infty$ from \ref{invariant toric measures} is here given by 
\begin{equation} \label{complex example formula}
\mu_\infty =    \frac{\alpha-1}{4\pi i} \cdot 1_{\{|z|\le 1/e\}} \cdot
\frac{dz\wedge d\bar z}{(-\log |z|)^{2-\alpha }z\bar z}.
\end{equation}
Note that $\mu_\infty$ is the unique 
  $\SS_{\infty}$-invariant Radon measure with
$\trop_{\infty,*}(\mu_\infty)=\mu$. Hence \eqref{complex example
  formula} follows from \eqref{example definition of mu},
\eqref{correspondences of differential operators} and the projection
formula in \cite[Corollary 5.1.8]{burgos-gubler-jell-kuennemann1}.
For $\alpha =0$, we recover the Poincar\'e metric of the punctured disk. 
Applying Remark \ref{compare global approaches for complex toric} or a direct computation show that $\varphi_\infty \coloneqq \varphi \circ \trop_\infty$ is an unbounded $\theta$-psh function in $\Ecal(\P_\C^{1,\rm an},\theta_\infty)$ solving the Monge--Amp\`ere equation $dd^c \varphi_\infty + \theta_\infty = \mu_\infty$. 
Moreover, we have $\varphi_\infty \in \Ecal^1(\P_\C^{1,\rm an},\theta_\infty)$ if and only if $\alpha < 1/2$.

Finally, we look at a non-archimedean field $K$ with valuation $v$. 
Similary as in the complex setting, $\SS_v$ is the unit circle in $\P_K^{1,\rm an}$. 
Let $\metr_{v,\rm can}$ be the canonical metric on $L_v^{\rm an}$ induced by $\Psi$, then $\varphi$ induces the continuous singular psh metric $\metr_v \coloneqq e^{-\varphi}\metr_{v,\rm can}$ on $L_v^{\rm an}$ which is in $\Ecal(L_v^{\rm an})$ and solves the Monge--Amp\`ere equation $c_1(L_v^{\rm an},\metr_v)^n = \mu_v$ by \ref{compare global approaches for non-arch toric}. Moreover, we have $\metr_v \in \Ecal^1(L_v^{\rm an})$ if and only if $\alpha < 1/2$.

We claim that $\mu_v=\mu_{\alpha,v}$ has finite energy if and only if $\alpha < 1/2$. 
Indeed, if $\alpha < 1/2$, then $\varphi_v=\varphi_{\alpha,v} \in \Ecal^1(L_v^{\rm an})$ and hence $\mu_{\alpha,v}$ is of finite energy by \cite[Proposition 7.2]{bj:singular_metrics}. 
It is clear that $\mu_{\alpha,v}$ has not finite energy for $\alpha > 1/2$ by integrating against the  $\theta_v$-psh function $\varphi_{\beta,v}$ for any $\beta$ with $1/2>\beta \geq 1-\alpha$. 
For  $\alpha=1/2$, we see that $\mu_{\alpha,v}$ has infinite energy by integrating against another test function $g$  from $\Ecal^1(L_v^{\rm an})$ which is defined by  
\[
g(u) \coloneqq \frac{1-u^{1/2}}{(1+\log(u))/2} 
\]
for $u \geq 1$, by  $1-u$ for $u \leq 1$ and by $u=-\infty$ at the boundary points. 
\end{ex}

The example {above shows}  that there is a finite positive Radon measures $\mu_v$ on the canonical skeleton $N_\Sigma$ of $X_{\Sigma,v}^{\an}$  which does not charge the boundary $N_\Sigma \setminus N_\R$ and which  is not of finite energy. 
This  explains that the following theorem goes beyond \cite{bj:singular_metrics}.

\begin{thm} \label{solution of non-arch global toric MAeq}
Let $L$ be an ample toric line bundle on a proper 
toric variety $X_\Sigma$ over a non-archimedean field $K$ with valuation $v$.  
Let $\mu_v$ be a positive Radon measure on $X_{\Sigma,v}^{\rm an}$ supported in the canonical skeleton $N_\Sigma$   with $\mu_v(N_\Sigma \setminus N_\R)=0$ and with $\mu_v(X_{\Sigma,v}^{\rm an})= \deg_{L}(X_\Sigma)$. 
Then there exists a unique up to scaling toric continuous singular psh metric $\metr_v$ on $L_v^{\rm an}$ solving the non-archimedean  Monge--Amp\`ere equation 
\begin{equation}\label{eq:5}
c_{1}(L^{\an}_{v},\metr_{v})^{\wedge n}=\mu_v.
\end{equation}
\end{thm}

\begin{proof}
Choosing a toric meromorphic section $s$ of $L$, we have seen in
\ref{toric semipositive metrics for complex} that $(L,s)$ induces a piecewise linear function $\Psi$ on $\Sigma$ with integral slopes.  
Since $L$ is ample, the function $\Psi$ is concave on $N_\R$. 
Choosing the canonical Green function $g_0 \coloneqq -\Psi$, we set
$\theta \coloneqq c_1(\Psi,g_0)$. Let $\mu $ be the measure on
$N_{\Sigma }$ determined by $\mu _{v}$. It satisfies the hypothesis of Theorem \ref{missing-label}.
Hence, there exists a $\theta $-psh function $\varphi\in
\Ecal(N_\Sigma,\theta)$ such that $(d'd''\varphi + \theta)^{\wedge n}=\mu$.

Applying Proposition \ref{local MA equation} to the dense torus $\T$ and using that the real and the non-archimedean Monge--Amp\`ere measures do not charge the boundaries, we see that the psh metric $\metr_v$ on $L_v^{\rm an}$ corresponding to $\varphi$ solves \eqref{eq:5} and is the unique up to scaling toric continuous singular metric solving this non-archimedean Monge--Amp\`ere equation.
\end{proof}

\begin{rem} \label{compare with BJ}
If the measure $\mu_v$ is of finite energy, then Theorem \ref{solution of non-arch global toric MAeq} follows from Theorem \ref{thm:9} of Boucksom--Jonsson which goes beyond the toric case and also gives then uniqueness of solutions in the space $\Ecal^1(X_{\Sigma,v}^{\rm an})$. 
As we have seen in Example \ref{projective line}, there are positive Radon measures $\mu_v$ on the canonical skeleton $N_\Sigma$ in $X_{\Sigma,v}^{\rm an}$ with $\mu_v(N_\Sigma \setminus N_\R)=0$ which do not charge the boundary and which are not of finite energy.
Hence Theorem \ref{solution of non-arch global toric MAeq} is also a strengthening of Theorem \ref{thm:9} and holds without any restrictions on the non-archimedean field $K$.    
\end{rem}	

\begin{rem} \label{generalization of main global thm}
The above proof shows that Theorem \ref{solution of non-arch global toric MAeq} generalizes to $L$ big and nef. In this toric case, we can use \eqref{non-archimedean torus formula} as a definition of the non-pluripolar Monge--Amp\`ere operator.
\end{rem}

\section{Monge--Amp\`ere equations on totally degenerate abelian varieties} \label{section: totally degenerate abelian varieties}

We apply our toric correspondence results to totally degenerate abelian 
varieties over a non-trivially valued non-archimedean field 
${(K,|\phantom{a}|)}$. 
The goal is to generalize Liu's solution of the non-archimedean 
Calabi--Yau problem \cite{liu2011}. 
We denote by ${v\coloneqq -\log |\phantom{a}|}$ the additive 
valuation of $K$ and by $\Gamma \coloneqq v(K^\times) \subset \R$ 
the additive valuation group.

\subsection{Polarized tropical abelian varieties}
\label{subsection:Polarized tropical abelian varieties}

Recall from \cite[Definition 2.4]{foster-rabinoff-shokrieh-soto} 
the definition of a polarized tropical abelian variety.

\begin{definition}\label{def-polarized-abelian-variety}
A \emph{polarized tropical
abelian variety of dimension $n$ with polarization
of degree $d^2$} is given by a quadruple
$(\Lambda,M, [\cdot , \cdot],\lambda)$ where 
\begin{enumerate}
\item
$M$ and $\Lambda$ are finitely generated free abelian groups  of rank $n$, 
\item
the pairing $[\cdot , \cdot]\colon\Lambda\times M\to \R$ 
is bilinear,
\item \label{non-degenerate}
{the bilinear form $\Lambda_\R \times M_\R \to \R$ induced by 
$[\cdot , \cdot]$ is 
non-degenerate,\footnote{It is not enough to assume that the canonical 
maps $M \to \Hom(\Lambda,\R)$ and $\Lambda \to \Hom(M,\R)$ are injective 
as then the images of $M$ and $N$ are not necessarily sublattices. 
A counterexample is $\Lambda=\Z+\Z\pi \subset \R$ and $M=\Z^2$ with pairing 
$[a,m]=a(m_1+ \nu m_2)$ where we fix $\nu \in \R \setminus \Q$. 
Using $\lambda(m_1+m_2\pi)=(m_1,m_2)\in M$, all other conditions are also satisfied.}}
\item
$\lambda\colon \Lambda\to M$ is a homomorphism with 
$\#\,\coker (\lambda)= d$,
\item \label{positive definite}
the induced pairing
$[\cdot , \lambda(\cdot)]\colon
\Lambda\times \Lambda\to \R$
is symmetric and positive-definite. 
\end{enumerate}
A polarized tropical variety $(\Lambda,M, [\cdot , \cdot],\lambda)$
is called \emph{$\Gamma$-rational} if $[\Lambda,M]\subset \Gamma$.
A \emph{morphism of polarized tropical abelian varieties} 
$(f,g)\colon (\Lambda,M, [\cdot , \cdot],\lambda)
\rightarrow(\Lambda',M', [\cdot , \cdot]',\lambda')$
is given by homomorphisms $f\colon \Lambda\to \Lambda'$
and $g\colon M'\to M$ such that $[f(\cdot),\cdot]'=[\cdot,g(\cdot)]$
and $\lambda=g\circ\lambda'\circ f$.
\end{definition}

Let $(\Lambda,M, [\cdot , \cdot],\lambda)$ be a polarized tropical abelian variety. We put $N \coloneqq \Hom_\Z(M,\Z)$. The natural pairing $\langle\cdot ,\cdot\rangle\colon N_\R\times M_\R\to \R$ leads to a unique monomorphism $\iota\colon \Lambda\to N_\R$
such that $[\cdot,\cdot]=\langle\iota(\cdot),\cdot\rangle$. Note that
\ref{non-degenerate} is equivalent to require that $\iota(\Lambda)$ is
a lattice in $N_\R$. 
From now on, we will identify $\Lambda$ with this sublattice of $N_\R$ using $\iota$. Then we have  $[\cdot,\cdot]=\langle \cdot,\cdot\rangle$ on $\Lambda \times M$.
We write 
\[
\deg(\Lambda,M, [\cdot , \cdot],\lambda) \coloneqq d^2
\coloneqq (\#\coker (\lambda))^2
\] 
for the degree of $(\Lambda,M, [\cdot , \cdot],\lambda)$.

The \emph{associated tropical abelian variety} is the real torus
  $N_\R/\Lambda$ with integral structure given by $N$ and we will
  always denote the quotient homomorphism by $\pi\colon N_\R \to N_\R/\Lambda$. Tropical abelian varieties were first considered in \cite{mikhalkin-zharkov2008}.
Furthermore the tropical abelian varieties $N_\R / \Lambda$ are tropical spaces in the sense of \cite{jell-shaw-smacka2015} and hence Lagerberg forms are defined on $N_\R/\Lambda$ (see Remark \ref{tropical spaces}).

\begin{rem}\label{rem-polarized-abelian-variety}
Let $b(\cdot,\cdot)\colon N_\R\times N_\R\to \R$ be the unique symmetric bilinear form which satisfies $b(\cdot,\cdot)=[\cdot , \lambda(\cdot)]$ on $\Lambda\times \Lambda$. 
The non-degeneracy in \ref{non-degenerate} yields that $\lambda$ extends to an isomorphism $\lambda_\R\colon \Lambda_\R=N_\R \to M_\R$ and that 
\begin{equation} \label{rem-polarized-abelian-variety-eq1}
b(x,y) = \langle x, \lambda_\R(y) \rangle
\end{equation}
for all $x,y \in N_\R$. 
Indeed \eqref{rem-polarized-abelian-variety-eq1} is clear on $N \times \Lambda$, shows that $b(N,\Lambda)\subset \Z$ and extends by linearity to $N_\R \times N_\R$. 
Using that $\lambda_\R$ is an isomorphism, we also deduce from \eqref{rem-polarized-abelian-variety-eq1} that the bilinear form $b$ is non-degenerate and hence \ref{positive definite} shows that $b$ is positive definite.
\end{rem}

\begin{art}  \label{canonical form for tropical polarized abelian variety}
Any basis $u_1,\ldots,u_n$ of $M_\R$ determines coordinates on $N_\R$.
Let $u_1',\ldots,u_n'$ denote the dual basis of $N_\R$.
The positive translation invariant Lagerberg $(1,1)$-form
\begin{equation} \label{canonical form for ptav}
\theta\coloneqq \theta_{(\Lambda,M, [\cdot , \cdot],\lambda)}\coloneqq
\sum_{k,l=1}^nb(u'_k,u'_l)d'u_k\wedge d''u_l\in A^{1,1}(N_\R)
\end{equation}
does not depend on the choice of the basis.
There is a unique $\omega \in A^{1,1}(N_\R/\Lambda)$ with $\pi^*(\omega)=\theta$.
We call $\theta$ (resp.~$\omega$) the \emph{canonical $(1,1)$-form of the 
	polarized tropical abelian variety $(\Lambda,M, [\cdot , \cdot],\lambda)$ on $N_\R$ (resp.~$N_\R/\Lambda$)}. 
This fits to the setting of \S \ref{subsection: Lagerberg setting} by using $N_\Sigma = N_\R$, $\Psi \coloneqq 0$, the Green function $g_0(u) \coloneqq \frac{1}{2} b(u,u)$ for $\Psi$ noting that $\theta = c_1(\Psi,g_0)$. 
Choosing the basis $u_1,\ldots,u_n$ such that the dual basis $u_1',\ldots,u_n'$ is a $\Z$-basis of $\Lambda$, one computes
\begin{equation} \label{volume of Lambda}
\int_{N_\R/\Lambda} \omega^n = \int_{F_\Lambda} \theta^n = n! \cdot \#\coker (\lambda) = n! \cdot \sqrt{\deg(\Lambda,M, [\cdot , \cdot],\lambda)},
\end{equation}
{where $F_\Lambda$ is a fundamental domain for the action of $\Lambda$ on $N_\R$.}
\end{art}

\subsection{Tropicalization of polarized totally degenerate 
abelian varieties} 
\label{subsection: tropicalization of polarized totally degenerate abelian varieties}
{We work over a non-archimedean field $K$ with valuation $v$ and non-trivial value group $\Gamma \coloneqq v(K^\times)$.}

\begin{art} \label{totally degenerate split abelian variety}
{A \emph{totally degenerate split abelian variety over $K$} is an abelian variety $A$ over $K$ such that there exists a morphism of analytic groups $\pi\colon \T^{\rm an}\to A^{\rm an}$ for some split torus $\T$ over $K$, which induces an isomorphism $A^{\rm an}= \T^{\rm an}/ \Lambda$ for a lattice $\Lambda$ in $\T^{\rm an}$. 
Here a \emph{lattice in $\T^{\rm an}$} means a closed discrete subgroup $\Lambda$ of $\T(K)$ such that $\trop_v$ maps $\Lambda$ isomorphically onto a complete lattice of $N_\R$, where $N$ is the cocharacter lattice of $\T$.
If $\SS$ denotes the maximal affinoid torus in $\T^{\rm an}$, then $\T$ is the split torus with character lattice $M\coloneqq \Hom\,({A_1},\G_{{\rm m},1})$ given by homomorphisms of analytic groups, where ${A_1}$ is the maximal affinoid torus $\pi(\SS)$ in $A^{\rm an}$ and $\G_{{\rm m},1}$ is the maximal affinoid torus in $\G_{\rm m}^{\rm an}$.}

{We will always identify $\Lambda$ with the lattice $\trop_v(\Lambda)$ in $N_\R$ along $\trop_v$. 
Note that $\trop_v$ induces a map $\tropbar_v\colon A^{\rm an} \to N_\R/\Lambda$  called the \emph{canonical tropicalization map of $A$}. Moreover, there is a canonical section
\begin{equation}\label{tot-deg-skeleton}
\iota_v\colon N_\R/ \Lambda\longrightarrow A^{\rm an}
\end{equation}
of $\trop_v$ which identifies $N_\R/ \Lambda$ with a closed subset 
of $A^{\rm an}$, called the \emph{canonical skeleton}, such that $\trop_v$ 
is a strong deformation retraction (see \cite[6.5]{berkovich-book}).
} 
\end{art}

\begin{art}\label{tropicalization-non-archimedean-abelian-variety} 

Let $A$ be a  totally degenerate split abelian variety over $K$ with $A^{\rm an}= \T^{\rm an}/ \Lambda$ as in \ref{totally degenerate split abelian variety}, let $M$ be the character lattice of the torus $\T$ and let $N \coloneqq \Hom_\Z(M,\Z)$. 
Then 
\[
[\cdot,\cdot]\colon \Lambda\times M\longrightarrow \R,\quad
(\gamma,m)\longmapsto {v(m(\gamma))} 
\]
is a bilinear pairing. 
We denote by $\T^\vee$ the torus with character lattice $\Lambda$.
Restricting the characters of $M$ to $\Lambda$ identifies $M$ with a lattice in $(\T^\vee)^{\rm an}$. Then the dual abelian variety $A^\vee$ of $A$ has the uniformization
$(A^\vee)^\an=(\T^\vee)^\an/M$. 
A polarization ${\phi}\colon A\to A^\vee$ induces homomorphisms
$\T\to \T^\vee$ and $\lambda\colon \Lambda\to M$.
It is verified in 
\cite[Proposition 4.7]{foster-rabinoff-shokrieh-soto}
that 
\begin{equation}\label{tropi-tot-deg-av}
\trop_v(A,{\phi})\coloneqq(\Lambda,M, [\cdot , \cdot],\lambda)
\end{equation}
is a polarized tropical abelian variety.
Recall that a \emph{morphism $f\colon (A,{\phi}) \to (A',{\phi}')$ 
of polarized abelian varieties over $K$} is given by a 
homomorphism $f\colon A \to A'$ of abelian varieties over $K$ such 
that ${\phi}=  f^\vee \circ {\phi}' \circ f$. 
Then $(A,{\phi})\mapsto \trop_v(A,{\phi})$ 
induces a functor
\begin{equation}
\label{functor-tropicalization-non-archimedean-abelian-variety-eq1}
{\trop_v}\colon
\begin{pmatrix}
\mbox{polarized totally degenerate}\\
\mbox{{split} abelian varieties over $K$}
\end{pmatrix}
{\longrightarrow}
\begin{pmatrix}
\mbox{{$\Gamma$-rational} polarized}\\
\mbox{tropical abelian varieties}
\end{pmatrix}.
\end{equation}
The rank of $M$ equals the dimension of $A$ and
the polarized tropical variety $\trop_v(A,{\phi})$ has the same 
degree as the polarized abelian variety $(A,{\phi})$  
\cite[Theorem 6.15]{bosch-luetkebohmert1991}.
Observe that it follows from results of Bosch and Lütkebohmert 
that the functor
\eqref{functor-tropicalization-non-archimedean-abelian-variety-eq1}
is essentially surjective 
\cite[\S 2,  Theorem 6.13]{bosch-luetkebohmert1991}.
\end{art}

\begin{definition} \label{definition canonical form on totally degenerate abelian variety}
{Let $(A,{\phi})$ be a polarized totally degenerate
split abelian variety over $K$ with tropicalization \eqref{tropi-tot-deg-av} .
Let $\theta\in A^{1,1}(N_\R)$ and $\omega \in A^{1,1}(N_\R/\Lambda)$
denote the canonical $(1,1)$-forms of the polarized tropical 
abelian variety $\trop_v(A,{\phi})$ from \eqref{canonical form for ptav}.
We call 
\begin{equation}\label{canonical-non-arch-theta}
\text{$\theta_v\coloneqq\trop_v^*(\theta)\in A^{1,1}(\torus^\an)$  and  $\omega_v\coloneqq\tropbar_v^*(\omega)\in A^{1,1}(A^\an)$}
\end{equation}
the \emph{canonical $(1,1)$-forms of  $(A_,{\phi})$}.}
\end{definition}

\begin{rem}\label{remark-canonical-non-archimedean-metric}
In the situation of \ref{tropicalization-non-archimedean-abelian-variety}, there is always a non-zero  $m \in \N$ such that the polarization  $m\varphi$ is  induced by an ample  line bundle $L$ on $A$.
{In fact, we can always choose $m=2$ \cite[Proposition 6.10]{mumford-etal1994}.
If $K$ is algebraically closed, then we can even choose $m=1$.} 
Let $\pi\colon \T^\an\to A^\an$ denote the 
uniformization morphism. 
The line bundle $L^\an$ on $A^\an$ carries a canonical
metric $\|\phantom{a}\|_{\mathrm{can}}$ 
\cite[Example 8.15]{gubler-kuenne2017} and we have  
\begin{equation}\label{remark-canonical-non-archimedean-metric-eq1}
\text{$c_1(\pi^*L^\an,\pi^*\|\phantom{a}\|_{\mathrm{can}})={m}\,\theta_v$ and $c_1(L^\an,\|\phantom{a}\|_{\mathrm{can}})= {m}\,\omega_v$.}
\end{equation}
This follows from the description of the canonical metric in
\cite[Example 8.15]{gubler-kuenne2017}.
It is shown  in \cite[Lemma 2.2]{bosch-luetkebohmert1991}
that the line bundle $\pi^*L^\an$ on $\torus^\an$ is trivial.
\end{rem}

\subsection{Tropicalization of decomposed polarized complex abelian varieties} 
\label{subsection: tropicalization of polarized complex abelian varieties}
We work complex analytically and consider a polarized complex abelian variety $(A,H)$ 
of dimension $n$.
The complex manifold $A^\an$ admits a uniformization
$V_A/U_A$ where $U_A$ is the lattice given by the 
image of $H_1(A^\an,\Z)$ in the tangent space
at the origin $V_A\coloneqq T_{0}A^\an$. 
The polarization is given by a Riemann form $H$, i.e.~a positive
definite sesquilinear form $H \colon V_A\times V_A\to \C$
(antilinear in the second argument) 
whose imaginary part $E\coloneqq \im(H)$ takes
integral values on the lattice $U_A$.
Recall that we can recover $H$ from $E$ by the formula
\begin{equation}\label{trop-complex-pol-ab-var-eq1}
H(v,w)=E(iv,w)+iE(v,w) \,\,\,(v,w\in V_A).
\end{equation}

\begin{rem}\label{trop-complex-pol-ab-var}
A \emph{decomposition for $(A,H)$} is a direct sum decomposition
$U_A=U_1\oplus U_2$ for subgroups $U_1,U_2$ of $U_A$ which are isotropic 
for $E$ \cite[Chapter 3 \S 1]{birkenhake-lange2004}. 
According to Frobenius 
\cite[5.1 Th.~1]{bourbaki-algebre-chapitre-9} 
we can choose a $\Z$-basis $a_1,\ldots,a_{n},b_1,\ldots,b_n$ of 
the lattice $U_A$ such that $E$ is given in this basis by a matrix
\[
\begin{pmatrix}
0&D\\-D&0
\end{pmatrix}
\]
for a diagonal matrix $D=\mathrm{diag}(d_1,\ldots,d_n)$
with  $d_j \in \N_{>0}$ satisfying $d_j|d_{j+1}$ for
$j=1,\ldots,n-1$.
The matrix $D$ is uniquely determined by $U_A$ and $E$.
We call $D$ the \emph{type} and $d^{{2}}\coloneqq
\det D^{{2}}$
the \emph{degree} of the polarized complex abelian variety $(A,H)$.
Any symplectic basis as above defines 
a decomposition for $(A,H)$ if we put 
$U_1=\langle a_1,\ldots,a_n\rangle$ and 
$U_2=\langle b_1,\ldots,b_n\rangle$ and all decompositions 
come from such a symplectic basis \cite[\S 3.1, p.~46]{birkenhake-lange2004}.
\end{rem}

\begin{definition} \label{decomposed abelian varieties}
A \emph{decomposed polarized complex abelian variety $(A,H,U_1,U_2)$}
is a polarized complex abelian variety $(A,H)$ together with
a fixed decomposition $U_A=U_1\oplus U_2$.
A \emph{morphism of decomposed polarized complex abelian varieties}
is a morphism of polarized abelian varieties which preserves
the decomposition.
\end{definition}
Decomposed polarized complex abelian varieties of 
dimension $n$ and type $D$ admit a moduli space which can
be realized as a quotient of the Siegel upper half plane $\H_{{n}}$
\cite[Proposition 8.3.3]{birkenhake-lange2004}.

Let $(A,H,U_1,U_2)$ be a decomposed polarized complex abelian variety
of dimension {$n$} and degree $d^{{2}}$.
Put $N\coloneqq U_1$ and 
let $\torus \coloneqq \Spec\,\C[M]$ be the 
multiplicative torus  with character lattice $M \coloneqq  \Hom(N,\Z)$. 
As $H$ is positive definite {and hence $E$ is non-degenerate}, 
we see that the $\C$-span of $N$ is $V_A$, i.e.~$V_A=N_\C\coloneqq N\otimes_\Z\C$.
We use the homomorphisms 
\[
e\colon \C\longrightarrow \C^\times,\quad e(z)\coloneqq \exp(2\pi iz)
\qquad\mbox{ and }\qquad t\colon \C^\times \longrightarrow \R,\quad z\mapsto -\log |z|
\]
with $t\circ e(z)=2\pi \im(z)$ to get the maps
\begin{equation} \label{complex exp and trop}
2\pi \im_N\colon
N_\C=N\otimes_\Z\C\stackrel{e_N}{\longrightarrow} 
N\otimes _\Z\C^\times {=\torus^{\rm an}}   
\stackrel{{\trop_\infty} }{\longrightarrow} N\otimes_\Z\R=N_\R
\end{equation}
where $e_N\coloneqq \id_N\otimes e$, $\trop_\infty=\id_N\otimes t$
and $\im_N\coloneqq \id_N\otimes \im$. Note that $e_N\colon V_A \to \torus^{\rm an}$ maps $U_2$ isomorphically onto a discrete closed subgroup $\Lambda \coloneqq e_N(U_2)$ of $\torus^{\rm an}$. Moreover, $\Lambda$ 
maps under $\trop_\infty$ isomorphically onto the complete lattice
\[
\trop_\infty(\Lambda)
=2\pi   \im_N (U_2)\subset N_\R.
\]
By abuse of notation similarly as in \ref{totally degenerate split abelian variety}, we will identify $\Lambda$ with the lattice $\trop_\infty(\Lambda)$. 
Note that $A^{\rm an}=V_A/U_A \simeq (V_A/U_1)/(U_A/U_1) 
\simeq \torus^\an/\Lambda$ leading to the following definition.

\begin{definition} \label{definition complex Tate uniformization}
The  \emph{complex Tate uniformization} of the decomposed polarized
complex abelian variety $(A,H,U_1,U_2)$ is defined by 
$A^{\rm an}\simeq\torus^\an/\Lambda$. 
The map $\overline{\trop}_\infty\colon A^{\rm an} \to N_\R/\Lambda$ induced by $\trop_\infty\colon \torus^{\rm an} \to N_\R$ is called \emph{the canonical tropicalization map of $(A,H,U_1,U_2)$}.
\end{definition}

The complex Tate uniformization $A^{\rm an} \simeq \mathbb
T^{\rm an}/\Lambda$	is not intrinsically associated to the
polarized abelian variety $(A,H)$, it depends on the 
decomposition $U_A=U_1 \oplus U_2$.

We also have a non-degenerate pairing
\begin{equation*}\label{def-non-deg-pairing}
[\cdot,\cdot]\colon \Lambda\times M\longrightarrow \R,\quad
[w,m]\coloneqq {m(\trop_\infty(w))}
\quad (w\in \Lambda, m\in M),
\end{equation*}
and a homomorphism
\begin{equation*}\label{def-pol-map}
\lambda\colon \Lambda\longrightarrow M,\quad
{\lambda(w)(u)\coloneqq E(v,u) \quad (w=e_N(v)\in \Lambda = e_N(U_2),\, v \in U_2, \, u\in N).}
\end{equation*}
Given $w=e_N(v)\in \Lambda$ as above, we use that $N=U_1$ 
is isotropic for $E$, \eqref{complex exp and trop}
and \eqref{trop-complex-pol-ab-var-eq1} for
\[
[w,\lambda(w)]
= E(v,\trop_\infty(w))
=2\pi E(i\im_N(v),\im_N(v))
=2\pi H(\im_N(v),\im_N(v)).
\]
We conclude that 
\begin{equation}\label{functor-trop-dec-complex-pol-ab-var}\trop_\infty(A,H,U_1,U_2)\coloneqq
(\Lambda,M,[\cdot , \cdot],\lambda)
\end{equation}
is a polarized tropical abelian variety of dimension {$n$} and degree $d^2$. 
The latter is seen by choosing a symplectic basis as in Remark \ref{trop-complex-pol-ab-var}.
The assigment \eqref{functor-trop-dec-complex-pol-ab-var}
induces a functor
\begin{equation}\label{functor-complex-pol-trop}
{\trop_\infty}\colon \begin{pmatrix}
\mbox{\rm decomposed polarized}\\
\mbox{\rm complex abelian varieties}
\end{pmatrix}
{\longrightarrow}
\begin{pmatrix}
\mbox{\rm polarized tropical}\\
\mbox{\rm abelian varieties}
\end{pmatrix}.
\end{equation}
which preserves the dimension and the degree.

\begin{prop}\label{complex-trop-surjective}
{The functor \eqref{functor-complex-pol-trop} is essentially surjective.}
\end{prop}

\begin{proof}
Let $(\Lambda,M, [\cdot , \cdot],\lambda)$ be a polarized
tropical abelian variety. 
As in Subsection \ref{subsection:Polarized tropical abelian varieties}, we consider $\Lambda$  as a lattice in $N_\R$.
We get a decomposed lattice
\[
U\coloneqq U_1\oplus U_2\coloneqq 
N\oplus \frac{1}{2\pi i}\Lambda
\subset N_\R\oplus iN_\R=N_\C
\]
and consider the alternating map $E\colon U\times U\longrightarrow \Z$ given by
\begin{equation*}
E\biggl(x+\frac{i}{2\pi }a,x'+\frac{i}{2\pi}a'\biggr)\coloneqq
b(a,x')-b(a',x) \quad(x,x'\in N, \,a,a'\in \Lambda).
\end{equation*}
We have seen after \eqref{rem-polarized-abelian-variety-eq1}
that $E$ has indeed values in $\Z$. 
We denote the unique extension of $E$ to a 
{real} bilinear form on $N_\C$ also by $E$. It is given by
\begin{equation}\label{complex-trop-surjective-eq1}
E(x+iy,x'+iy')\coloneqq
2\pi (b(x',y)-b(x,y')) \quad  (x,y,x',y'\in N_\R)
\end{equation}
and hence we have 
$E(iz,iw)=E(z,w)$ for  $z,w \in N_\C$.
By \cite[Lemma 2.1.7]{birkenhake-lange2004}, 
there is a unique hermitian form $H$ on $N_\C$
such that $E=\im(H)$.
By \eqref{trop-complex-pol-ab-var-eq1} and \eqref{complex-trop-surjective-eq1}, we get
\begin{equation} \label{complex-trop-surjective-eq1'}
H(x+iy,x+iy)= E(i(x+iy),x+iy))=2\pi\bigl(b(x,x)+b(y,y)\bigr)\,\,\,(x,y\in N_\R),
\end{equation}
which shows that $H$ is positive definite. 
Hence $H$ is a Riemann form and $N_\C/U$  is
{the analytification of a complex abelian variety $A$.} 
Now $(A,H,U_1,U_2)$ is a decomposed polarized 
complex abelian variety which maps under $\trop_\infty$
to $(\Lambda,M, [\cdot , \cdot],\lambda)$.
\end{proof}

Let $(A,H)$ be a polarized complex abelian variety with canonical 
translation invariant $(1,1)$-form 
\[
\widetilde{\omega}_{(A,H)}\coloneqq {\frac{i}{2}}\sum_{k,l=1}^n
H\biggl(\frac{\partial}{\partial z_k},\frac{\partial}{\partial  z_l}
\biggr)dz_k\wedge {d\bar z_l}\in A^{1,1}(V_A).
\]
and induced K\"ahler form $\omega_{(A,H)} \in A^{1,1}(A^{\rm an})$. 
Both forms do not depend on the choice of {holomorphic} 
coordinates $z_1, \ldots, z_n$ on $V_A$. 
We call $\theta_{(A,H)}\coloneqq p^*(\omega_{(A,H)})$ \emph{the canonical $(1,1)$-form on 
the Tate uniformization $p\colon \torus^{\rm an} \to A^{\rm an}$}.
We compare these forms to their tropical analogues from \eqref{canonical form for ptav}.
{We have introduced in \ref{comparison with complex forms} 
a natural map 
$\trop_\infty^*\colon A^{p,q}(\torus^{\rm an}) \to A^{p,q}(N_\R)$
where   $N=\Hom(U_2,\Z)$ is the cocharacter lattice of the torus $\torus^\an$.
There is also an induced map 
$\overline\trop_\infty^*\colon A^{p,q}(A^{\rm an}) \to A^{p,q}(N_\R/\Lambda)$.}

\begin{prop} \label{example-canonical-archimedean-metric}
Let $(A,H,U_1,U_2)$ be a decomposed polarized complex abelian variety with $A^{\rm an}=V_A/U_A \simeq \mathbb T^{\rm an}/\Lambda$
as above. 
We have 
\begin{equation}\label{example-canonical-archimedean-metric-eq1}
\theta_{(A,H)}=
\trop_\infty^*(\theta_{\trop_\infty(A,H,U_1,U_2)}) \quad \text{and} \quad \omega_{(A,H)}=\overline\trop_\infty^*(\omega_{\trop_\infty(A,H,U_1,U_2)}).
\end{equation}
\end{prop}

\begin{proof}
Write $\trop_\infty(A,H,U_1,U_2)=(\Lambda,M, [\cdot , \cdot],\lambda)$. Let us pick a $\Z$-basis $u_1, \ldots, u_n$ of $M$. 
The corresponding holomorphic  coordinates $z_1, \ldots, z_n$ on $V_A$ and $w_1, \ldots, w_n$ on $T^{\rm an}$ are given by $\trop^*(u_j)=w_j= \exp(2  \pi i z_j)$. 
We have $dw_j=2\pi i  w_j dz_j$ and we take 
\begin{equation*}
\trop^{\ast}(d'u_{j})=-\frac{dw_{j}}{2 {\sqrt{\pi}}  w_{j}},\quad
\trop^{\ast}(d''u_{j})=-\frac{id\bar w_{j}}{2{\sqrt{\pi}} \bar w_{j}}.
\end{equation*}
from \cite[(4.5)]{burgos-gubler-jell-kuennemann1}. 
Using also that $H=2\pi {b}$ on the isotropic real span of $U_1$ (see \eqref{complex-trop-surjective-eq1'}), we deduce easily  the claim  from  the definition of the canonical forms.	
\end{proof}

\begin{rem} \label{line bundles}
Observe that $\mathbb T^{\rm an}$ is not simply connected.
In fact, it is a domain of holomorphy and hence 
a complex Stein manifold.
Then we can  use {Cartan's theorem B} to compute the holomorphic Picard group of 
$\mathbb T^{\rm an}$ 
via the exponential sequence as $H^2(\mathbb T^{\rm an},\Z)$,
which, by the Künneth formula, is easily seen to be non-zero.
\end{rem}

Using that $H^2(A^{\rm an}, \Z)\simeq \Lambda^2 \Hom(U_A,\Z)$ (see \cite[Lemma 2.1.3]{birkenhake-lange2004}), the first Chern class associates to any line bundle $L$ on $A$ an integer valued alternating form on the lattice $U_A$. 
We say that \emph{$L$ induces the polarization $H$} if the first Chern class of $L$ corresponds to $E = \im H$. 
Clearly, such line bundles are ample and they are unique up to $\Pic^0(A)$. 
Moreover, any polarization $H$ on $A$ is induced by an ample line bundle
{\cite[Theorem 2.2.3 and Proposition 4.5.2]{birkenhake-lange2004}.}

\begin{lemma}\label{lemm:5}
  Let $(A,H,U_{1},U_{2})$ be a decomposed polarized abelian variety.
  Let $L$ be an ample line bundle on $A$ inducing the polarization and
  $p\colon \torus ^{\an}\to A^{\an}$ the Tate uniformization
  corresponding to the decomposition. Then $p^{\ast}L^{\an}$ is a
  trivial line bundle on $\torus^{\an}$. 
\end{lemma}
\begin{proof}
{Any line bundle $L^{\an}$ on $A^{\an}$ has trivial pull-back to $V_A$
  and hence we can write $L^{\an}= (V \times \C)/U_A$ for an
  action of the group $U_A$ given by
  \begin{displaymath}
  U_A \times (V_A \times \C) \longrightarrow V_A \times \C, \quad 
(u;v,z) \longmapsto (v+u, \alpha_u(v)z)  
  \end{displaymath}
for a cocycle $u \mapsto \alpha_u$  of the group $U_A$ in $H^0(V_A,
\Ocal_{V_A}^\times)$. 
If $L \in \Pic^0(A)$, then $c_1(L)$ is zero as well as the corresponding  alternating form $E$. 
Then there is a unique character $\chi$ of $U_A$ such that $L^{\an}$ is given by the cocycle $\alpha_u\equiv \chi(u)$ \cite[Proposition 2.2.2]{birkenhake-lange2004}.}

Now let $L$ be an (ample) line bundle inducing the polarization
  $H$. To show that $p^*(L^{\an})$ is trivial on $\torus^{\an}$, we
  have to prove that $L^{\an}$ can be given 
  by a cocycle $u \to \alpha_u$ with $\alpha_u \equiv 1$ for all $u
  \in U_1$ as then the trivialization descends to $V_A/U_1 \simeq
  \torus^{\an}$. 
It is shown in \cite[Lemma 3.2.2 and thereafter]{birkenhake-lange2004} that there is a line bunde $L_0$ on $A$ {which induces the polarization $H$ and} a cocycle $u \mapsto e_u$ of $L_0$ which is trivial on $U_1$ (note that $U_1 = \Lambda_2$ and $U_2= \Lambda_1$ in \cite{birkenhake-lange2004}). 
This $e_u$ is the factor of automorphy of classical theta functions for $L_0$. 

{Since $L\otimes L_0^{-1} \in \Pic^0(A)$, there is a character $\chi$
  such that $u \to \chi(u) e_u$ is a cocycle for $L$. There is a
  unique $\ell_0 \in \Hom(U_A, \Z)$ such that $\chi(u)=\exp(2\pi i
  \ell_0(u))$ for all $u \in U_A$. 
If we change the trivialization to a frame $h \in
H^0(V_A,\Ocal_{V_A}^\times)$, then the cocycle changes to
\begin{displaymath}
  \alpha_u(v) = h(v+u)^{-1} e_u(v) h(v) \quad (u \in U_A, v \in V_A).
\end{displaymath}
Now let $\ell$ be the unique complex linear form on $V_A$ which agrees
with $\ell_0$ on $U_1$. Then $h \coloneqq \exp \circ (2\pi i \ell)$
does the job as $\alpha_u = \exp(-2 \pi i \ell(u)) \chi(u) e_u$ is
trivial on $U_1$.}
\end{proof}

\subsection{Corresponding polarized abelian varieties and psh functions} 
\label{subsection: totally degenerate and complex abelian varieties}
We continue to denote by $K$ a non-archimedean field with valuation 
{$v=-\log |\phantom{a}|$ and value group $\Gamma$}. 
We define a correspondence between  polarized totally degenerate split abelian varieties over  $K$  and  decomposed polarized complex abelian varieties. 
The goal {of this subsection} is to relate psh functions on corresponding abelian varieties.

\begin{definition}
{A polarized totally degenerate split abelian variety
$(A_v,{\phi})$ over the complete non-archimedean field
$K$ \emph{corresponds} to the decomposed complex
polarized abelian variety $(A_\infty,H,U_1,U_2)$ if their tropicalizations
$\trop_v(A_v,{\phi})$ and $\trop_\infty(A_\infty,H,U_1,U_2)$
are isomorphic as polarized tropical abelian varieties.}
\end{definition}

\begin{rem} \label{remark about existence of corresponding abelian varieties}
For any polarized totally degenerate split abelian variety over $K$ there exists a corresponding decomposed complex polarized abelian variety by Proposition \ref{complex-trop-surjective}.
Furthermore a decomposed complex polarized abelian variety, whose associated
polarized tropical abelian variety is $\Gamma$-rational, corresponds to a
polarized totally degenerate split abelian variety over $K$ by \ref{tropicalization-non-archimedean-abelian-variety}.
However, corresponding abelian varieties do not determine each other up to isomorphism.
\end{rem}

{In the following consider a polarized totally degenerate split 
abelian variety $(A_v,{\phi})$ over $K$ which corresponds
to the decomposed complex polarized abelian variety $(A_\infty,H,U_1,U_2)$.
We will read the isomorphism of associated polarized tropical
abelian varieties as an identification.
Hence we have, 
\begin{equation}\label{corr-abelian-var}
\trop_v(A_v,{\phi})
=(\Lambda,M, [\cdot , \cdot],\lambda)
=\trop_\infty(A_\infty,H,U_1,U_2).
\end{equation}
{From \ref{tropicalization-non-archimedean-abelian-variety} and Definition \ref{definition complex Tate uniformization},} we obtain the  Tate uniformizations
\[
\pi_v\colon \T_v^\an\longrightarrow A_v^\an,\,\,\,\,\,
\pi_\infty \colon\T_\infty^\an\longrightarrow A_\infty^\an,
\]
{with kernels identified with $\Lambda$ along the tropicalizations.} 
Here $\T_v^\an$ and $\T_\infty^\an$ is the non-archimedean
resp.~the archimedean analytification of the split
algebraic torus with cocharacter lattice $N\coloneqq \Hom_\Z(M,\Z)$.}
From our data, we get 
complex $(1,1)$-forms $\theta_\infty\coloneqq \theta_{(A,H)}$ on $\torus_\infty^\an$
and $\omega_\infty\coloneqq \omega_{(A,H)}$ on $A_\infty^\an$,
smooth $(1,1)$-forms 
$\theta_v\coloneqq \theta_{(A,{\phi}_v)}$ on $\torus_v^\an$
and $\omega_v\coloneqq \omega_{(A,{\phi}_v)}$ on $A_v^\an$,
and Lagerberg $(1,1)$-forms 
$\theta\coloneqq \theta_{(\Lambda,M, [\cdot , \cdot],\lambda)}$ on $N_\R$
and $\omega\coloneqq \omega_{(\Lambda,M, [\cdot , \cdot],\lambda)}$ 
on $N_\R/\Lambda$.
We have seen in Proposition \ref{example-canonical-archimedean-metric} and Definition \ref{definition canonical form on totally degenerate abelian variety} that $\trop_\infty^*(\theta)=\theta_\infty$ and 
	$\trop_v^*(\theta)=\theta_\infty$.

\begin{art} \label{finiteness-remark}
Let $U$ be an open subset of $N_\R$. The cone of $\theta$-psh functions on $U$ is denoted by $\PSH(U,\theta)$. 
For $w \in \{\infty, v\}$, we have 
$U_w\coloneqq \trop_w^{-1}(U)\subset \torus_w^\an$ and
$W\coloneqq \trop_w^{-1}(U)\subset \torus_w^\an$
are invariant under the action of the compact torus $\SS_w$. 
We denote by  $\PSH(U_w,\theta_w)^{\SS_w}$ 
the cone of $\theta_w$-psh 
functions which are invariant
under the action of $\SS_w$. 
Recall from Theorems \ref{main comparison thm of psh} and
\ref{main correspondence theorem for semipositiv in non-arch} that
we have a canonical isomorphism
\begin{equation}\label{psh-corr}
\PSH(U,\theta)\longrightarrow \PSH(U_w,\theta_w)^{\SS_w},\quad
\varphi\longmapsto \varphi_\infty\coloneqq \varphi\circ\trop_w
\end{equation}
of cones. By Proposition \ref{lemm:4}, all these functions $\varphi$ and $\varphi_w$ are continuous. If they are not identically $-\infty$, then the functions are finite.
\end{art}

\begin{art} \label{theta-psh functions on abelian varieties}
We assume that $U$ is a $\Lambda$-invariant subset of $N_\R$. 
Equivalently, this means that there is an open subset $W$ of $N_\R/\Lambda$ with $U = \pi^{-1}(W)$ for the quotient map $\pi\colon N_\R \to N_\R/\Lambda$. 
Then $U_w$ is $\Lambda$-invariant  and we have $U_w = \pi_w^{-1}(W_w)$  for the open subset $W_w \coloneqq \trop_w^{-1}(W)$ of $A_w^\an$ for $w\in \{v,\infty\}$. 

{If we apply Definition \ref{theta psh for Lagerberg} literally, then we get a natural generalization of the class of $\theta$-psh functions for any tropical space and any closed symmetric $(1,1)$-current $\theta$.}
We apply this for the open subset $W$ of the tropical space $N_R/\Lambda$ and the canonical form $\omega$, then the $\omega$-psh functions on $W$ are given by the $\Lambda$-invariant $\theta$-psh functions on $U$ and for the corresponding cones we get  
\[
\PSH(W,\omega) \simeq \PSH(U,\theta)^\Lambda.
\]
Note that $\SS_w$-acts on $A_w^{\rm an}$ in a unique way such that the Tate uniformization  $\pi_w\colon \T_w^{\rm an} \to A_w^{\rm an}$
becomes $\SS_w$-equivariant. We denote by $\PSH(W_w, \omega_w)^{\SS_w}$ the cone of $\SS_w$-invariant $\omega_w$-psh functions on $W_w$. 
\end{art}

{\begin{prop}\label{lambda-invariance-cor}
For any  subset $W$ of $N_\R/\Lambda$ and $w \in \{\infty,v\}$,  
 we have an isomorphism 
	\begin{equation}
	\label{psh-corr-non-arch-arch-descent}
	\PSH(W, \omega)
	\stackrel{\sim}{\longrightarrow} \PSH(W_w, \omega_w)^{\SS_w}, \quad \varphi \longmapsto \varphi \circ \tropbar_w.
	\end{equation}
\end{prop}}

\begin{proof}
	This holds as the isomorphism \eqref{psh-corr} for $U_w \coloneqq \pi_w^{-1}(U)$ is $\Lambda$-equivariant.
\end{proof}

\subsection{Monge--Amp\`ere equations on corresponding abelian varieties} \label{subsection: MA equations on corresponding abelian varieties}

Let $K$ be a non-trivially valued non-archimedean field with valuation $v$. 
The goal of this subsection is to show existence of an invariant solution of the  Monge--Amp\`ere equation on a polarized totally degenerate split abelian variety over $K$ with respect to a positive Radon-measure supported on the canonical skeleton. 
We will generalize our result to arbitrary abelian varieties in Section \ref{section: short variant} where we also address uniqueness.

\begin{art} \label{setup for corresponding measures}
Let $(A_v,{\phi})$ be a polarized totally degenerate split abelian variety  over $K$ which corresponds to the decomposed complex polarized abelian variety $(A_\infty,H,U_1,U_2)$, i.e.	
\begin{equation*}
\trop_v(A_v,{\phi})
=(\Lambda,M, [\cdot , \cdot],\lambda)
=\trop_\infty(A_\infty,H,U_1,U_2)
\end{equation*}
as in \eqref{corr-abelian-var}. 
For $w \in \{\infty,v\}$, we have  the  Tate uniformization $\pi_w\colon \T_w^\an\to A_w^\an$ with kernel identified with $\Lambda$ along $\trop_w$, where $\T$  is the split algebraic torus with character lattice $M$ and cocharacter lattice $N$.
 {Observe that} the maximal compact torus $\SS_w$ in $\torus_w^\an$ acts on $A_w^\an$.
Let $\omega \in A^{1,1}(N_\R/\Lambda)$ be the canonical form of the polarized tropical abelian variety $(\Lambda,M, [\cdot , \cdot],\lambda)$ introduced in  \eqref{canonical form for ptav}. 
If $\omega_w$ denotes the canonical form of the polarization of $A_w$, then it follows from \eqref{canonical-non-arch-theta} and  \eqref{example-canonical-archimedean-metric-eq1} that $\omega_w = \tropbar_w^*(\omega)$ and hence $\omega_w$ is $\SS_w$-invariant.
\end{art}

\begin{art} \label{related toric measures-new}
In the setup of \ref{setup for corresponding measures}, we  want to fix corresponding invariant measures. 
We start with a positive Radon measure $\mu$ on the tropical abelian variety $N_\R/\Lambda$.

By push-forward along the proper map $\iota_v$ in \eqref{tot-deg-skeleton} to the canonical skeleton of the totally degenerate abelian variety $A_v$, we get a
positive Radon measure $\mu_v{\coloneqq \iota_v(\mu)}$ on  $A_v^\an$. 
	We have
	$\tropbar_v(\mu_v)=\mu$ for the proper map $\tropbar_v$.

	By \cite[Corollary 5.1.17]{burgos-gubler-jell-kuennemann1}, the positive $\SS_\infty$-invariant Radon measures on $\torus_\infty^{\rm an}$ correspond bijectively to the positive Radon measures on $N_\R$, i.e.~there is a unique $\SS_\infty$-in\-va\-riant positive Radon measure  $\mu_\infty$ on $A_\infty^{\rm an}$ with  $\tropbar_\infty(\mu_\infty)=\mu$.

	Let $n$ denote the common dimension of $A_\infty$, $A_v$ and $N_\R$.
	In the following we assume that 
	\begin{equation}\label{normalize-measure1}
	\mu(N_\R/\Lambda)= n!\,
	\sqrt{\deg(\Lambda,M, [\cdot , \cdot],\lambda)} 
	= \int_{N_\R/\Lambda} \omega^{\wedge n}
	\end{equation}
	where we used \eqref{volume of Lambda} on the right. Using that the tropicalization functor \eqref{functor-complex-pol-trop} keeps the degree and that $\tropbar_\infty^*$ leaves integrals invariant, we  
	note that \eqref{normalize-measure1} is equivalent to  
	\begin{equation}\label{normalize-measure2}
	\mu_\infty(A_\infty^\an)=n!\,\sqrt{\deg_{H}(A_\infty^\an)} 
	=   \int_{A_\infty^\an}\omega_\infty^{\wedge n}.
	\end{equation}
	Using that tropicalization \eqref{functor-tropicalization-non-archimedean-abelian-variety-eq1} preserves the degree, we deduce that \eqref{normalize-measure1} is 
	equivalent to 
	\begin{equation}\label{normalize-measure3}
	\mu_v(A_v^\an)=n!\,\sqrt{\deg_{{\phi}}(A_v)}
	=\int_{A_v^\an}\omega_{v}^{\wedge n}.
	\end{equation}
These equivalent conditions are necesssary to solve the following Monge--Amp\`ere equations.
\end{art}

\begin{prop}\label{equivalent-ma-equations}
In the setup of \ref{setup for corresponding measures} and with the measures $\mu,\mu_\infty,\mu_v$ from \ref{related toric measures-new}, the maps $\varphi \mapsto \varphi_w \coloneqq \varphi \circ \tropbar_w$ for $w \in \{\infty,v\}$ induce bijections between
\begin{enumerate}
\item\label{tot-deg-ma-corr-i}
the set of functions $\varphi\in\PSH(N_\R/\Lambda,\omega)\setminus \{-\infty\}$ such that $(\omega+d'd''\varphi)^{\wedge n}=\mu$,
\item\label{tot-deg-ma-corr-ii}
the set of functions $\varphi_\infty\in\PSH(A_\infty^\an,\omega_\infty)^{\SS_\infty}\setminus \{-\infty\}$ such that $(\omega_\infty+dd^c\varphi_\infty)^{\wedge n}=\mu_\infty$,
\item\label{tot-deg-ma-corr-iii}
the set of functions 
		$\varphi_v\in \PSH(A_v^\an,\omega_v)^{\SS_v}\setminus \{-\infty\}$
		such that $(\omega_v+d'd''\varphi_v)^{\wedge n}=\mu_v$.
	\end{enumerate}
\end{prop}

It follows from \ref{finiteness-remark} and \ref{theta-psh functions on abelian varieties} that the functions $\varphi,\varphi_\infty,\varphi_v$ are finite and continuous, hence the Monge--Amp\`ere measures in \ref{tot-deg-ma-corr-i}--\ref{tot-deg-ma-corr-iii} are defined as in the first part of \S \ref{sec:monge-ampere-equat}.

\begin{proof}
We have seen in Proposition \ref{lambda-invariance-cor} that  we have an isomorphism 
$$	\PSH(N_\R/\Lambda, \omega)
\stackrel{\sim}{\longrightarrow} \PSH(A_w^{\rm an}, \omega_w)^{\SS_w}, \quad \varphi \longmapsto \varphi_w \coloneqq \varphi \circ \tropbar_w$$
of cones for  $w \in \{\infty,v\}$. It remains to see that $\varphi$ is a solution of the tropical Monge--Amp\`ere equation in   \ref{tot-deg-ma-corr-i} if and only if $\varphi_w \coloneqq \varphi \circ \tropbar_w$ is a solution of the Monge--Amp\`ere equation in \ref{tot-deg-ma-corr-ii} or \ref{tot-deg-ma-corr-iii}, respectively. This can be checked
locally  on an open subset $W$ of $N_\R/\Lambda$ and on the corresponding open subsets $W_w \coloneqq \tropbar_w^{-1}(W)$ of $A_w^{\rm an}$. Then we may assume that  the quotient homomorphism $\pi\colon N_\R \to N_\R/\Lambda$ maps an open subset $U$ isomorphically onto $W$ and hence the open subset $U_w \coloneqq \trop_w^{-1}(U)$ of the Tate uniformization $\torus_w^{\rm an}$ is mapped isomorphically onto $W_w$ by $\pi_w$. We have noticed in \ref{canonical form for tropical polarized abelian variety} that the canonical form $\theta \coloneqq \pi^*(\omega) \in A^{1,1}(N_\R)$ fits into the setup of \S \ref{subsection: Lagerberg setting}. In this toric situation, Proposition \ref{local MA equation} shows that a continuous finite $\theta$-psh function $f$ on $U$ is a solution of $(\theta+d'd''f)^{\wedge n} = \tilde\mu$ if and only if $f_w \coloneqq f \circ \trop_w$ is a solution of $(\theta_w+d'd''f_w)^{\wedge n} = \tilde\mu_w$ for the canonical form $\theta_w \coloneqq \trop_w^*(\theta) \in A^{1,1}(\torus_w^{\rm an})$. Applying this to the lift $f$ of $\varphi|_W$ to $U$ and using that $\pi_w^*(\omega_w) = \theta_w$, we deduce that 
$(\omega+d'd''\varphi)^{\wedge n}=\mu$ if and only if $(\omega_w+d'd''\varphi_w)^{\wedge n}=\mu_w$.
\end{proof}

\begin{prop} \label{complex solution is toric}
	Let $(A_\infty,H,U_1,U_2)$ be a decomposed polarized complex abelian variety of dimension $n$ with canonical form  $\omega_\infty \in A^{1,1}(A_\infty^\an)$. 
	Let $\mu_\infty$ be an $\SS_\infty$-invariant positive Radon measure on  $A_\infty^{\rm an}$  with respect to the $\SS_\infty$-action on $ A_\infty^\an$ from \ref{theta-psh functions on abelian varieties} and with  
	\begin{equation}\label{normalize-measure2'}
	\mu_\infty(A_\infty^\an)=n!\,\sqrt{\deg_{H}(A_\infty^\an)}. 
	\end{equation}
	Then 
	the complex Monge--Amp\`ere equation 
	\[
	(\omega_\infty+dd^c\varphi_\infty)^{\wedge n}=\mu_\infty
	\]
	is solved by a continuous finite  $\omega_\infty$-psh function 
	$\varphi_\infty$ on $A_\infty^\an$, 
	unique up to adding a constant. 
	Any such solution  is $\SS_\infty$-invariant.
\end{prop}

\begin{proof}
	Existence and uniqueness of a solution is known from complex K\"ahler geometry (see Theorem \ref{thm:4}).   Here, we use that $\omega_\infty$ is a K\"ahler form and that $\mu_\infty$ does not charge pluripolar sets. 
	To see the latter, we apply Lemma \ref{no-charge-pluriplolar-sets-toric} by lifting the positive Radon measures $\mu_\infty$ and $\mu \coloneqq \tropbar_\infty(\mu_\infty)$ locally to the Tate uniformization $\torus_\infty^\an$ with cocharacter lattice $N$ and to $N_\R \to N_\R/\Lambda$, respectively. The canonical form  $\omega_\infty$ and the measure $\mu_\infty$  is invariant under $\SS_\infty$. It follows now from uniqueness of the solution up to adding constant that the solutions are $\SS_\infty$-invariant. By \ref{finiteness-remark},
	 the function  
	$\varphi_\infty$ is finite and continuous.
\end{proof}

\begin{cor}\label{tropical-ma}
	Let $(\Lambda,M, [\cdot , \cdot],\lambda)$ be a polarized 
	tropical abelian variety with canonical form $\omega$ on the associated tropical abelian variety $N_\R/\Lambda$. We assume  the condition  
\begin{equation}\label{normalize-measure'}
\mu(N_\R/\Lambda)= n!\,
\sqrt{\deg(\Lambda,M, [\cdot , \cdot],\lambda)}. 
\end{equation}
for a positive Radon
measure $\mu$ on  $N_\R/\Lambda$.   Then
	the tropical Monge--Amp\`ere equation 
	$$(\omega+d'd''\varphi)^{\wedge n}=\mu$$
	 is solved by a 
	continuous $\omega$-psh function 
	$\varphi\colon N_\R/\Lambda \to \R$, unique up to constants. 
\end{cor}

\begin{proof}
	By Proposition \ref{complex-trop-surjective}, there is a polarized decomposed complex abelian variety $(A_\infty,H,U_1,U_2)$ with $\trop_\infty(A_\infty,H,U_1,U_2)=(\Lambda,M, [\cdot , \cdot],\lambda)$. Let $\mu_\infty$ be the unique positive $\SS_\infty$-invariant Radon measure on $A_\infty^\an$ with $\trop_\infty(\mu_\infty)=\mu$. We have seen in \ref{related toric measures-new} that \eqref{normalize-measure'} is equivalent to \eqref{normalize-measure2'}. By Proposition \ref{complex solution is toric}, the complex Monge--Amp\`ere equation 
	$$	(\omega_\infty+dd^c\varphi_\infty)^{\wedge n}=\mu_\infty$$
	has a solution $\varphi_\infty \in \PSH(A_\infty^\an,\omega_\infty)$, unique up to adding a constant. Now the claim  follows from Proposition \ref{equivalent-ma-equations}.
\end{proof}

\begin{cor}\label{non-arch-ma-tot-deg-av'}
	Let $(A_v,{\phi})$ be a polarized totally degenerate split 
	abelian variety  over $K$ with canonical form $\omega_v \in A^{1,1}(A_v^\an)$. Let $\mu_v$ be a positive Radon measure supported on the canonical skeleton of $A_v^\an$ with 
	\begin{equation}\label{normalize-measure33}
	\mu_v(A_v^\an)=n!\,\sqrt{\deg_{{\phi}}(A_v)}.
	\end{equation}
	Then there is a  $\varphi_v \in \PSH(A_v^\an,\omega_v)^{\SS_v}$, unique up to adding constants,  with  
	\begin{equation}\label{non-arch-ma-tot-deg-av-eq}
	(\omega_v+d'd''\varphi_v)^{\wedge n}=\mu_v.
	\end{equation} 
	All such solutions $\varphi_v$ are continuous and finite.
\end{cor}

\begin{proof}
	This follows from Proposition \ref{equivalent-ma-equations} 
	and {Corollary \ref{tropical-ma}}
\end{proof}

\section{Monge--Amp\`ere equations on arbitrary non-archimedean  abelian varieties} \label{section: short variant}

In this section, $K$ denotes a non-archimedean field ${(K,|\phantom{a}|)}$ with valuation ${v\coloneqq -\log |\phantom{a}|}$ and value group $\Gamma \coloneqq v(K^\times) \subset \R$. 
We will extend the solvability of the Monge--Amp\`ere equation for totally degenerate split abelian varieties from Section \ref{section: totally degenerate abelian varieties} to arbitrary abelian varietes over $K$. 
Instead of using $\theta$-psh functions, we will use in this section the equivalent language of psh metrics. 
Furthermore we will show that the solution is semipositive in the sense of Zhang which is then unique up to scaling by a result of Yuan and Zhang.

In \S \ref{short subsection: Raynaud uniformization and canonical tropicalization}--\S \ref{short subsection: MA equations on abelian varieties},
we assume that $K$ is algebraically closed and non-trivially valued and use the same strategy as in the totally degenerate case to solve the non-archimedean Monge--Amp\`ere equation.
Semipositivity will follow from a piecewise  linear approximation of the tropical solution and then by relying  on the results in \cite{gubler-compositio} about Mumford models. 
In \S \ref{subsection: descent}, we give a descent argument to deduce our main result over any non-archimedean $K$. 

\subsection{Raynaud uniformization and canonical tropicalization} 
\label{short subsection: Raynaud uniformization and canonical tropicalization}
We assume that the non-archimedean field $K$ is non-trivially valued and algebraically closed. 
We can associate to an abelian variety $A$ over $K$ in a functorial way an exact sequence of algebraic groups 
\[
0 \longrightarrow \torus \longrightarrow E 
\stackrel{q}{\longrightarrow} B \longrightarrow 0
\]
called the \emph{Raynaud extension of $A$} and a discrete subgroup $\Lambda$ of $E(K)$ such that $A^{\an}= E^{\an}/\Lambda$, i.e.~$E^{\an}$ is the uniformization of $A^{\an}$. 
Here $B$ is an abelian variety of good reduction over $K$ and $\torus$ is a  torus over $K$ (see \cite{bosch-luetkebohmert1991} and \cite[6.5]{berkovich-book}).
The quotient homomorphism $p\colon E^{\an} \to A^{\an}$ is only an analytic morphism, but the Raynaud extension is algebraic.

\begin{rem} \label{short canonical tropicalization}
We denote the character lattice of the  torus $\torus$ by $M$ and 
set $N \coloneqq \Hom_\Z(M,\Z)$. 
For $u \in M$, the pushout of the Raynaud extension with respect to 
the character  $\chi_u\colon \torus \to \mathbb G_m$ gives rise to a 
translation invariant
rigidified $\mathbb G_m$-torsor over $B$ and hence to 
a rigidified translation invariant line bundle $E_u$ on $B$ (see 
\cite{bosch-luetkebohmert1991} and \cite[3.2]{foster-rabinoff-shokrieh-soto}). 
Note that the line bundle $q^*(E_u)$ is trivial over $E$ and the above 
pushout construction gives a canonical frame $e_u\colon E \to q^*(E_u)$. 
Following \cite[3.2]{foster-rabinoff-shokrieh-soto}), we define 
the \emph{canonical tropicalization map of $E$} as the unique map
\[
\trop_v\colon E^{\an} \longrightarrow N_\R
\] 
satisfying
\[
\langle \trop_v(x), u \rangle = - \log q^*\|e_u(x)\|_{E_u}
\quad\quad(x\in E^\an, u\in M),
\]
where $\|\phantom{a}\|_{E_u}$ is the canonical metric 
of the rigidified line bundle $E_u$. 
The canonical tropicalization map agrees with the classical 
tropicalization map on $\torus^{\an}$ and maps $\Lambda$ onto 
a complete lattice in $N_\R$ 
(see \cite[Theorem 1.2]{bosch-luetkebohmert1991}). 
By abuse of notation, we will identify $\Lambda$ with the lattice 
$\trop_v(\Lambda)$ in $N_\R$. 
Then $\trop_v$ induces  a continuous map
\[
\overline{\trop}_v\colon A^{\an} \longrightarrow N_\R/ \Lambda
\]
which is called \emph{the canonical tropicalization of the abelian variety $A$}.
\end{rem}

\begin{rem} \label{short canonical skeleton}
The canonical tropicalization map $\trop_v\colon E^{\an} \to N_\R$ 
admits a canonical section $\iota_v\colon N_\R\to {E^\an}$ which identifies 
$N_\R$ with a closed subset of $E^{\an}$ called the 
\emph{canonical skeleton of $E$}, such that $\trop_v$ 
is a strong deformation retraction. Hence $E^{\an}$ is contractible (see \cite[6.5]{berkovich-book}).  
Passing to the quotient, we see that there is a canonical section 
$\overline{\iota}_v$ of $\overline{\trop}_v$ which identifies 
$N_\R/\Lambda$ with a closed subset of $A^{\an}$  
called the \emph{canonical skeleton of $A$} such that $\overline\trop_v$ 
is a strong deformation retraction onto $N_\R/\Lambda$.
\end{rem}

In the following, we consider a polarized abelian variety $(A,\phi)$ over $K$ of dimension $g$. 
Recall that its degree $\deg(A,\phi)$ is defined as the degree of the isogeny $\phi{\colon A\to A^\vee}$.

\begin{rem} \label{short polarization and tropical pairing}
	Using the above notations, we have a natural bilinear pairing
	\[
	\Lambda \times M \longrightarrow \Z, \quad (a,u) \longmapsto [a,u] 
	\coloneqq \langle  \trop_v(a), u \rangle.
	\]
Using the lift of $\phi$ to Raynaud uniformizations, 
it is shown in \cite{bosch-luetkebohmert1991},
\cite{foster-rabinoff-shokrieh-soto} that $\phi$ induces a 
homomorphism $\lambda\colon \Lambda \to M$ and that 
$(\Lambda,M, [\cdot , \cdot],\lambda)$ is a $\Gamma$-rational polarized tropical 
abelian variety as defined in \S \ref{subsection:Polarized tropical
  abelian varieties}. We call $ \trop_v(A,\phi)\coloneqq (\Lambda,M,
[\cdot , \cdot],\lambda)$ 
the \emph{tropicalization of the polarized abelian variety $(A,\phi)$}. 
	This induces a functor
	\begin{equation}
	\label{functor-tropicalization-non-archimedean-abelian-variety-eq1'}
	{\trop_v}\colon
	\begin{pmatrix}
	\mbox{polarized abelian}\\
	\mbox{varieties over $K$}
	\end{pmatrix}
	{\longrightarrow}
	\begin{pmatrix}
	\mbox{{$\Gamma$-rational} polarized}\\
	\mbox{tropical abelian varieties}
	\end{pmatrix}
	\end{equation}
	extending the tropicalization functor  \eqref{functor-tropicalization-non-archimedean-abelian-variety-eq1}.
	Observe that the rank $n$  of $M$ might be smaller than $g = \dim(A)$
	and  equality occurs precisely in the totally degenerate case.    
\end{rem}

As $K$ is algebraically closed, there is an ample line bundle $L$ on 
$A$ such that $\phi=\phi_L$ is the isogeny $\phi_L\colon A \to A^\vee$ induced by $L$. 
By \cite[Proposition 6.5 and Theorem 6.13]{bosch-luetkebohmert1991}, there is an ample line bundle $H$ on $B$ with $q^*(H^\an) \cong p^*(L^\an)$ on $E^{\an}$. Let $\chi(A,L)$ and $\chi(B,H)$ be
the Euler characteristics.

\begin{prop} \label{compare degree of polarizations wrt tropicalizations}
Under the above assumptions, we have $\deg(A,\phi)= \chi(A,L)^2$ and 
\[
\chi(A,L) = \frac{\deg_{L}(A)}{g!}=\sqrt{\deg(\trop_v(A,\phi))}\cdot \chi(B,H)
= \sqrt{\deg(\trop_v(A,\phi))}
\cdot \frac{\deg_H(B)}{(g-n)!}.
\]
\end{prop}

\begin{proof}
This follows from the Riemann--Roch theorem and the vanishing theorem on abelian varieties \cite[III.16]{mumford1970} and the dimension formula in	\cite[Theorem 6.13]{bosch-luetkebohmert1991}.
\end{proof}

\subsection{Mumford models} \label{short subsection: Mumford models}
We  recall the construction of Mumford models from \cite[\S 4]{gubler-compositio}. 
We assume that the reader is familiar with the notation and results in \cite{gubler-compositio}. 
In particular, we will use admissible formal models and formal analytic structures as recalled in \cite[\S 2]{gubler-compositio}. 
If the reader is unfamiliar with such models, then he can replace them often by algebraic models using the algebraic approach to toric schemes given in \cite{gubler-guide}. 

We use the same assumptions and notation as in Subsection \ref{short subsection: Raynaud uniformization and canonical tropicalization}.
Since $K$ is algebraically closed and the valuation $v$ is non-trivial, the value group $\Gamma \coloneqq v(K^\times)$ is dense in $\R$.
Again, we consider an abelian variety $A $ over $K$ with Raynaud extension
\[
0 \longrightarrow \torus  \longrightarrow E  
\stackrel{q}{\longrightarrow} B  \longrightarrow 0.
\]
Let $g$ be the dimension of $A $ and $n$ the rank of the torus $\torus$. 
The abelian variety $B $ over $K$ is the generic fiber of an abelian scheme $\Bcal$ over $\kcirc$. 
For any  scheme $\mathscr X$ over $K^\circ$,  we denote $\widehat{\Xcal}$ the formal completion along the special fiber and 
by $\mathscr X_s$ its special fibre over the residue field $\widetilde K=K^\circ/K^{\circ\circ}$. 

We first fix the terminology from convex geometry expanding the conventions from \S \ref{section: notation}.

\begin{art} \label{conventions from convex geometry}   
\label{definition: locally piecewise affine functions}
A \emph{polytopal decomposition} $\Pi$ of $N_\R$ is a locally finite set $\Pi$ of polytopes in $N_\R$ 
such that the set $\Pi$ contains with a polytope also all its faces,
polytopes in $\Pi$ intersect only in common faces, and the support condition $\bigcup_{\Delta \in \Pi} \Delta = N_\R$ holds.
The polytopal decomposition $\Pi$ is called \emph{integral $\Gamma$-affine} if every polytope $\Delta \in \Pi$ is  integral 
	$\Gamma$-affine (often also called \emph{$\Gamma$-rational}). 
The \emph{star of $\Pi$ in  $\omega \in N_\R$} is the
fan $\Sigma\coloneqq \{\sigma_\Delta\,|\,\Delta\in \Pi\mbox{ with }
\omega \in \Delta\}$
where $\sigma_\Delta$ is the cone in $N_\R$ generated by 
$\Delta-\omega$.

A function $f\colon {N_\R} \to \R$ is called \emph{piecewise affine with respect to a polytopal decomposition $\Pi$} if for any polytope $\Delta \in \Pi$ there are $u_\Delta \in M_\R = \Hom(N,\R)$ and $c_\Delta \in \R$ such that $f=u_\Delta + c_\Delta$ on $\Delta$. 
Such a function $f$ is called \emph{piecewise $\Gamma$-affine} if $\Pi$ is integral $\Gamma$-affine and $c_\Delta \in \Gamma$ for all $\Delta \in \Pi$. We say that $f$ has \emph{integral (resp.~rational) slopes} if $u_\Delta \in M$ (resp.~$u_\Delta \in M_\Q$) for all $\Delta \in \Pi$.
The \emph{recession function} of $f$ in $\omega$ is the unique function $g\colon N_\R \to \R$ which is piecewise linear with respect to the star of $\Pi$ in $\omega$ 
and satisfies $g(x-\omega)= u_\Delta(x)$ for each $\Delta \in \Pi$ and each $x\in \Delta$.

We call $f\colon N_\R \to \R$ a \emph{locally piecewise affine function} if $f$ is piecewise affine with respect to a polytopal decomposition $\Pi$ of $N_\R$. 
Here, the term ``locally'' emphasizes the fact that a polytopal decomposition is locally a polyhedral complex and hence a locally piecewise affine function is locally a piecewise affine function in the sense of our conventions in  \S \ref{section: notation}.
\end{art}

\begin{rem} \label{short Mumford models}	
Mumford's construction associates to an integral $\Gamma$-affine polytopal decomposition $\Pi$  of $N_\R$ an admissible formal $\kcirc$-model $\hat\Ecal_\Pi$ of $E$ given by the formal analytic atlas $\{\trop_v^{-1}(\Delta) \mid \Delta \in \Pi\}$ on $\trop^{-1}(N_\R)=(\hat\Ecal_\Pi)^\an$ (see \cite[\S 4]{gubler-compositio} for details). 
	
The irreducible components of the special fiber of $\hat\Ecal_\Pi$ are in bijective correspondence to the vertices of $\Pi$. 
In fact, the irreducible component $Y_\omega$ corresponding to the vertex $\omega$ of $\Pi$ is a fiber bundle over $\Bcal_s$. 
The torus $\torus_{\ktilde} \coloneqq \Spec(\ktilde[M])$ acts naturally on the fibers of $Y_\omega$ over $\Bcal_s$ and each  fiber is $\torus_{\ktilde}$-equivariantly isomorphic to the $\torus_{\ktilde}$-toric variety $Y_\Sigma$ associated to the star $\Sigma$ of $\Pi$ in $\omega$ \cite[Proposition 4.8]{gubler-compositio}.
The fiber over zero is even canonically isomorphic to the $\torus_{\ktilde}$-toric variety $Y_\Sigma$. 
It follows in particular that $Y_\omega$ is proper over $\Bcal_s$.
	
We call $\hat\Ecal_\Pi$ the \emph{formal Mumford model} of $\trop_v^{-1}(|\Pi|)$ associated to $\Pi$. 
In fact, using toric schemes over $\kcirc$ from \cite[\S 7]{gubler-guide}, one can construct (again similar as in \cite[\S 4]{gubler-compositio}) a flat algebraic model $\Ecal_\Pi$ of $E $ with formal completion $\hat\Ecal_\Pi$. 
Note that $q$ extends uniquely to morphisms $\hat\Ecal_\Pi \to \hat\Bcal$ and $\Ecal_\Pi \to \Bcal$ which we denote by $q$ as well.

Let  $f\colon N_\R \to \R$ be a piecewise $\Gamma$-affine function with respect to $\Pi$. 
If $f$  has integral slopes, then $f$ defines a formal model $\Ocal_{\hat\Ecal_\Pi}(f)$ of the trivial line bundle $\Ocal_{\trop_v^{-1}(N_\R)}$ living on $\hat\Ecal_\Pi$ which is determined by 
\begin{equation} \label{short piecewise linear functions and line bundles}
	f \circ \trop_v = -\log \bigl(\| 1 \|_{\Ocal_{\hat\Ecal_\Pi}(f)}\bigr),
\end{equation}
where $\metr_{\Ocal_{\hat\Ecal_\Pi}(f)}$ denotes the model metric on $\Ocal^\an_{\trop_v^{-1}(N_\R)}$ induced by $\Ocal_{\hat\Ecal_\Pi}(f)$.
\end{rem}

In the following, we consider a piecewise $\Gamma$-affine function $f\colon N_\R \to \R$ with respect to a polytopal decomposition $\Pi$ of $N_\R$ and we assume that $f$ has integral slopes. 
We fix a vertex $\omega$ of $\Pi$ with associated irreducible component $Y_\omega$ of  the special fiber of $\hat\Ecal_\Pi$ (see Remark \ref{short Mumford models}). 
The function $f$  is called 
\emph{convex in $\omega$} if $f$ is convex in 
a neighbourhood of $\omega$
or equivalently if the recession function $g$ of $f$ in $\omega$ is convex on $N_\R$.

\begin{lemma} \label{short globally generated}
	Let $\omega$ be a vertex of $\Pi$. 
	Then the function $f$ is convex in $\omega$ if and only if 
	the line bundle
	$\Ocal_{\hat\Ecal_\Pi}(f)|_{Y_\omega}$ 
	on the proper scheme $Y_\omega$ is nef. 	
\end{lemma}

\begin{proof} 
It follows from Remark \ref{short canonical tropicalization} that any $u\in M$ induces  a rigidified translation invariant line bundle $E_u$ on $B $ which has a model $\Ecal_u$ on $\Bcal$ algebraically equivalent to zero.
	
For any piecewise $\Gamma$-affine function $h$ with respect to $\Pi$ such that $h$ has integral slopes, the line bundle $\Ocal_{\hat\Ecal_\Pi}(h)$ has a unique meromorphic section $s_h$ extending the constant section $1$ from the generic fiber. 
This construction is linear in $h$. 
If $h$ is equal to some $u\in M$ (resp.~equal to some $c\in \Gamma$), then $\Ocal_{\hat\Ecal_\Pi}(h)$ is numerically trivial as it is isomophic to the pull-back of $\Ecal_u$ (resp.~a trivial line bundle).

Suppose first that $\Ocal_{\hat\Ecal_\Pi}(f)|_{Y_\omega}$ is nef. 
We have seen in Remark \ref{short Mumford models} that $Y_\omega \to \Bcal_s$ is a fiber bundle with fiber isomorphic to the proper toric variety $Y_\Sigma$ where the fan $\Sigma$ is given as the star of $\Pi$ in $\omega$. 
The restriction of $\Ocal_{\hat\Ecal_\Pi}(f)$ to 
the fiber of $Y_\omega \to \Bcal_s$ over zero leads to a line bundle  on $Y_\Sigma$ corresponding to the recession function $g$ of $f$ in $\omega$. 
This restriction  is also nef and hence $g$ is convex (see \cite[Theorem 6.3.12]{cox-little-schenck} or  \cite[\S 3.4]{fulton-toric-varieties}, but note that in these books concave functions are called convex and hence $-g$ is used there).

Conversely, let $f$ be convex in $\omega$.  
We have to show that $\Ocal_{\hat\Ecal_\Pi}(f)|_{Y_\omega}$ is nef. 
Let $P$ be the union of all  polytopes  of $\Pi$  with vertex $\omega$. 
Then there is a unique formal open subset $\Ucal$ of $\hat\Ecal_\Pi$ with $\trop^{-1}(P)=\Ucal^\an$. 
Note that $Y_\omega$ is also an irreducible component of $\Ucal_s$.
		 
For any $\Delta \in \Pi$ with vertex $\omega$, there is $c \in \Gamma$ and $u \in M$ such that $f=u+c$ on $\Delta$. 
For  the formal metric $\metr$ induced by $\Ocal_{\hat\Ecal_\Pi}(f)$, the convexity of $f|_P$ yields that $\|s_{f-u-c}\| \leq 1$.
Hence  $s_{f-u-c}|_\Ucal$ is a global section of $\Ocal_{\hat\Ecal_\Pi}(f-u-c)|_\Ucal$  which is nowhere vanishing over the formal open subset $\Ucal_\Delta$ of $\Ucal$ with  $\Ucal_\Delta^\an=\trop_v^{-1}(\Delta)$.  
Now let $C$ be any closed curve in $Y_\omega$. 
We pick any $y \in C$. 
By \cite[Proposition 2.17]{gubler-rabinoff-werner2} there exists some
$x \in \Ucal^\an$ with reduction $y$. 
We choose $\Delta \in \Pi$ with $\trop_v(x) \in \Delta$. 
The above shows that the restriction of the global section $s_{f-u-c}|_\Ucal$ to $C$ does not vanish in $y$.
Hence the intersection number 
\[
c_1(\Ocal_{\hat\Ecal_\Pi}(f)).C
=c_1(\Ocal_{\hat\Ecal_\Pi}(f-u-c)).C
+c_1(\Ocal_{\hat\Ecal_\Pi}(u)).C+c_1(\Ocal_{\hat\Ecal_\Pi}(c)).C
= \div(s_{f-u-c}).C 
\]
is non-negative. This proves the claim. 
\end{proof}

Let $F$ be a line bundle on $E^{\an}$ with $F=q^*(H)$ 
for a rigidified line bundle $H$ on $B^\an$. 
Then $H$ has a unique rigidified model $\Hcal$ on the formal 
abelian scheme $\hat\Bcal$ and 
\[
q^*\Hcal(f) \coloneqq q^*\Hcal \otimes \Ocal_{\hat\Ecal_\Pi}(f)
\]
is a model of $F$ living on the formal Mumford model $\hat\Ecal_\Pi$.

\begin{prop} \label{short degree of irreducible component}
Consider  $F=q^*(H)$ as above. 
Let $f\colon N_\R \to \R$ be a piecewise $\Gamma$-affine function with respect to the polytopal decomposition $\Pi$ of $N_\R$. 
We assume that $f$ has integral slopes and is convex in the given vertex $\omega$ of $\Pi$. 
We consider the dual polytope 
\[
\{\omega\}^f\coloneqq \{u \in M_\R \mid 
u(x-\omega) \leq f(x)-f(\omega) \,\text{for all $x$ in a neighborhood of $\omega$} \}
\]
of the star $\Sigma$  in $\omega$ where $M=\Hom_\Z(N,\Z)$.  
Then we have 
\[
\deg_{q^*\Hcal(f)}(Y_\omega)= \frac{g!}{(g-n)!} \cdot \deg_H(B ) \cdot \vol(\{\omega\}^f)
\]
	where $\vol$ is the Haar measure on $M_\R$ such that the lattice $M$ has covolume one. 
\end{prop}

\begin{proof}  
	The proof is similar as in \cite[Proposition 5.18]{gubler-compositio}. 
	We use the canonical strata of a scheme as defined in \cite[\S 2]{berkovich-1999}.
	There is a bijective correspondence
	between the strata $S$ of 
	the special fiber of the Mumford model $\hat\Ecal_\Pi$ and the open faces $\tau$ of $\Pi$ given by
	$S={\rm red}(\trop_v^{-1}(\tau))$ where ${\rm red}$ is the reduction
	map of $\hat\Ecal_\Pi$ \cite[Proposition 4.8]{gubler-compositio}.
	
	Similarly as  in the proof of Lemma \ref{short globally generated},
	the 
	intersection product $c_1(\Ocal_{\hat\Ecal_\Pi}(f))^r\cdot Y_\omega$
	can be 
	performed for any $r \leq g$ by using proper intersections with
	suitable Cartier divisors $\div(s_{f-u-c})$ and hence can be
	represented by a positive linear combination of strata closures in
	$Y_\omega$. 
The above correspondence shows that there are no strata $S$ of $Y_\omega$ with $\codim(S,Y_\omega) >n$.
Then it follows from the projection formula with respect to $q\colon Y_\omega \to \Bcal_s$ that 
\begin{equation} \label{projection formula applied}
\deg_{q^*\Hcal(f)}(Y_\omega)= \binom{g}{n}  c_1(\Hcal)^{g-n} . 
q_*\Bigl( c_1\bigl(\Ocal_{\hat\Ecal_\Pi}(f)\bigl)^{n} \cdot Y_\omega \Bigr).
\end{equation}
For any closed point $b$ of $\Bcal_s$, the fibre $q^{-1}(b)$ is 
$\T_{\tilde{K}}$-equivariantly isomorphic to the toric variety $Y_\Sigma$ over $\ktilde$ associated to the star $\Sigma$ of $\Pi$ in $\omega$. Along this isomorphism, the restriction of $\Ocal_{\hat\Ecal_\Pi}(f)$ to $q^{-1}(b)$  corresponds to the line bundle $\Ocal_{Y_\Sigma}(g)$ for the recession function $g$ of $f$ in $\omega$. 
Hence the degree of $q^{-1}(b)$ with respect to $\Ocal_{\hat\Ecal_\Pi}(f)$ is equal to $n!\cdot\vol(\{\omega\}^f)$ \cite[\S 5.3]{fulton-toric-varieties}. 
Using that the intersection product $c_1(\Ocal_{\hat\Ecal_\Pi}(f))^{n} \cdot Y_\omega$  is represented by a positive linear combination of strata closures and that every stratum of $\hat\Ecal_s$ maps onto $\Bcal_s$, we may compute the intersection product locally over $\Bcal_s$, say over an open subset $\Wcal$ of $\hat\Bcal$  which trivializes $\Ecal$ and hence $Y_\omega$ is over $\Wcal_s$ isomorphic to $Y_\Sigma \times \Wcal_s$. 
Along this isomorphism, the line bundle $\Ocal_{\hat\Ecal_\Pi}(f)|_{q^{-1}(\Wcal_s)}$ is given by pull back of $\Ocal_{Y_\Sigma}(g)$. This shows immediately that
	\[
	q_* \bigl( c_1(\Ocal_{\hat\Ecal_\Pi}(f))^{n} \cdot Y_\omega \bigr) 
	= n!\cdot\vol(\{\omega\}^f) \Bcal_s
	\]
	and hence the claim follows from \eqref{projection formula applied} and $\deg_{\Hcal}(\Bcal_s)=\deg_H(B)$. 
\end{proof}

\begin{rem}\label{periodicity and M-linearization}
	A line bundle $F$ on $E ^{\an}$ descends to $A ^{\an}$ 
	if and only if $F$ admits a 
	$\Lambda$-linearization over the action of $\Lambda$ on $E^\an$. 
	In this case, we have $F=p^*(L^\an)$ for the line bundle 
	$L^\an=F/\Lambda$ on $A^{\an}$. 
By \cite[Proposition 6.5]{bosch-luetkebohmert1991}, for any rigidified line  bundle $L$ on $A $ there is a rigidified line bundle $H$ on $B ^{\an}$ such that $p^*(L^\an) \simeq q^*(H)$ as $\Lambda$-linearized cubical sheaves. 
	The line bundle $H$ is unique up to a tensor product with 
	a line bundle $E_u$ for some $u \in M$. 
	Using the above isomorphism for identification, 
	it is shown in \cite[4.3]{gubler-compositio} that the 
	$\Lambda$-linearization
	yields a canonical cocycle 
	$(z_\lambda\colon N_\R \rightarrow \R)_{\lambda \in \Lambda}$ and a 
	canonical symmetric bilinear form 
	$b\colon \Lambda \times \Lambda \rightarrow \R$ associated to $L$ such that 
	\begin{equation} \label{cocycle and b}
	z_\lambda(\omega) = z_\lambda(0) + b(\omega,\lambda) \quad 
	\forall\,\omega \in N_\R, \lambda \in \Lambda.
	\end{equation}
	and
	\begin{equation}
	\label{eq:9}
	z_\lambda(0) \in \Gamma  \quad \forall \lambda \in \Lambda .
	\end{equation}
	Moreover, the cocycle condition shows that 
	$\lambda\mapsto z_\lambda(0)$ is a quadratic function 
	on $\Lambda$ with associated bilinear form $b$. 
	The line bundle $L$ is ample if and only if $H$ is ample and $b$ is 
	positive definite \cite[Theorem 6.13]{bosch-luetkebohmert1991}. 
	In this case, $b$ is the bilinear form of 
	the polarized tropical abelian variety $\trop_v(A ,\phi_L)$ 
	considered 
	in \eqref{functor-tropicalization-non-archimedean-abelian-variety-eq1'} and 
	\S \ref{subsection:Polarized tropical abelian varieties}. 
	Moreover, $H$ is the generic fiber of a unique rigidified ample line 
	bundle $\Hcal$ on $\hat\Bcal$. 
	
A polytopal decomposition $\Pi$  of $N_\R$ is called \emph{$\Lambda$-periodic} if for all $\Delta \in \Pi$ and for all $\lambda \in \Lambda\setminus\{0\}$ the polytope $\Delta + \lambda$ is a face of $\Pi$ disjoint from $\Delta$. 
For a $\Lambda$-periodic integral $\Gamma$-affine polytopal decomposition $\Pi$ of $N_\R$, the following facts are shown in \cite[4.3]{gubler-compositio}:
	
\begin{enumerate}
\item\label{item-1}
The  Mumford model $\hat\Ecal_\Pi$ is a formal model of $E^{\an}$ over $\kcirc$ and $\Acal_\Pi \coloneqq \hat\Ecal_\Pi/M$ is a formal model of $A^{\an}=E^{\an}/M$ over $\kcirc$.  
Furthermore the formal models $\hat\Ecal_\Pi$ and  $\Acal_\Pi$ are locally isomorphic.
\item\label{item-2}
Let $f\colon N_\R \to \R$ be a piecewise $\Gamma$-affine function with respect to $\Pi$. 
Assume that $f$ has integral slopes and let $\Fcal\coloneqq q^*\Hcal(f)$. 
If $f$ satisfies the automorphy condition 
\begin{equation} \label{new cocycle condition for f}
f(\omega + \lambda)= f(\omega) + z_\lambda(\omega)\quad
\forall \,\omega \in N_\R, \lambda \in \Lambda,
\end{equation}
then there is a unique line bundle $\Lcal$ on $\Acal_\Pi$ with $p^*(\Lcal)=\Fcal$.
\end{enumerate}
\end{rem}

\begin{prop} \label{new periodic approximation}
Every convex function $f\colon N_\R \to \R$  satisfying \eqref{new cocycle condition for f} is a uniform limit of locally finite piecewise $\Gamma$-affine convex functions with rational slopes satisfying also \eqref{new cocycle condition for f}. 
\end{prop}

\begin{proof}
	    Convexity,  \eqref{new cocycle condition for f} and \eqref{cocycle and b} yield that $f(\omega )$ grows like
		the positive definite quadratic form $b(\omega ,\omega )$. Fix $\varepsilon >0$ and let $p\in N_{\R}$. By
		the quadratic growth of $f$, the convex set
		\begin{displaymath}
		C_{p}=\{x\in M_{\R}\mid f(\omega )\ge f(p)+x(\omega -p)-  \varepsilon/{3}
	     \: \forall \omega \in N_{\R}\}
		\end{displaymath}
has nonempty interior and hence there is a point $m_{p}\in C\cap M_{\Q}$. 
Choose a value $c_{p}\in \Gamma $ with
\begin{displaymath}
f(p)-m_{p}(p)\ge c_{p} +\varepsilon /3 \ge f(p)-m_{p}(p)-\varepsilon /3.
\end{displaymath}
This is possible because the value group $\Gamma$ is dense in $\R$. 
Then the function $h_{p} \coloneqq c_{p}+m_{p}$ is $\Gamma$-affine with rational slope and satisfies
\begin{displaymath}
h_{p}(p)=c_{p}+m_{p}(p)\le f(p)-\varepsilon /3.
\end{displaymath}
Since $m_{p}\in C_{p}$, this implies that $h_{p}\le f$ on $N_{\R}$. Moreover, 
		\begin{displaymath}
		h_{p}(p)=c_{p}+m_{p}(p)\ge f(p)-2\varepsilon /3.
		\end{displaymath}
		By  continuity of $h_{p}$ and $f$, there is an
		open neighborhood $U_{p}$ of $p$ such that
		\begin{displaymath}
		h_{p}(\omega )\ge f(\omega )-\varepsilon \quad \forall \omega
		\in U_{p}.
		\end{displaymath}
We choose such a function $h_{p}$ for every point $p\in N_{\R}$. 
We may assume that 
\begin{displaymath}
h_{p+\lambda}(\omega+\lambda)=h_p(\omega) + z_\lambda(\omega)
\end{displaymath}
for all $\lambda \in \Lambda$. 
Note that $h_{p+\lambda }$ is still $\Gamma$-affine with rational slope because the cocycle $z_\lambda$ is	$\Gamma$-affine with integral slope. 
This follows from \eqref{cocycle and b}, \eqref{eq:9} and Remark \ref{rem-polarized-abelian-variety} using the fact that $b$ is the bilinear form of a polarized tropical abelian variety.

Let $F_\Lambda$ be the closure of a fundamental domain of the lattice $\Lambda$. 
Since $F_\Lambda$ is compact, there is a finite subset $I \subset N_\R$ such that the open subsets $(U_p)_{p \in I}$ cover $F_\Lambda$. 
Then we define 
\begin{equation} \label{new definition of h as a sup}
h\colon N_\R \longrightarrow \R, \quad \omega \longmapsto h(\omega) \coloneqq \sup_{p \in I, \lambda \in \Lambda} h_{p+\lambda}(\omega).
\end{equation}
Obviously, we have $f-\varepsilon \leq h \leq f$. 
Since we have $h_{p+\lambda}(\omega+\lambda)=h_\lambda(\omega) + z_\lambda(\omega)$, it is clear that $h$ satisfies the automorphy condition \eqref{new cocycle condition for f}. 
It remains to show that $h$ is a convex locally finite piecewise $\Gamma$-affine function. 
To see this, it is enough to show that the supremum in \eqref{new definition of h as a sup} is locally the maximum of finitely many of the functions $h_p$. 
Recall that $m_p$ is the linear part of the affine function $h_p$.  
For $\omega \in N_\R$ and $\lambda \in \Lambda$, we have
\begin{displaymath}
h_{p+\lambda}(\omega)\leq h_p(\omega) \Leftrightarrow 
h_p(\omega-\lambda) + z_\lambda(\omega-\lambda) \leq h_p(\omega)
\Leftrightarrow z_\lambda(\omega-\lambda) \leq m_p(\lambda).
\end{displaymath}
Using \eqref{cocycle and b}, we have $z_\lambda(\omega-\lambda)=z_\lambda(0)+b(\lambda,\omega-\lambda)$ and hence the above inequalities are equivalent to
$z_\lambda(0)-b(\lambda,\lambda) \leq {m}_p(\lambda)-b(\lambda,\omega)$.
Since the left hand side decreases quadratically like $\sim -b(\lambda,\lambda)$
and the right hand side grows at most linearly for $\|\lambda\|
		\to \infty$, we conclude that for any bounded set $\Omega$, there
		is $R \geq 0$ such that the original inequality
		$h_{p+\lambda}(\omega)\leq h_p(\omega)$ is satisfied for all
		$\omega \in \Omega$ and all $\lambda \in \Lambda$ with
		$\|\lambda\| \geq R$. This proves that locally the supremum in
		\eqref{new definition of h as a sup} is the maximum of only finitely
		many $h_p$.
\end{proof}

\begin{rem} \label{algebraicity of the model}
	If $f$ is a function as in Remark \ref{periodicity and M-linearization} \ref{item-2}   satisfying \eqref{new cocycle condition for f} and
		if $L$ is ample, then $\Lcal$ in \ref{item-2}
		is ample and  $\Acal_\Pi$ is the formal completion of an algebraic model. 
	
	Here is a sketch of proof. 
	To show ampleness, by the formal GAGA-principle in \cite[Th\'eor\`eme 5.4.5]{ega3} and \cite[Theorem I.10.1.2]{fujiwara-kato-1}, it is enough to show that the restriction of $\Lcal$  to any irreducible component $Y$ of the special fiber of $\Acal_\Pi$ is ample. 
	There is a vertex $\omega$ of $\Pi$ such that $p$ maps $Y_\omega$ isomorphically onto $Y$. By  \cite[Proposition 4.12]{gubler-compositio}, strict convexity of $f$ yields that $ p^*(\Lcal)|_{Y_\omega}$ is relatively ample with respect to the toric fiber  bundle $Y_\omega \to \Bcal_s$. 
	Recall that $p^*(\Lcal)= q^*\Hcal \otimes \Ocal_{\hat\Ecal_\Pi}(f)$ and  the line bundle $\Hcal$ on $\hat\Bcal$ is ample as $L$ is ample. 
	Using the global sections $s_{f-\ell-c}|_{Y_\omega}$ as in the proofs of Lemma \ref{short globally generated} and Proposition \ref{short degree of irreducible component}, the Nakai--Moishezon criterion shows that $ p^*(\Lcal)|_{Y_\omega}$ is ample and hence the same is true for $\Lcal|_Y$.
	
\end{rem}

\subsection{Monge--Amp\`ere equations for abelian varieties} \label{short subsection: MA equations on abelian varieties}
 
{Our} goal is to solve the invariant non-archimedean Monge--Amp\`ere equation for arbitrary abelian varieties. 
In \S \ref{short subsection: MA equations on abelian varieties}, we deal with an algebraically closed non-archimedean ground field $K$ with non-trivial valuation $v$. 
In \S \ref{subsection: descent}, we will handle arbitrary non-trivially valued non-archimedean ground fields. 

We fix a $g$-dimensional polarized abelian variety 
$(A ,\phi)$ over $K$. 
Let $L $ be an ample line bundle on $A $ which 
induces the polarization $\phi$. 
Again, we denote by $n$ the rank of the torus $\T$ in the Raynaud extension 
\[
0 \longrightarrow \torus  \longrightarrow E  
\stackrel{q}{\longrightarrow} B  \longrightarrow 0.
\]
of $A$. 
Let $\theta \in A^{1,1}(N_\R)$ and $\omega \in A^{1,1}(N_\R/\Lambda)$
be the canonical $(1,1)$-Lagerberg forms  from \eqref{canonical form for ptav} of $(\Lambda,M,[\cdot,\cdot],\lambda)\coloneqq \trop_v(A,\phi)$. 
We also rely on the canonical metric $\metr_{L}$ of $L$.

\begin{prop} \label{short corresponding psh metris and theta-psh functions for polarized abelian varieties}
	For  $\varphi \in \PSH(N_\R/\Lambda,\omega)$, let $\varphi_v \coloneqq \varphi \circ \tropbar_v$. Then 
	$\metr_{L (\varphi)} \coloneqq e^{-{\varphi_v}}\metr_{L }$ is a continuous semipositive metric as introduced in  Remark \ref{Zhang semipositive for non-archimedean}.
\end{prop}

\begin{proof}  
The canonical metric $\metr_L$ of $L$ is determined by choosing a rigidification of $L$. 
We first recall some facts from Remark \ref{periodicity and M-linearization}. There is a rigidified ample line bundle $H$ on $B^\an$ such that $p^*(L^\an) \simeq q^*(H)$ as 
$\Lambda$-linearized cubical sheaves.  
The ample line bundle $L$ induces a cocycle $(z_\lambda)_{\lambda \in \Lambda}$ and a scalar product $b$ on $N_\R$. We have seen that $Q(\lambda) \coloneqq z_\lambda(0)$ is a quadratic function in $\lambda$ and that $b$ is the associated bilinear form. Since $b$ is also the bilinear form of  $\trop_v(A,\phi)=(\Lambda,M,[\cdot,\cdot],\lambda)$, we have $\theta = d'd''Q$ by \ref{canonical form for tropical polarized abelian variety}. Using \cite[Example 8.15]{gubler-kuenne2017}, we have
 that 
\begin{equation} \label{theta and cocycle}
p^*\metr_{L }= e^{-Q} q^*\metr_{H}.
\end{equation}
Let $\tilde\varphi\colon N_\R \to \R$ be the lift of $\varphi$ to $N_\R$. Then $\tilde\varphi$ is a    $\theta$-psh function $\varphi$ on $N_\R$ which means that 
$f \coloneqq \tilde\varphi+Q$ is convex. For $x \in N_\R$ and     $\lambda \in \Lambda$, we deduce from \eqref{cocycle and b} that
$$f(x+\lambda)= \tilde\varphi(x+\lambda)+Q(x+\lambda)= \tilde\varphi(x)+Q(x)+Q(\lambda)+b(x,\lambda)=f(x)+z_\lambda(0).$$
This means that $f$ satisfies the automorphy condition \eqref{new cocycle condition for f}. By Proposition \ref{new periodic approximation}, $f$ is a uniform limit of locally finite piecewise $\Gamma$-affine functions $f_k$ on $N_\R$ which have rational slopes and which satisfy also \eqref{new cocycle condition for f}. We have seen in the proof of Proposition \ref{new periodic approximation} that $z_\lambda(0)$ is an affine function with integral slopes and hence there is a non-zero $m_k \in \N$ such that $m_k f_k$ is a piecewise $\Gamma$-affine function with integral slopes. Going the above steps backwards, we see that $f_k-Q$ is the lift of a unique $\omega$-psh function $\varphi_k$. By construction, $\metr_{L(\varphi)}$ is the uniform limit of the metrics $\metr_{L(\varphi_k)}$. As in \S \ref{short subsection: Mumford models}, let $\Hcal$ be the model of $H$ on the formal completion $\hat\Bcal$ of the abelian scheme $\Bcal$ over $\kcirc$ with generic fiber $B$. Then we have 
\begin{equation*} \label{model metric identities}
p^*\metr_{L(\varphi_k)}^{\otimes m_k} = e^{-Q-\varphi_{k,v}\circ p}\metr_{q^*\Hcal}^{\otimes m_k}= e^{-f_k} \metr_{q^*\Hcal^{\otimes m_k}} = \metr_{q^*\Hcal^{\otimes m_k}(m_kf_k)}.
\end{equation*}
It follows from Remark \ref{periodicity and M-linearization} that $\metr_{L(\varphi_k)}^{\otimes m_k}$ is the metric induced by a line bundle $\Mcal_k$ on a formal model $\Acal_k=\Ecal_k/M$ of $A$ for a formal Mumford model $\Ecal_k$ of $E$ induced by a integral $\Gamma$-affine $\Lambda$-periodic decomposition $\Pi_k$ of $N_\R$. We claim that these metrics are semipositive. To see this, we have to show that the restriction of $\Mcal_k$ to any irreducible component $Y$ of $\Acal_k$ is nef. Using $\Acal_k= \Ecal_k/M$, we see that $Y\simeq Y_\omega$ for an irreducible component $Y_\omega$ of $\Ecal_k$ associated to a vertex $\omega$ of $\Pi_k$. By Lemma \ref{short globally generated}, we have that $\Ocal_{{\Ecal}_k}(m_kf_k)|_{Y_\omega}$ is nef. Using the $\Hcal$ is ample, we deduce that  $p^*\Mcal_k=q^*\Hcal(m_kf_k)$ restricts to a nef line bundle on $Y_\omega$. This proves our intermediate claim and hence $\metr_{L(\varphi_k)}$ is a semipositive model metric. By definition, we get that $\metr_{L (\varphi)} \coloneqq e^{-{\varphi_v}}\metr_{L }$ is semipositive.
\end{proof}

\begin{rem} \label{short related toric measures}
Let $\mu$ be a positive Radon measure on $N_\R/\Lambda$. 
Let $\overline\iota_v\colon N_\R/\Lambda \to A ^{\an}$ be the 
canonical section of $\overline\trop_v$ used in Remark \ref{short canonical skeleton} to identify $N_\R/\Lambda$ 
with the canonical skeleton of $A$. 
Then we obtain a Radon measure on $A ^{\an}$ as
the image measure   
\begin{equation} \label{short normalization non-archimedean measure}
\mu_v\coloneqq \binom{g}{g-n}\deg_{H }(B )\cdot \overline\iota_v(\mu).
\end{equation}
where $H $ is again the line bundle on $B^\an $ with 
$q^*(H ) = p^*(L^\an)$. 
\end{rem}

\begin{thm} \label{non-arch Calabi-Yau for non-archimedean abelian variety}
We consider a positive Radon measure $\mu_v$ on $A^\an$ 
which is supported in the canonical skeleton of $A$ and satisfies 
$\mu_v(A )=\deg_L(A)$. 
Then there is a continuous semipositive metric $\metr$ of $L $ 
which satisfies the non-archimedean Monge--Amp\`ere equation 
\begin{equation} \label{main non-arch MA-equation}
c_1(L ,\metr )^{\wedge g}=\mu_v.
\end{equation}
The continuous semipositive metric $\metr$ is unique up to scaling.
\end{thm}

For totally degenerate abelian varieties, this was shown by Liu \cite{liu2011} in case of  a measure $\mu_v$ with smooth density proving also that then the solution is a smooth metric.

\begin{proof} 
	It was shown by Yuan and Zhang \cite[Corollary 1.2]{yuan-zhang} that uniqueness of a continuous semipositive solution up to scaling holds on any projective variety over $K$. 
	
	For existence, we note that 
	the measure $\mu_v$ is obtained from a unique Radon
	measure $\mu$ on $N_\R/\Lambda$ as in Remark
	\ref{short related toric measures}. 
	By Proposition \ref{compare degree of polarizations wrt tropicalizations}, 
	we have 
	\begin{equation}\label{norm-condition-eq1}
	\mu(N_\R/\Lambda)=n!\sqrt{\deg(\trop_v(A,\phi))}.
	\end{equation}
	By Proposition \ref{tropical-ma}, there is $\varphi \in \PSH(N_\R/\Lambda)$ solving the tropical  Monge--Amp\`ere equation
	\begin{equation} \label{solution of tropical MA-eq}
	(\omega +d'd''\varphi)^{\wedge n} = \mu.
	\end{equation}
    We claim that the semipositive metric $\metr_{L(\varphi)}$ of $L$ from Proposition \ref{short corresponding psh metris and theta-psh functions for polarized abelian varieties} is a solution of \eqref{main non-arch MA-equation}. 
    We use now the notation and the results from the proof of Proposition  \ref{short corresponding psh metris and theta-psh functions for polarized abelian varieties}. Recall that  $f= \tilde\varphi+Q$ is the uniform limit of the sequence $f_k = \tilde\varphi_k+Q$. Since $\metr_{L(\varphi_k)}$ is a model metric for $L$, its Monge--Amp\`ere measure is the discrete measure on $A^\an$ supported in the Shilov points associated to the irreducible components $Y$ of the Mumford model $A_{\Pi_k}$ and with multiplicity give by the degree of $Y$ with respect to the model of the metric \cite[Theorem 10.5]{gubler-kuenne2017}. By \cite[Proposition 4.8]{gubler-compositio},   $Y$ corresponds to a vertex $u$ of $\Pi_k$, determined up to translation by $\Lambda$, and  the Shilov point of $Y$ is just $p(u)$ in the canonical skeleton of $N_\R$. 
    Proposition \ref{short degree of irreducible component} shows that the non-archimedean Monge--Amp\`ere measure $c_1(L,\metr_{L(\varphi_k)})^{\wedge g}$ is equal to
    \begin{equation} \label{crucial formula} \frac{g!}{(g-n)!} \deg_H(B)\sum_u  \vol(\{u\}^{f_k})
    =\binom{g}{g-n} \deg_H(B)  \bar\iota_v (\omega +d'd''\varphi_k)^{\wedge n}
    \end{equation}  
    where we sum over a system of representatives of vertices $u$ of $\Pi_k$ modulo $\Lambda$. For \eqref{crucial formula}, we use that $n! \vol(\{u\}^{f_k})$  is the multiplicity of the real Monge--Amp\`ere measure of the piecewise affine $f_k$ in $u$ \cite[Proposition 2.7.4]{bps-asterisque} and then Remark \ref{rem:4}. Using continuity of the tropical and the non-archimedean Monge--Amp\`ere measure along uniformly convergent sequences, we deduce from \eqref{short normalization non-archimedean measure}, \eqref{solution of tropical MA-eq} and \eqref{crucial formula} that $\metr_{L(\varphi_k)}$ is a solution of \eqref{main non-arch MA-equation}.
\end{proof}

\subsection{Descent} \label{subsection: descent}

In this subsection, $K$ is an arbitrary non-archimedean field with non-trivial valuation $v$. 
We  solve the invariant non-archimedean Monge--Amp\`ere equation for any abelian variety $A$ over $K$. 
We  use the solution from Theorem \ref{non-arch Calabi-Yau for non-archimedean abelian variety} for the base change of $A$ to an algebraic closure $\overline K$ and  apply a descent argument.  
We  use continuous semipositive metrics on an ample line bundle as introduced in Remark \ref{Zhang semipositive for non-archimedean}. 

\begin{art} \label{notation for base change}
Let $L$ be an ample line bundle on a  geometrically integral projective variety $Y$ over $K$ of dimension $n$. We consider an extension $F/K$ of non-archimedean fields. The base change of $Y$ (resp.~$L$) is denoted by $Y_F$ (resp.~$L_F$). Note that the base change morphism induces a proper map $\pi\colon Y_F^\an \to Y^\an$. Let $\metr$ be a continuous metric on $\Lan$ and let $\metr_F \coloneqq \pi^*\metr$ be the base change metric on $L_F^\an$. 	
\end{art}

We  recall results of Boucksom and Eriksson about the base change of semipositive metrics.
\begin{prop} \label{base extension and MA}
In the setting of \ref{notation for base change}, the following properties hold.
 \begin{enumerate}
 	\item \label{psh and base change} 
 	The continuous metric $\metr$ is semipositive if and only if  $\metr_F$ is semipositive.
 	\item \label{item: MA and base change}
 	If $\metr$ is a continuous semipositive metric, then $c_1(L,\metr)^{\wedge n}=\pi_*(c_1(L_F,\metr_F)^{\wedge n})$.
 \end{enumerate}
\end{prop}

\begin{proof}
Property \ref{psh and base change} is  \cite[Theorem 7.32]{boucksom-eriksson2018} and \ref{item: MA and base change} is shown in \cite[\S 8.1]{boucksom-eriksson2018}.
\end{proof}

\begin{art} \label{norm and metric}
We keep the setting of \ref{notation for base change}. We assume that $F/K$ is a finite Galois extension of degree $d$ with Galois group $G$. Then base change induces a finite flat morphism $Y_F \to Y$ and  hence we have the norm  of a line bundle $L'$ on $Y_F$ as a line bundle $N(L')$ on $Y$ (see \cite[\S A.8]{boucksom-eriksson2018}). For a metric $\metr'$ of $L'$, we have the norm of the metric as a metric $N(\metr')$ of $N(L')$. If $\metr'$ is a continuous semipositive metric, then $N(\metr')$ is a continuous semipositive metric \cite[Proposition 8.22]{boucksom-eriksson2018}. 

We apply this now to $L' \coloneqq L_F$ for the ample line bundle $L$ on $Y$. Then $N(L_F)=L^{\otimes d}$ and hence $\pi^*(N(L_F))= L_F^{\otimes d}$. The finite Galois group $G$ acts continuously on $\Yan$ and on $L_F^\an$. For a continuous semipositive metric $\metr'$ of $L_F$, the metric $\metr_\sigma' \coloneqq \sigma^*\metr'$ is also a continuous semipositive metric of $L_F$ and we deduce from the definitions that
\begin{equation} \label{norm and tensor metric}
\pi^*\bigl(N(L_F),N(\metr')\bigr)
= \bigotimes_{\sigma \in G} (L_F,\metr_\sigma).
\end{equation}
Moreover, there is a metric $\metr$ on $L$ with $\pi^*\metr=\metr'$ if and only if $\metr'$ is $G$-invariant.
\end{art}

Uniqueness of the complex Monge--Amp\`ere equation was originally proven by Calabi. 
In non-archimedean geometry, we have the following result of Yuan and Zhang.

\begin{thm} \label{yuan-zhang}
Let $L$ be an ample line bundle on a geometrically integral variety $Y$ over the non-trivially valued non-archimedean field $K$. 
If $\metr$ and $\metr'$ are continuous semipositve metrics of $L$ with $c_1(L,\metr)^{\wedge n}= c_1(L,\metr')^{\wedge n}$, then there is  $r \in \R_{>0}$ with $\metr'=r \metr$.
\end{thm}

\begin{proof}
In \cite[Corollary 1.2]{yuan-zhang}, this is proven for a non-trivially valued algebraically closed non-archimedean field. We will deduce from it the claim for any non-trivially valued non-archimedean field $K$. 
Let $F$ be the completion of an algebraic closure $\overline K$ of $K$ and let $G$ be the Galois group of $\overline K/K$.  By \cite [Corollary 1.3.6]{berkovich-book}, we have $X_F^{\rm an}/G\simeq\Xan$ as a topological space. 
Using that the profinite group $G$ is compact, we note that $c_1(L_F,\metr_F)^{\wedge n}$ is the unique $G$-invariant positive Radon measure on $Y_F^\an$ 
with image measure $c_1(L,\metr)^{\wedge n}$ on $\Yan$ (see \S \ref{section: notation}). 
This shows that $c_1(L_F,\metr_F)^{\wedge n}=c_1(L_F',\metr_F')^{\wedge n}$. 
By \cite[Corollary 1.2]{yuan-zhang}, we get  $\metr_F'=r \metr_F$ for some $r \in \R_{>}$ and hence  $\metr'=r \metr$. 
\end{proof}

\begin{art} \label{split abelian varieties}
We say that an abelian variety $A$ is \emph{split over $K$} if $A$ has a Raynaud extension as  in \S \ref{short subsection: Raynaud uniformization and canonical tropicalization} with split torus $\T$. 
Then we get a canonical tropicalization map $\tropbar_v\colon A^\an \to N_\R/\Lambda$ as in Remark \ref{short canonical tropicalization} where $N$ is the cocharacter lattice of $\T$.  
	For any abelian variety $A$ over $K$, there is a finite
	separable extension $F$ of $K$ such that $A_F$ is split over $F$.
	
	Now let $L$ be a rigidified ample line bundle on the split abelian variety $A$ and let 
\[
0 \longrightarrow \torus \longrightarrow E \stackrel{q}{\longrightarrow} B \longrightarrow 0
\]
be the {Raynaud extension of $A$} such that $A^{\an}= E^{\an}/\Lambda$ for a lattice  $\Lambda$ of $E(K)$. 
By \cite[Proposition 6.5]{bosch-luetkebohmert1991}, there is a rigidified ample line bundle $H$ on $B^\an$ such that $p^*(\Lan) \simeq q^*(H)$.
\end{art}

\begin{art}
We prove now Theorem \ref{solution of non-archimedean MAeqn for abelian varieties in intro} from the introduction which generalizes Theorem \ref{non-arch Calabi-Yau for non-archimedean abelian variety} to any abelian variety over any non-trivially valued non-archimedean field $K$. 
Let $S(A)$ be the canonical skeleton of $A$ from \cite[\S 6.5]{berkovich-book}. 
We consider an ample line bundle $L$ on an abelian variety $A$ over $K$ and  a positive Radon measure $\mu_v$ on $A^\an$ which is supported in $S(A)$ and satisfies $\mu_v(A )=\deg_L(A)$. 
We have to show that 
\begin{equation} \label{main non-arch MA-equation++}
	c_1(L ,\metr )^{\wedge g}=\mu_v.
\end{equation}
has a solution given by a continuous semipositive metric $\metr$ of $L$ unique up to scaling.

\begin{proof}
Uniqueness up to scaling follows from Theorem \ref{yuan-zhang} and hence it remains to show existence. 
By \ref{split abelian varieties}, there is a finite Galois extension $F/K$ such that $A_F$ is split over $F$. Let $\varphi$ be solution of the associated tropical Monge--Amp\`ere equation \eqref{solution of tropical MA-eq} as in the proof of Theorem \ref{non-arch Calabi-Yau for non-archimedean abelian variety}. Using $\varphi_v \coloneqq \varphi \circ \tropbar_v$, we define the metric $\metr_{L(\varphi)}$ of $L_F$ as in 
Proposition \ref{short corresponding psh metris and theta-psh functions for polarized abelian varieties}. We have seen in the proof of Theorem \ref{non-arch Calabi-Yau for non-archimedean abelian variety} that $\metr' \coloneqq \metr_{L(\varphi)}$ satisfies the non-archimedean Monge--Amp\`ere equation \ref{main non-arch MA-equation++} over the completion of the algebraic closure of $F$ and hence also over $F$ by using 
Proposition \ref{base extension and MA}. 

Let $G$ be the Galois group of $F/K$, let $d \coloneqq [F:K]$ and let $N$ be the norm introduced  in \ref{norm and metric}. 
By \cite[\S 6.5]{berkovich-book}, we have $S(A)=S(A_F)/G$ and hence there is a unique $G$-invariant positive Radon measure $\mu_v'$ on $A_F^\an$ which is supported in $S(A_F)$ and with image measure $\mu_v$ (see \S \ref{section: notation}).
Since $\mu_v$ is invariant under the $G$-operation, it follows from the uniqueness theorem  \ref{yuan-zhang} that $\metr'= \metr_\sigma'$ for all $\sigma \in G$. 
Using \eqref{norm and tensor metric}, we conclude that the metric $\pi^*N(\metr')$ of $L_F^{\otimes d}$ solves the non-archimedean Monge--Amp\`ere equation on $A_F^\an$ for the measure $d^g \cdot\mu_v'$. 
Again by Proposition \ref{base extension and MA}, it follows that the metric $N(\metr')$ of $N(L_F)=L^{\otimes d}$ satisfies the non-archimedean Monge--Amp\`ere equation on $A^\an$ for the measure $d^g \mu_v$. We conclude that $\metr \coloneqq   N(\metr')^{1/d}$ is the desired solution of \eqref{main non-arch MA-equation++}.
\end{proof}
\end{art}

\appendix

\section{Proof of Theorem \ref{thm:6}}
\label{sec:proof-theor-refthm:6}

We give a proof of uniqueness in Theorem \ref{thm:6} as the argument
is omitted in Berman-Berndtsson \cite{berman-berndtsson2013} and
Bakelman's uniqueness proof for \cite[Theorem
17.1]{bakelman94:_convex_analy} relies on \cite[Lemma
10.2]{bakelman94:_convex_analy} which is false
as the following example shows. We will adjust his arguments.

  \begin{ex}
   Let $u, v$ the convex functions on $G \coloneqq \R$ defined as
    follows
    \begin{displaymath}
      u(x)=
      \begin{cases}
        -4x-3,&\text{ if }x\le -1,\\
        x^{2},&\text{ if }-1\le x\le 1\\
        4x-3,&\text{ if }x\ge 1,
      \end{cases}
      \qquad
      v(x)=
      \begin{cases}
        -2x-1,&\text{ if }x\le 0, \\
        2x-1,&\text{ if }x\ge 0.
      \end{cases}
    \end{displaymath}
    On the open interval $Q=(-1,1)$
    we have $v<u$. On $\partial Q=\{-1,1\}$, we have $v=u$. 
    For the multivalued subdifferentials, we have 
    \begin{displaymath}
      \partial u (1) = [2,4],\qquad \partial v(1) = \{2\}.
    \end{displaymath}
Hence  $\partial u (1)\setminus \partial v(1)\not = \emptyset$. 
Under these hypotheses, it was claimed in \cite[Lemma 10.2]{bakelman94:_convex_analy} that $\MA(v)\big(Q\big) > \MA(u)\big(Q\big)$, but we have
    $\MA(v)\big(Q\big)= 4 = \MA(u)\big(Q\big)$.
  \end{ex}

We keep the same setting as in \ref{real setting for MA}, i.e.~$N$ is
a free abelian group of finite rank $n$ with an identification $N
\simeq \Z^n$ and $M \coloneqq \Hom(N,\Z)$. For a set $W$ of $M_\R$, we
denote by $W^\circ$ the \emph{interior of $W$}. In the following, we
fix a convex body $\Delta$ in $M_\R$ and we consider a convex function
$f: N_\R \to \R$ with $\Delta ^{\circ}\subset \Delta (f) \subset
\Delta$ for the stability set $\Delta(f)$. For the multivalued
subdifferential $\partial f$, it follows from \cite[Theorems 23.4 and
23.5]{rockafellar1970} that
    \begin{equation} \label{sandwich for subdifferential}
    \Delta ^{\circ}\subset \partial f(N_\R) \subset \Delta (f) \subset \Delta.   
    \end{equation}
    We first prove a \emph{comparison principle}.
    
    \begin{prop} \label{comparison principle}
    Let $f$ and $g$ be convex
    functions on $N_\R \simeq \R^n$ and let $\Delta$ be a convex body
    in $M_\R$ with
    $\Delta ^{\circ}\subset \Delta (f) \subset \Delta $. 
    For
    $\Omega \coloneqq \{x \in N_\R\mid f(x) < g(x)\}$, we have
    \begin{displaymath}
    \partial g(\Omega )\cap \Delta ^{\circ}\subset \partial
    f(\Omega ).
    \end{displaymath}
    \end{prop}

    \begin{proof} Let $H\in \partial g(\Omega )\cap \Delta
    ^{\circ}$. Since $H\in \partial g(\Omega )$, there is an $y_{1}\in
    \Omega $ such that $H\in \partial g(y_{1})$. Since $H\in \Delta
    ^{\circ}\subset \partial(f)(N_\R)$  using \eqref{sandwich
      for subdifferential}, there is an $y_{2}\in N_{\R}$ such
    that $H\in \partial f(y_{2})$. The conditions $H\in \partial
    g(y_{1})$ and $y_{1}\in \Omega $ yield
    \begin{equation}
      \label{eq:12}
      g(y_{2})\ge g(y_{1}) + H(y_{2}-y_{1}) > f(y_{1}) +
      H(y_{2}-y_{1}). 
    \end{equation}
    The condition $H\in \partial f(y_{2})$ yields
    \begin{equation}\label{eq:13}
      f(y_{1})\ge f(y_{2})+ H(y_{1}-y_{2}).
    \end{equation}
    Equations \eqref{eq:12} and \eqref{eq:13} imply
    $f(y_{2})<g(y_{2})$. Therefore $y_{2}\in \Omega $ proving claim.
    \end{proof}

We will need also the following \emph{openness criterion.}
\begin{prop} \label{openness criterion}
Let $f$ be a convex function on an
open convex set $U$ of $N_\R$. Let $x_{0}\in U$ and $H_{0}\in \partial f(x_{0})$. If
there is a compact neighbourhood $B$ of $x_{0}$ in $U$ such that
\begin{equation}\label{eq:16}
f(y) > f(x_{0}) + H_{0}(y-x_{0})
\end{equation}
for all $y \in \partial B$.  Then $\partial f(B)$ is a neighborhood of $H_0$.
\end{prop}

\begin{proof}
Since $\partial B$ is compact, the condition \eqref{eq:16} implies
that there is a neighbourhood $V$ of $H_{0}$ such that $f(y) > f(x_{0}) + H(y-x_{0})$ for all $H\in V$
and all $y\in
\partial B$. We consider $H \in V$. We claim that 
for every point $y\in U\setminus B$, it is still true that $f(y) >
f(x_{0}) + H(y-x_{0})$. To see that, consider the segment $\overline
{x_{0}y}$. There is a point $y_{0}$ in the
interior of $\overline
{x_{0}y}$ such that $y_{0} \in \partial B$.
Writing $y_{0}=\alpha  x_{0}+(1-\alpha )y$ with $0<\alpha <1$ and
using that $x\mapsto g(x) \coloneqq f(x)-f(x_{0})-H(x-x_{0})$ is convex, we have 
\begin{displaymath}
0 < f(y_{0})-f(x_{0})-H(y_{0}-x_{0})\le \alpha g(x_0) + (1-\alpha)g(y) = (1-\alpha)g(y)
\end{displaymath}
proving $f(y) >
f(x_{0}) + H(y-x_{0})$. 
We get
\begin{displaymath}
\emptyset \not = \{x\in U\mid f(x)\le
f(x_{0})+H(x-x_{0})\}\subset B 
\end{displaymath}
which implies 
$H\in 
\partial f(B)$ as claimed.
\end{proof}

\begin{proof}[Proof of uniqueness in Theorem \ref{thm:6}.]
  Let $\Delta \subset M_\R$ be a convex body. For convex functions
  $f,g$ on $N_\R$ such that $\MA(f)=\MA(g)$ and 
  such that
  \begin{gather*}
    \Delta ^{\circ} \subset \Delta (f) \subset \Delta \quad {\rm and}  \quad  \Delta ^{\circ} \subset \Delta (g) \subset \Delta,
  \end{gather*}
  we have to show that $f-g$ is constant.

  Let $f^*(L) \coloneqq \sup\{x \in N_\R \mid L(x)-f(x)\}$  be the
  \emph{conjugate function of $f$}, also called the \emph{Legendre
    dual of $f$}. It is a closed convex function on $M_\R$ which
  takes the value $\infty$ precisely outside the stability set
  $\Delta(f)$, and we have $f^{**}=f$ (see \cite[\S
  12]{rockafellar1970}).
    
  Assume that there
    exists a point $H\in \Delta ^{\circ}$ such that $\partial
    f^{*}(H)\cap \partial g^{*}(H)=\emptyset$. We will see that
    this contradicts the hypothesis that $\MA(f)=\MA(g)$. For
    simplicity, 
    we can assume that $H=0$. This
    amounts to translate $\Delta $ and to add to $f$ and $g$ the same
    linear function which does not change the Monge--Amp\`ere
    measures. After adding a constant to $f$ and to $g$, we can assume
    that $\min(f) = \min(g)=0$. As before, let $\Omega \coloneqq \{x \in N_\R \mid f(x)<g(x)\}$. We are going to see that
    \begin{equation} \label{contradiction interior}
    0\in \partial f(\Omega )^{\circ},\quad 0\not \in \overline
    {\partial g(\Omega )}. 
    \end{equation}
    This and Proposition \ref{comparison principle} imply
    that $\MA(g)(\Omega )=\vol(\partial g(\Omega ))< \vol(\partial
    f(\Omega )) =\MA(f)(\Omega )$, which contradicts the hypothesis
    $\MA(f)=\MA(g)$.
    
    We show first that $0\in \partial f(\Omega )^{\circ}$.  Since
    $H=0$ and $\min(f)=0$, we have
    $A\coloneqq \{x \in N_\R \mid f(x)=0\}=\partial f^{*}(0)$ which, by
    hypothesis, is disjoint to $\{x \in N_\R \mid g(x)=0\}=\partial
    g^{*}(0)$. Thus
    $A\subset \Omega $.  By
    \cite[Theorem 24.7]{rockafellar1970}, the set $A$ is compact. Since
    $\Omega $
    is open and $A$ is compact, there is a compact subset
    $B$ of $\Omega$ 
    with $A\subset B^{\circ}$. Then $f|_{\partial B}>0$. 
    By Proposition \ref{openness criterion},   
    $\partial f(B)$ is a neighborhood of $0$.  Since 
    $\partial f(B) \subset   \partial f (\Omega )$, we get $0\in
    \partial f (\Omega )^{\circ}$.
    
    Assume now that $0\in \overline{\partial g (\Omega )}$. This means that 
    there are sequences $x_{k}\in \Omega $ and $H_{k}\in \partial
    g(x_{k})$ such that  $H_{k}$ 
    converges to zero and hence $S \coloneqq \{H_k \mid k \in \N\}
    \cup \{0\}$ is a compact subset of $M_\R$.  
    Since $\Delta^\circ \subset \partial g(N_\R)$ by \eqref{sandwich
      for subdifferential}, we may assume that $S$ is contained in the
    interior $\Delta^\circ$ of the domain $\Delta(g)$ of $g^*$.
    Applying \cite[Theorem 24.7]{rockafellar1970} to $g^*$, we deduce
    that $\partial g^*(S)$ is compact. 
    By  \cite[Corollary 23.5.1]{rockafellar1970}, 
    we have $x_k \in \partial g^*(H_k)$ and hence 
    we can find a subsequence $x_{i_{k}}$ that converges to a
    point $y \in N_\R$. By \cite[Theorem 24.4]{rockafellar1970}, we have $0\in
    \partial g(y)$ and so $g(y)=0< f(y)$. But $x_{k}\in
    \Omega $ implies that $f(x_{k})< g(x_{k})$ and hence $f(y)\le
    g(y)$ by continuity leading to a contradiction. This proves
    \eqref{contradiction interior}.
    
    We conclude from the whole argument above that for every
    $H\in \Delta ^{\circ}$, we have
    \begin{equation}
    \label{eq:14}
    \partial f^{\ast}(H)\cap \partial g^{\ast}(H)\not = \emptyset.
    \end{equation}
    The convex functions $ f^{\ast}$ and
    $g^{\ast}$ are almost everywhere differentiable on $\Delta$ and the
    condition \eqref{eq:14} implies that $df^{\ast}=dg^{\ast}$
    almost everywhere. Since $f^*$ and $g^*$ are convex  continuous finite functions on $\Delta^\circ$, this implies  $f^{*} -
    g^{*}$  constant. By biduality,  $f-
    g$ is also constant.
    \end{proof}

\bibliographystyle{alpha}
%\bibliography{bib-toric}

\end{document}